\documentclass[11pt]{article}
\usepackage{amsmath}
\usepackage{amssymb}
\usepackage{mathrsfs}
\usepackage{amsthm}
\usepackage{cite}
\usepackage{amscd}
\usepackage{authblk}
\usepackage{fancyheadings}
\usepackage[all]{xy}
\theoremstyle{definition}
\newtheorem{thm}{Theorem}[section]
\newtheorem{definition}[thm]{Definition}
\newtheorem{prop}[thm]{Proposition}
\newtheorem{lem}[thm]{Lemma}
\newtheorem{rem}[thm]{Remark}
\newtheorem{cor}[thm]{Corollary}
\newtheorem{ex}[thm]{Example}

\newtheorem{ques}[thm]{Question}
\newtheorem*{ack}{Acknowledgement}

\newtheorem{axio}[thm]{Axiom}
\numberwithin{equation}{section}
\usepackage[top=20truemm, bottom=25truemm, left=25truemm, right=25truemm]{geometry}
\newcommand{\Z}{\mathbb{Z}}
\newcommand{\Q}{\mathbb{Q}}

\newcommand{\Spec}{\operatorname{Spec}}

\newcommand{\Hom}{{\rm Hom}}
\newcommand{\Coh}{\mathsf{Coh}}
\newcommand{\Rep}{\mathsf{Rep}}
\newcommand{\Vecf}{\mathsf{Vec}}
\newcommand{\Vect}{\mathsf{Vect}}

\newcommand{\Qcoh}{\mathsf{QCoh}}
\newcommand{\Qcohfp}{\mathsf{QCoh}_{\rm fp}}
\newcommand{\Aut}{{\rm Aut}}
\newcommand{\Dim}{\operatorname{dim}}
\newcommand{\End}{{\rm End}}

\newcommand{\Ker}{\operatorname{Ker}}
\newcommand{\im}{\operatorname{Im}}

\newcommand{\Gal}{{\rm Gal}}
\newcommand{\calO}{\mathscr{O}}
\newcommand{\scrO}{\mathscr{O}}

\newcommand{\calC}{\mathcal{C}}
\newcommand{\calD}{\mathcal{D}}

\newcommand{\calX}{\mathcal{X}}
\newcommand{\calY}{\mathcal{Y}}

\newcommand{\calP}{\mathcal{P}}
\newcommand{\calE}{\mathcal{E}}
\newcommand{\calF}{\mathcal{F}}
\newcommand{\calG}{\mathcal{G}}

\newcommand{\calT}{\mathcal{T}}

\newcommand{\Def}{\overset{{\rm def}}{=}}
\newcommand{\ab}{{\rm ab}}
\newcommand{\et}{\text{{\rm \'et}}}
\newcommand{\lr}{{\rm lin.red}}
\newcommand{\loc}{{\rm loc}}

\newcommand{\fppf}{{\rm fppf}}
\newcommand{\liset}{\text{lis-\'et}}
\newcommand{\Liset}{\text{Lis-\'et}}

\newcommand{\Diag}{\operatorname{Diag}}
\newcommand{\X}{\mathbb{X}}
\newcommand{\G}{\mathbb{G}}
\newcommand{\F}{\mathbb{F}}
\newcommand{\A}{\mathbb{A}}

\newcommand{\C}{\mathbb{C}}

\newcommand{\unit}{\mathbf{1}}

\newcommand{\Pic}{\mathrm{Pic}}

\newcommand{\Aff}{\mathrm{Aff}}

\newcommand{\TORS}{\mathsf{TORS}}

\newcommand{\fD}{\mathfrak{D}}

\newcommand{\fX}{\mathfrak{X}}

\newcommand{\cB}{\mathcal{B}}

\newcommand{\bfD}{\mathbf{D}}
\newcommand{\bfr}{\mathbf{r}}
\newcommand{\calI}{\mathcal{I}}
\newcommand{\bfs}{\mathbf{s}}
\newcommand{\Fdiv}{\mathrm{Fdiv}}
\newcommand{\lto}{\longrightarrow}
\newcommand{\N}{\mathrm{N}}

\newcommand{\EFin}{\mathrm{EFin}}
\newcommand{\uAut}{\underline{\mathrm{Aut}}}

\title{\bf An embedding problem for finite local torsors\\ 
over twisted curves
}
\author{Shusuke Otabe}
\date{\vspace{-10mm}}

\begin{document}

\maketitle

\begin{abstract}
In his previous paper, the author proposed as a problem a purely inseparable analogue of the Abhyankar conjecture for affine curves in positive characteristic and gave a partial answer to it, which includes a complete answer for finite local nilpotent group schemes.  
In the present paper, motivated by the Abhyankar conjectures with restricted ramifications due to Harbater and Pop, we study a refined version of the analogous problem, based on a recent work on tamely ramified torsors due to Biswas--Borne, which is formulated in terms of root stacks. We study an embedding problem to conclude that the refined analogue is true in the solvable case. 
\end{abstract}

\thispagestyle{fancy}

\renewcommand{\headrulewidth}{0pt}
\renewcommand{\footrulewidth}{0.1mm}
\lfoot{{\footnotesize \textit{Date}: February 9, 2020. 5th version.
~ 
}}
\cfoot{}

\section{Introduction}

\subsection{The Abhyankar conjectures with restricted ramifications}

In \cite{gr71}, Grothendieck defined the \textit{\'etale fundamental group} $\pi^{\et}_1(X)$ for a scheme $X$ as a \textit{profinite} group which controls the Galois theory for finite \'etale coverings of $X$. If $X/\C$ is a complex algebraic variety, then one can calculate the \'etale fundamental group $\pi^{\et}_1(X)$ of $X$ as the \textit{profinite completion} of the topological fundamental group $\pi_1^{\rm top}(X(\C))$ of the associated analytic space $X(\C)$,
\begin{equation*}
\pi^{\et}_1(X)\simeq\pi^{\rm top}_1(X(\C))^{\widehat{}}.
\end{equation*}    
In fact, by the Lefschetz principle, this description can be adopted to any algebraically closed field of characteristic 0. For example, we can find that the \'etale fundamental group $\pi^{\et}_1(\A^1_k)$ of the affine line $\A^1_k$ over an algebraically closed field $k$ of characteristic 0 is trivial, i.e.\ $\pi^{\et}_1(\A^1_{k})=0$ because the associated Riemann surface $\A^1(\C)$ is simply connected.

However, if the base field $k$ is of positive characteristic $p>0$, then the situation is more complicated. For example, it is known that the \'etale fundamental group $\pi_1^{\et}(\A_k^1)$ of the affine line $\A^1_{k}$ over an algebraically closed field of characteristic $p>0$ is highly nontrivial. Indeed, the \textit{Artin--Schreier coverings} of $\A_k^1$ contribute to make the fundamental group very big,
\begin{equation*}
\Dim_{\F_p}\Hom(\pi^{\et}_1(\A^1_k),\F_p)=\infty.
\end{equation*}
In particular, $\pi^{\et}_1(\A_k^1)$ is far from topologically finitely generated.

The \textit{Abhyankar conjecture} for affine curves partially describes the \'etale fundamental group $\pi^{\et}_1(U)$ of an affine smooth curve $U$ in positive characteristic $p>0$ from the viewpoint of the \textit{inverse Galois problem}. In \cite{ab57}, Abhyankar asked which finite groups occur as a quotient of $\pi^{\et}_1(U)$ for an affine smooth curve $U$ defined over an algebraically closed field $k$ of characteristic $p>0$ and proposed a conjectural answer to the question. The conjecture was solved affirmatively by Raynaud and Harbater~\cite{ra94}\cite{ha94}. The Abhyankar conjecture says that contrary to the \'etale fundamental group $\pi^{\et}_1(U)$ itself, the inverse Galois problem over $U$ has a concise answer and the finite quotients of $\pi^{\et}_1(U)$ can be completely determined by the topological one $\pi^{\rm top}_1(U(\C))$ of a Riemann surface $U(\C)$ of the same type $(g,n)$, where $g=\Dim_k H^1(X,\scrO_X)$ denotes the genus of the smooth compactification $X$ of $U$ and $n$ denotes the cardinality of the complement $X\setminus U$ of $U$. 

\begin{thm}(The Abhyankar conjecture, cf.\cite{ab57}\cite{ra94}\cite{ha94})\label{thm:AC}
With the above notation, let $G$ be a finite group. Then $G$ occurs as a quotient of $\pi^{\et}_1(U)$ if and only if the quotient $G^{(p')}=G/p(G)$ can be generated by at most $2g+n-1$ elements, where $p(G)$ is the normal subgroup of $G$ generated by all the $p$-Sylow subgroups.
\end{thm}

Note that the `only if' part is proved by Abhyankar himself~\cite{ab57}. In fact, it can be also deduced from Grothendieck's result, which gives a description of the maximal pro-prime-to-$p$ quotient of the \'etale fundamental group, i.e.
\begin{equation}\label{eq:p' etale pi1}
\pi^{\et}_1(U)^{(p')}\simeq \widehat{F}_{2g+n-1}^{(p')},
\end{equation}
where $\widehat{F}_{2g+n-1}$ denotes the free profinite group of rank $2g+n-1$. 
For example, in the case where $U=\A^1_k$ the affine line, the conjecture claims that the set of finite quotients of $\pi^{\et}_1(\A^1_k)$ coincides with the set of \textit{quasi-$p$-groups}, where a finite group $G$ is said to be \textit{quasi-$p$} if $G=p(G)$. The conjecture had been widely open even in the case where $U=\A^1_k$ until Serre gave a partial answer to the problem. 

\begin{thm}(Serre, cf.~\cite{se90})\label{thm:Serre}
Suppose given a short exact sequence of finite groups,
\begin{equation*}
1\lto G'\lto G\lto G''\lto 1
\end{equation*}
which satisfies the following conditions.
\begin{enumerate}
\renewcommand{\labelenumi}{(\roman{enumi})}
\item $G$ is a quasi-$p$-group. 
\item $G''$ appears as a quotient of $\pi^{\et}_1(\A^1_k)$. 
\item $G'$ is solvable. 
\end{enumerate}
Then $G$ also appears as a quotient of $\pi^{\et}_1(\A^1_k)$. In particular, any solvable quasi-$p$-group appears as a quotient of $\pi^{\et}_1(\A_k^1)$.  
\end{thm}

Serre solved an \textit{embedding problem} to prove the theorem. More precisely, let $\overline{\phi}:\pi^{\et}_1(\A^1_k)\twoheadrightarrow G''$ be a chosen surjective homomorphism. He proved that the corresponding embedding problem
\begin{equation*}
\begin{xy}
\xymatrix{
&&&\pi^{\et}_1(\A^1_k)\ar@{->>}[d]^{\overline{\phi}}&\\
1\ar[r]&G'\ar[r]&G\ar[r]&G''\ar[r]&1,
}
\end{xy}
\end{equation*}
always has a solution, i.e.\ there exists a surjective lifting $\phi:\pi^{\et}_1(\A^1_k)\twoheadrightarrow G$ of $\overline{\phi}$ (after modifying $\overline{\phi}$).

Serre's theorem was used in Raynaud's proof. After Serre, in \cite{ra94}, Raynaud proved the conjecture for the affine line $U=\A^1_k$ and, soon after, in \cite{ha94}, Harbater solved the conjecture for general $U$ by applying the method of formal patching and by making use of Raynaud's result. Actually, Harbater proved the conjecture in a stronger form, which we call the \textit{strong Abhyankar conjecture}. 

\begin{thm}(The strong Abhyankar conjecture due to Harbater, cf.~\cite{ha94})\label{thm:SAC}
With the above notation, let $U$ be an affine smooth curve over $k$ of type $(g,n)$ and $G$ a finite group such that $G^{(p')}=G/p(G)$ can be generated by $2g+n-1$ elements. Then there exists a finite \'etale connected Galois covering $V\lto U$ with Galois group $G$ which is tamely ramified along $X\setminus U$ except for one point $x_0\in X\setminus U$.
\end{thm}

Independently of Harbater, in \cite{pop}, Pop gave a further refinement of Theorem \ref{thm:AC} in terms of embedding problems. Let us recall his approach. With the notation in Theorem \ref{thm:SAC}, let $G$ be a  finite group with $G^{(p')}$ generated by $2g+n-1$ elements. Then we get an embedding problem
\begin{equation}\label{eq:EP Pop}
\begin{xy}
\xymatrix{
&&&\pi^{\et}_1(U)\ar@{->>}[d]^{\overline{\phi}}&\\
1\ar[r]&p(G)\ar[r]&G\ar[r]&G^{(p')}\ar[r]&1,
}
\end{xy}
\end{equation}
where $\overline{\phi}$ is a fixed surjection, which always exists due to (\ref{eq:p' etale pi1}). 
Let us fix an \'etale Galois $G^{(p')}$-cover
\begin{equation}
V_0\lto U
\end{equation}
which realizes the surjective homomorphism $\overline{\phi}$. Note that $V_0\lto U$ is tamely ramified along $X\setminus U$. Furthermore, let $Y_0$ be the normalization of $X$ in $V_0$. Under this situation, Pop solved the embedding problem (\ref{eq:EP Pop}) in the following form (see \cite{pop} for the full statements).

\begin{thm}(Pop, cf.~\cite[Theorem B and Corollary]{pop})\label{thm:Pop}
With the above notation, if the quasi-$p$-group $p(G)$ appears as a quotient of $\pi^{\et}_1(\A^1_k)$, then there exists a surjective lifting $\phi:\pi^{\et}_1(U)\twoheadrightarrow G$ of $\overline{\phi}$ in (\ref{eq:EP Pop}) such that a corresponding \'etale Galois $G$-cover $V\lto U$ is ramified only at one point in $Y_0$. In particular, Theorem \ref{thm:SAC} is a consequence of Raynaud's solution of the Abhyankar conjecture for the affine line, i.e.\ Theorem \ref{thm:AC} for the affine line $\A^1_k$.    
\end{thm}

Note that the theorem holds for an arbitrary affine curve $U$ over $k$, a priori which has no relation with the affine line $\A^1_k$.

\subsection{Main theorem}

In \cite{no76}\cite[Chapter II]{no82}, Nori introduced a new invariant $\pi^{\N}(X)$, which he  called the \textit{fundamental group scheme}, as a generalization of Grothendieck's \'etale fundamental group $\pi^{\et}_1(X)$ for a scheme $X$ defined over a field $k$. The fundamental group scheme $\pi^{\N}(X)$ (if it exists) is a profinite $k$-group scheme which classifies $G$-torsors over $X$, where we can take as $G$ an arbitrary (not necessarily \'etale) finite $k$-group scheme. In the case where $k$ is of characteristic 0, then, by a theorem of Cartier, any finite $k$-group scheme $G$ is \'etale. Hence, if $k$ is an algebraically closed field of characteristic 0, then we can compute $\pi^{\N}(X)$ as a pro-constant $k$-group scheme associated with $\pi^{\et}_1(X)$. On the other hand, if $k$ has positive characteristic $p>0$, then the former group 
cannot be recovered from the latter one in general.    
For example, \textit{finite local torsors}, which are sometimes called \textit{purely inseparable coverings}, over $X$ make the group scheme $\pi^{\N}(X)$ larger. 
In \cite{ot17}, the author formulated a \textit{purely inseparable analogue of the Abhyankar conjecture} to estimate the difference between these two fundamental groups $\pi^{\et}_1(X)$ and $\pi^{\N}(X)$ from the viewpoint of the inverse Galois problem.

\begin{ques}(Purely inseparable analogue of the Abhyankar conjecture~\cite[Question 3.3]{ot17})\label{ques:PIAC}
Let $U$ be an affine smooth curve over an algebraically closed field $k$ of positive characteristic $p>0$ and $G$ a finite local $k$-group scheme. If there exists an injective homomorphism $\X(G)\hookrightarrow(\Q_p/\Z_p)^{\oplus \gamma+n-1}$, then does $G$ appear as a finite quotient of $\pi^{\N}(U)$? 
Here, $\X(G)$ stands for the group of characters of $G$, i.e.\ $\X(G)\Def\Hom(G,\G_{m,k})$, and $\gamma$ is the $p$-rank of the compactification $X$ of $U$ and $n=\#(X\setminus U)$. 
\end{ques}  

Here recall that if a finite local $k$-group scheme $G$ appears as a quotient of $\pi^{\N}(U)$, then the character group $\X(G)$ must be embedded into the group $(\Q_p/\Z_p)^{\oplus\gamma+n-1}$~(cf.~\cite[Proposition 3.1]{ot17}).

In \cite{bb17preprint}, Biswas--Borne introduced the notion of \textit{tamely ramified $G$-torsors} for an arbitrary finite group scheme $G$ as a generalization of tamely ramified Galois covers in the usual sense~(cf.~Definition \ref{def:tame ramif}). Therefore, we can formulate a purely inseparable analogue of the strong Abhyankar conjecture due to Harbater~(cf.~Theorem \ref{thm:SAC}). 

\begin{ques}(Purely inseparable analogue of the strong Abhyankar conjecture)\label{ques:PISAC}
Let $U$ be an affine smooth curve over an algebraically closed field $k$ of positive characteristic $p>0$ and $G$ a finite local $k$-group scheme. If there exists an injective homomorphism $\X(G)\hookrightarrow(\Q_p/\Z_p)^{\oplus \gamma+n-1}$, then does there exist a $G$-torsor over $U$ which is tamely ramified above $X\setminus U$ except for one point such that it realizes a surjective homomorphism $\pi^{\N}(U)\twoheadrightarrow G$?
\end{ques} 

In this paper, toward the strong analogue, we will consider the inverse Galois problem over \textit{root stacks}. 
Let $X$ be a projective smooth curve over $k$ of genus $g\ge 0$ and of $p$-rank $\gamma\ge 0$. Let $\emptyset\neq U\subsetneq X$ be a nonempty open subscheme with $\# X\setminus U=n\ge 1$. Let $X\setminus U=\{x_0,x_1,\cdots,x_{n-1}\}$. We shall denote by $X_0$ the one-punctured smooth affine curve $X\setminus\{x_0\}$ and consider $\bfD=(x_i)_{i=1}^{n-1}$ as a family of reduced distinct Cartier divisors on $X$. For each family $\bfr=(r_i)_{i=1}^{n-1}\in\prod_{i=1}^{n-1}\Z_{\ge 0}$ of integers, we denote by $\sqrt[p^{\bfr}]{\bfD/X_0}$ the root stack associated with $X_0$ and the data $(\bfD,p^{\bfr})$.
Biswas--Borne showed that finite flat torsors which are representable by a $k$-scheme over root stacks give candidates of tamely ramified torsors~(cf.~\cite[\S3.3]{bb17preprint}; see also Theorem \ref{thm:biswas-borne}(1)).
Hence, Question \ref{ques:PISAC} has the following  interpretation, but the author is not sure if these are equivalent to each other~(cf.~Remark \ref{rem:intro}(1)). 

\begin{ques}(Stacky counterpart of Question \ref{ques:PISAC})\label{ques:PISAC stacky}
With the above notation, let $G$ be a finite local $k$-group scheme. Suppose that there exists an injective homomorphism $\X(G)\hookrightarrow(\Q_p/\Z_p)^{\oplus \gamma+n-1}$ of abelian groups. Then, do there exist an $(n-1)$-tuple $\bfr=(r_i)_{i=1}^{n-1}$ of integers $r_i\ge 0$ and a Nori-reduced $G$-torsor $Y\lto\sqrt[p^{\bfr}]{\bfD/X_0}$ such that $Y$ is representable by a $k$-scheme?
\end{ques}

Note that if Question \ref{ques:PISAC} has an affirmative answer, then so does Question \ref{ques:PISAC stacky}~(cf.~\cite[\S3.3]{bb17preprint}; see also Theorem \ref{thm:biswas-borne}(1)). 
Toward the strong analogue, we shall study an embedding problem over the root stacks and conclude that Question \ref{ques:PISAC stacky} has an affirmative answer at least in the solvable case. Namely, as the main result of this paper, we will prove the following result.

\begin{thm}(cf.~Theorem \ref{thm:PISAC solv})\label{main thm:int}
Suppose given an exact sequence of finite local $k$-group schemes
\begin{equation}\label{eq:thm IAC solv}
1\lto G'\lto G\lto G''\lto 1.
\end{equation}
Suppose that the following conditions are satisfied.
\begin{enumerate}
\renewcommand{\labelenumi}{(\roman{enumi})}
\item There exists an injective homomorphism $\X(G)\hookrightarrow(\Q_p/\Z_p)^{\oplus\gamma+n-1}$.
\item There exists a surjective $k$-homomorphism $\pi^{\N}(\sqrt[\mathbf{p}^{\infty}]{\bfD/X_0})\Def\varprojlim_{\bfr}\pi^{\N}(\sqrt[p^{\bfr}]{\bfD/X_0})\twoheadrightarrow G''$. 
\item $G'$ is solvable.
\end{enumerate}
Then there exist an $(n-1)$-tuple $\bfr$ of non-negative integers and a Nori-reduced $G$-torsor $Y\lto\sqrt[p^{\bfr}]{\bfD/X_0}$ which is representable by a $k$-scheme. In particular, Question \ref{ques:PISAC stacky} has an affirmative answer for any finite local solvable $k$-group scheme $G$. 
\end{thm}

Here, let us recall the definition of \textit{solvable} group schemes $G$. For an affine group scheme $G$ over a field $k$, we denote by $D(G)$ the \textit{derived subgroup} of $G$~(cf.~\cite[\S10.1]{wa79}\cite[Definition A.1]{bb17preprint}). Then it turns out that the quotient $G^{\rm ab}\Def G/D(G)$ is abelian and the natural surjection $G\lto G^{\rm ab}$ is universal among morphisms $G\lto H$ into abelian $k$-group schemes~(cf.~\cite[Appendix \S A.1]{bb17preprint}). In particular, $G$ is abelian if and only if $D(G)=1$. Moreover, by the universality of the morphism $G\lto G^{\rm ab}$, it can be seen that any automorphism of an affine $k$-group scheme $G$ induces an automorphism of the derived subgroup $D(G)$. In particular, if $G'$ is a normal closed subgroup scheme of $G$, then the derived subgroup $D(G')$ of $G'$ is also normal in $G$. 
For any integer $m>0$, we define the subgroup scheme $D^m(G)$ of $G$ inductively by
\begin{equation*}
D^0(G)\Def G,\quad D^m(G)\Def D(D^{m-1}(G)).
\end{equation*}
Then an affine $k$-group scheme $G$ is said to be \textit{solvable} if $D^m(G)=0$ for sufficiently large $m>0$~(cf.~\cite[\S10.1]{wa79}). It turns out that an affine $k$-group scheme $G$ is solvable if and only if it admits a sequence of closed subgroup schemes
\begin{equation*}
1=G_l<G_{l-1}<\cdots<G_0=G
\end{equation*}
such that each $G_i$ is normal in $G_{i-1}$ with $G_{i-1}/G_i$ abelian~(cf.~\cite[\S16 Exercise 2]{wa79}).

By restricting to the affine curve $U$, as a consequence of Theorem \ref{main thm:int}, we get the following result, which is motivated by Serre's Theorem \ref{thm:Serre}. This is a generalization of one of the previous results due to the author~(cf.~\cite[Proposition 3.4]{ot17}).

\begin{cor}(cf.~Corollary \ref{cor:PIAC solv})\label{cor:int 1}
For any finite local solvable $k$-group scheme $G$, the group scheme $G$ appears as a quotient of $\pi^{\N}(U)$ if and only if the character group $\X(G)$ can be embedded into the abelian group $(\Q_p/\Z_p)^{\oplus\gamma+n-1}$. Namely, the purely inseparable analogue of the Abhyankar conjecture (Question \ref{ques:PIAC}) has an affirmative answer for any finite local solvable $k$-group scheme $G$.  
\end{cor}

On the other hand, in \cite{bb17preprint}, Biswas--Borne also proved that if $G$ is abelian, then there exists an equivalence of categories between the category of tamely ramified $G$-torsors and the category of $G$-torsors of root stacks which are representable by a scheme~(cf.~\cite[\S3.4]{bb17preprint}; see also Theorem \ref{thm:biswas-borne}(2)). Therefore, as a consequence of Theorem \ref{main thm:int}, we have the following result.

\begin{cor}(cf.~Corollary \ref{cor:PISAC solv})\label{cor:int 2}
With the above notation, let $G$ be a finite local abelian $k$-group scheme. Suppose that there exists an injective homomorphism $\X(G)\hookrightarrow(\Q_p/\Z_p)^{\oplus \gamma+n-1}$. Then there exists an $(n-1)$-tuple $\bfr=(r_i)_{i=1}^{n-1}$ of integers $r_i\ge 0$ and a tamely ramified $G$-torsor $Y\lto X_0$ with ramification data $(\bfD,p^{\bfr})$ such that the restriction $Y\times_{X_0} U\lto U$ gives a Nori-reduced $G$-torsor. Namely, the purely inseparable analogue of the strong Abhyankar conjecture (Question \ref{ques:PISAC}) has an affirmative answer for any finite local abelian $k$-group scheme $G$.  
\end{cor}

The following are further remarks. 

\begin{rem}\label{rem:intro}~
\begin{enumerate}
\renewcommand{\labelenumi}{(\arabic{enumi})}
\item As we state just before Corollary \ref{cor:int 2}, by \cite[\S3.4]{bb17preprint}, for a finite abelian $k$-group scheme $G$, a $G$-torsor $Y\lto\sqrt[\bfr]{\bfD/X}$ which is representable by a $k$-scheme gives rise to a tamely ramified $G$-torsor $Y\lto X$ with ramification data $(\bfD,\bfr)$. However, the argument in \cite[\S3.4]{bb17preprint} does not work for non-abelian $k$-group schemes, which is noticed by David Rydh~(cf.~\cite[Appendix \S B]{bb17preprint}, see also Remark \ref{rem:Rydh}). In this point, the author is not sure whether or not Questions \ref{ques:PISAC} and \ref{ques:PISAC stacky} are equivalent to each other. 

\item Although Questions \ref{ques:PISAC} and \ref{ques:PISAC stacky} are motivated by Harbater's Theorem \ref{thm:SAC}, our proof of the main theorem is done by solving embedding problems and it is more like  Pop's approach~(cf.~Theorem \ref{thm:Pop}). In fact, one can ask whether an analogous result holds in our situation as follows.  
\end{enumerate}
\end{rem}

Let $G$ be a finite local $k$-group scheme. Then we have an  exact sequence
\begin{equation}\label{eq:int G to DXG}
1\lto G'\lto G\lto\Diag(\X(G))\lto 0, 
\end{equation}
where $\Diag(\X(G))$ is the diagonalizable group scheme associated with the group $\X(G)$ of characters of $G$, and $G'$ is the kernel of the natural surjective homomorphism $G\twoheadrightarrow\Diag(\X(G))$. Then one can check that $\X(G')=1$.
Now we have reached the following question motivated by Pop's Theorem \ref{thm:Pop}.

\begin{ques}(Purely inseparable analogue of Theorem \ref{thm:Pop})\label{ques:PISAC Pop}
With the same notation as in Questions \ref{ques:PISAC} and \ref{ques:PISAC stacky}, suppose given a finite local $k$-group scheme $G$ whose character group $\X(G)$ can be embedded into the group $(\Q_p/\Z_p)^{\oplus \gamma+n-1}$. Suppose further that $G'\Def\Ker(G\twoheadrightarrow\Diag(\X(G)))$ is a quotient of the Nori fundamental group scheme $\pi^{\N}(\A^1_k)$ of the affine line $\A^1_k$, i.e.\ Question \ref{ques:PIAC} has an affirmative answer for $U=\A^1_k$ and for $G'$. Then do there exist an $(n-1)$-tuple $\bfr$ and a Nori-reduced $G$-torsor over $\sqrt[p^{\bfr}]{\bfD/X_0}$ which is representable by a $k$-scheme? Moreover, does there exist a Nori-reduced $G$-torsor over $U$ which is tamely ramified above $X\setminus U$ except for one point? 
\end{ques}

One of our missing points for this direction is that we have no answer to the embedding problem 
\begin{equation*}
\begin{xy}
\xymatrix{
&&&\pi^{\N}(\sqrt[\mathbf{p}^{\infty}]{\bfD/X_0})\ar@{->>}[d]&\\
1\ar[r]&G'\ar[r]&G\ar[r]&\Diag(\X(G))\ar[r]&1
}
\end{xy}
\end{equation*}
unless $G'$ is solvable~(cf.~Lemma \ref{lem:PISAC solv}). On the other hand, under the assumption that $G'$ is solvable, the former question in Question \ref{ques:PISAC Pop} has an affirmative answer without the further assumption on $G'$, which is a  part of the main theorem (Theorem \ref{main thm:int}).

Finally we explain the organization of the present paper.

In \S\ref{sec:FG}, we recall the notion of \textit{fundamental group schemes} and \textit{fundamental gerbes}, and their tannakian interpretations, following Borne--Vistoli \cite{bv15} and Tonini--Zhang \cite{tz17}. 
In \cite{bb17}, Biswas--Borne studied the fundamental gerbes of proper tame stacks, e.g.\ root stacks  associated with proper $k$-schemes. On the other hand, we have to deal with root stacks associated with non-proper curves, so more general setups, e.g.\ a tannakian interpretation of the Nori fundamental gerbes for non-proper tame stacks might be required. However,  our main interest is the fundamental group schemes of smooth root stacks, for which a more elementary setup is sufficient, hence one can access the proof of the main theorem in a more direct way~(cf.~\S\ref{subsec:Nori ger root stack}), especially without relying on the tannakian interpretation in terms of \textit{Frobenius divided sheaves}~(cf.~\S\ref{sec:Fdiv}), which will be used for dropping the smoothness assumption~(cf.~Appendix \ref{sec:Nori ger root stack appendix}).      

In \S\ref{sec:RS}, we recall the definition and basic properties of \textit{root stacks}, for which the main references are \cite{cadman} and \cite[Chapter 1]{ta}. In \S\ref{sec:FG} and \S\ref{sec:RS}, there are no original results and, all the results are well-understood by the experts.
 
In \S\ref{sec:EP}, after reviewing a recent work due to Biswas--Borne on tamely ramified torsors~\cite[\S3]{bb17preprint}, we shall give a proof of Theorem \ref{main thm:int}. The idea is that a quite similar argument as in the proof of \cite[Proposition 3.4]{ot17} still works even if one replaces an affine curve $U$ by the pro-system of root stacks $\{\sqrt[p^{\bfr}]{\bfD/X_0}\}_{\bfr}$. The main ingredients of the proof are Lieblich's work on \textit{twisted sheaves}~\cite{li08} (cf.~\S\ref{subsec:Br}) and a theorem of Alper on vector bundles over algebraic stacks having \textit{good moduli spaces}~\cite{alp13}. The former one ensures that our  embedding problem is always unobstructed~(cf.~Lemma \ref{lem:PISAC solv}), and the latter one ensures the existence of representable torsors~(cf.~Lemma \ref{lem:tors rs}).

In Appendix \S\ref{sec:local global}, as a related topic, we consider the possibility of a generalization of the \textit{Katz--Gabber correspondence}~(cf.~\cite{ka86}) for finite local torsors, and prove that a na\"ive generalization is far from true~(cf.~Proposition  \ref{prop:ab varpi eta}).

\begin{ack}
The present paper is based on a part of the author's doctoral thesis. The author would like to express his gratitude to his advisor Professor Takao Yamazaki for his long-term support and guidance. 
The work in the present paper was started in Berlin. The author would like to thank Professor H\'el\`ene Esnault for accepting his stay at Freie Universit\"at Berlin and all of the members of her workgroup for their kindness. The author especially would like to say thank you to Doctor Fabio Tonini for answering his questions and having discussions on root stacks and the fundamental gerbes of them and Doctor Lei Zhang for answering his questions on algebraic stacks. 
The author thanks Professor Takeshi Saito and Professor Takehiko Yasuda for the suggestions he received about the Katz--Gabber correspondence. 
The author thanks Professor Niels Borne for sending him the latest updated version of the paper on tamely ramified torsors and for kindly informing about the counter-examples. The author thanks  Professor Florian Pop for sending him the paper on the further generalization of Harbater's strong Abhyankar conjecture and for suggesting proceeding the present work in that direction. The author thanks Professor Madhav Nori for accepting his visit to Chicago and for fruitful discussions and suggestions. The author thanks the two referees for carefully reading the first version of this paper and for fruitful criticisms and detailed comments.   
The author's visit to Berlin was supported by JSPS Overseas Challenge Program for Young Researchers. During the whole his doctoral course, the author was supported by JSPS Grant-in-Aid for JSPS Research Fellow, Grant Number 16J02171. The author is supported by JSPS Grant-in-Aid for JSPS Research Fellow, Grant Number 19J00366. 
\end{ack}

%%%%%%%%%%%%%%%%%%%%%%%%%%%%%%%%%%%%%%%%%%%%%%%%%%%%%%%%%%%%
%%%%%%%%%%%%%%%%%%%%%%%%%%%%%%%%%%%%%%%%%%%%%%%%%%%%%%%%%%%%

\section*{Notation}

In the present paper, $k$ always means a field. 
We denote by $\Vecf_k$ the category of finite dimensional vector spaces over a field $k$. For an affine $k$-group scheme $G$, we denote by $\Rep(G)$ the category of finite dimensional left $k$-linear representations of $G$. For each $(V,\rho)\in\Rep(G)$, we denote by $V^G$ the $G$-invariant subspace of $V$, i.e.\ $V^G=\{v\in V\,|\,\rho(v)=v\otimes 1\}$.

In the present paper, for the basic definitions and notions about algebraic stacks, we follow Stacks Project~\cite{stack}. In particular, as a definition of algebraic spaces and algebraic stacks, we adopt \cite[Definition 025X]{stack} and \cite[Definition 026O]{stack} respectively. For any algebraic stack $\calX$, we denote by $|\calX|$ the underlying topological space~(cf.~\cite[04XE]{stack}) and by $|\calX|_0$ the set of closed points of $\calX$.

For a ring $A$, we denote by $(\Aff/A)$ the category of affine $A$-schemes. We consider the fibered categories
\begin{equation*}
\Vect\subseteq\Qcohfp\subseteq\Qcoh
\end{equation*} 
over $(\Aff/A)$ of locally free sheaves of finite rank, quasi-coherent sheaves of finite presentation and quasi-coherent sheaves respectively. For any fibered category $\calX$ over $(\Aff/A)$, we define the categories
\begin{equation*}
\Vect(\calX)\subseteq\Qcohfp(\calX)\subseteq\Qcoh(\calX)
\end{equation*}
to be
\begin{equation*}
\Vect(\calX)\Def\Hom_{A}(\calX,\Vect)
\end{equation*}
and similarly $\Qcohfp(\calX)\Def\Hom_A(\calX,\Qcohfp)$ and $\Qcoh(\calX)\Def\Hom_A(\calX,\Qcoh)$.
For a field $k$, we have $\Vect(\Spec k)=\Vecf_k$. If $\calX$ is a Noetherian algebraic stack over a scheme $S$, we can also consider the category $\Coh(\calX)$ of coherent sheaves on $\calX$. If $\calX=\cB_S G$ is the classifying stack of an affine flat and finitely presented $S$-group scheme $G$, then  $\Qcoh(\cB_S G)$ is nothing but the category of \textit{$G$-equivariant sheaves} over $S$, i.e.\  quasi-coherent sheaves $\calF$ on $S$ endowed with an action of $G$~(cf.~\cite[\S2.1]{aov08}). Therefore, we get a natural functor $\Qcoh(\cB_S G)\lto\Qcoh(S)$ which maps each $G$-equivariant sheaf $\calF$ to the $G$-invariant subsheaf $\calF^G$. In the case where $S=\Spec k$ is the spectrum of a field $k$, all the three categories $\Coh(\cB G)$, $\Vect(\cB G)$ and $\Rep(G)$ are canonically equivalent to each other,
\begin{equation*}
\Coh(\cB G)=\Vect(\cB G)=\Rep(G).
\end{equation*} 

For a scheme $X$ over a field $k$ of characteristic $p>0$ and a positive integer $n>0$, we define the $n$th Frobenius twist $X^{(n)}$ by the Cartesian diagram
\begin{equation*}
\begin{xy}
\xymatrix{\ar@{}[rd]|{\square}
X^{(n)}\ar[r]\ar[d]&X\ar[d]\\
\Spec k\ar[r]^{F^n}&\Spec k,
}
\end{xy}
\end{equation*} 
where $F$ is the absolute Frobenius of $\Spec k$. Then the $n$th power of the absolute Frobenius morphism $F^n:X\lto X$ factors uniquely through $X^{(n)}$ and we get a $k$-morphism $X\lto X^{(n)}$, which is denoted by $F^{(n)}$ and is called the $n$th relative Frobenius morphism of $X$. If $k$ is perfect, the projection $X^{(n)}\lto X$ is an isomorphism of schemes, but not of $k$-schemes. Moreover, in this case, we can also consider the $(-n)$th Frobenius twist $X^{(-n)}$ of $X$ by using the inverse morphism $F^{-n}$ of $F^n:\Spec k\xrightarrow{~\simeq~}\Spec k$. If $\calX\lto\Spec k$ is a fibered category of groupoids  over a field $k$ of characteristic $p>0$, we define the absolute Frobenius morphism $F:\calX\lto\calX$ as the collection of maps
\begin{equation*}
\calX(S)\lto\calX(S)\,;\,\xi\longmapsto F_S^*(\xi)=\xi\circ F_S
\end{equation*}
for $S\in(\Aff/k)$~(cf.~\cite[Notations and conventions]{tz17}). If $\calX=X$ is a scheme over $k$, then this coincides with the one in the previous sense. Moreover, we can also define the relative Frobenius morphisms $F^{(n)}:\calX\lto\calX^{(n)}$ for $n>0$ in the same manner as before.

%%%%%%%%%%%%%%%%%%%%%%%%%%%%%%%%%%%%%%%%%%%%%%%%%%%%%%%%%%%%
%%%%%%%%%%%%%%%%%%%%%%%%%%%%%%%%%%%%%%%%%%%%%%%%%%%%%%%%%%%%

\section{Fundamental gerbes and their tannakian interpretations}\label{sec:FG}

\subsection{Fundamental gerbes}

Let $k$ be a field. A \textit{finite stack} over $k$ is an algebraic stack $\Gamma$ over $k$ which has finite flat diagonal and admits a flat surjective morphism $U\lto\Gamma$ for some finite $k$-scheme $U$~(cf.~\cite[Definition 4.1]{bv15}). A \textit{finite gerbe} over $k$ is a finite stack over $k$ which is a gerbe in the fppf topology. A finite stack $\Gamma$ is a finite gerbe if and only if it is geometrically connected and geometrically reduced~(cf.~\cite[Proposition 4.3]{bv15}). A finite gerbe $\Gamma$ over $k$ is a \textit{tannakian gerbe} over $k$ in the sense of \cite[III \S2]{sa72}. A \textit{profinite gerbe} over $k$ is a tannakian gerbe over $k$ which is equivalent to a projective limit of finite gerbes over $k$~(cf.~\cite[Definition 4.6]{bv15}).

Let $\calX$ be a fibered category in groupoids over $k$. Suppose that $\calX$ is \textit{inflexible} over $k$ in the sense of \cite[Definition 5.3]{bv15}, namely, it admits a profinite gerbe $\Pi$ over $k$ together with a morphism $\calX\lto\Pi$ such that, for any finite stack $\Gamma$ over $k$, the induced functor
\begin{equation*}
\Hom_k(\Pi,\Gamma)\lto\Hom_k(\calX,\Gamma)
\end{equation*} 
is an equivalence of categories~(cf.~\cite[Theorem 5.7]{bv15}). If such a gerbe $\Pi$ exists for $\calX/k$, it  is unique up to unique isomorphism, so we denote it by $\Pi_{\calX/k}^{\N}$, or $\Pi^{\N}_{\calX}$ for simplicity, and call it the \textit{Nori fundamental gerbe} for $\calX$ over $k$~(cf.~\cite{bv15}). If $\calX$ is an algebraic stack of finite type over $k$, from the definition of $\Pi_{\calX/k}^{\N}$, for any finite group scheme $G$ over $k$, there exists a natural bijection
\begin{equation*}
\Hom_k(\Pi_{\calX/k}^{\N},\cB_kG)\xrightarrow{~\simeq~}\Hom_k(\calX,\cB_kG)=H_{\fppf}^1(\calX,G).
\end{equation*}
In this sense, the Nori fundamental gerbe $\Pi_{\calX/k}^{\N}$ classifies $G$-torsors over $\calX$ for any finite group scheme $G$ over $k$. 
If we suppose that an inflexible algebraic stack $\calX$ over $k$ has a $k$-rational point $x\in \calX(k)$, then the composition $\Spec k\xrightarrow{~x~}\calX\lto\Pi_{\calX/k}^{\N}$ defines a section $\xi\in\Pi_{\calX/k}^{\N}(k)$, whence $\Pi^{\N}_{\calX/k}\simeq\cB_k\underline{\mathrm{Aut}}_k(\xi)$. We denote by $\pi^{\N}(\calX,x)$ the group scheme $\underline{\mathrm{Aut}}_k(\xi)$ over $k$, which is nothing other than the \textit{fundamental group scheme} of $(\calX,x)$ in the sense of Nori~\cite[Chapter II]{no82}. Namely, for any finite group scheme $G$ over $k$, the set of homomorphisms $\Hom_k(\pi^{\N}(\calX,x),G)$ is naturally bijective onto the set of isomorphism classes of pointed $G$-torsors over $(\calX,x)$. We call the profinite group scheme $\pi^{\N}(\calX,x)$ the \textit{Nori fundamental group scheme} of $(\calX,x)$.

Let $\calX$ be an inflexible algebraic stack of finite type over $k$. A morphism $\calX\lto\Gamma$ into a finite gerbe $\Gamma$ is said to be \textit{Nori-reduced}~(cf.~\cite[Definition 5.10]{bv15}) if for any factorization $\calX\lto\Gamma'\lto\Gamma$ where $\Gamma'$ is a finite gerbe and $\Gamma'\lto\Gamma$ is faithful, then $\Gamma'\lto\Gamma$ is an isomorphism. According to \cite[Lemma 5.12]{bv15}, for any morphism $\calX\lto\Gamma$ into a finite gerbe, there exists a unique factorization $\calX\lto\Delta\lto\Gamma$, where $\Delta$ is a finite gerbe,  $\calX\lto\Delta$ is Nori-reduced and $\Delta\lto\Gamma$ is representable. A $G$-torsor $\calP\lto \calX$ is said to be \textit{Nori-reduced} if the morphism $\calX\lto\cB G$ is Nori-reduced.

Next we introduce two variants of the Nori fundamental gerbe, namely, the \textit{\'etale} and \textit{local fundamental gerbes}~(cf. \cite[\S8]{bv15}\cite[\S4 Definition 4.1]{tz17}).

\begin{definition}(cf.~\cite[\S3 Definition 3.1]{tz17})
A finite stack $\Gamma$ over $k$ is said to be \textit{\'etale} if it admits a finite \'etale surjective  morphism $U\lto\Gamma$ from a finite \'etale $k$-scheme $U$.\end{definition}

If $k$ is of characteristic 0, then any finite gerbe is \'etale. Hence, suppose that $k$ is of positive characteristic $p>0$. 
To define a well-defined notion of \textit{finite local stack}, we need more preliminaries. 
For each $k$-algebra $A$, we denote by $A_{\et}$ as the union of $k$-subalgebras of $A$ which are finite \'etale over $k$~(cf.~\cite[\S2 Definition 2.1]{tz17}). Note that $A_{\et}$ depends on the base field $k$. If $A\twoheadrightarrow B$ is a surjective $k$-algebra homomorphism with nilpotent kernel, then the induced homomorphism $A_{\et}\to B_{\et}$ is an isomorphism~(cf. \cite[\S2 Remark 2.2]{tz17}). Moreover, for the $i$th relative Frobenius homomorphism
\begin{equation*}
A^{(i)}=A\otimes_{k}k\lto A~;~a\otimes\lambda\longmapsto a^{p^i}\lambda,
\end{equation*}
the induced morphism $(A^{(i)})_{\et}\to A_{\et}$ is an isomorphism~(cf. \cite[\S2 Remark 2.3]{tz17}). If $A$ is a finite $k$-algebra, then there exists an integer $n\ge 0$ such that the image of the relative Frobenius homomorphism 
$A^{(n)}\to A$ is an \'etale $k$-algebra and the residue fields of $A^{(n)}$ are separable over $k$~(cf. \cite[Lemma 2.4]{tz17}). In particular, the surjective homomorphism $A^{(n)}\twoheadrightarrow (A^{(n)})_{\rm red}$ induces an isomorphism $(A^{(n)})_{\et}\xrightarrow{\simeq}(A^{(n)})_{\rm red}$. 

For each affine scheme $U=\Spec A$ over $k$, we define $U^{\et}\Def\Spec A_{\et}$. Note that the canonical morphism $U\to U^{\et}$ is faithfully flat. Let $R\rightrightarrows U$ be a flat groupoid where $R$ and $U$ finite over $k$. Then the induced morphisms $R^{\et}\rightrightarrows U^{\et}$ define a groupoid. Moreover, if the residue fields of $R$ and $U$ are separable over $k$, then the induced morphisms $R_{\rm red}\rightrightarrows U_{\rm red}$ also define a groupoid which is canonically isomorphic to the first one $R^{\et}\rightrightarrows U^{\et}$~(cf.~\cite[\S3 Lemma 3.4]{tz17}). 

\begin{definition}(cf.~\cite[\S3 Definition 3.5]{tz17})\label{def:et quot ger}
Let $\Gamma$ be a finite stack over a field $k$ and let $U\lto\Gamma$ be a finite atlas from an affine $k$-scheme $U$ with $R=U\times_{\Gamma}U$. We define an \textit{\'etale quotient} $\Gamma^{\et}$ as the stack associated with the groupoid $R^{\et}\rightrightarrows U^{\et}$.
\end{definition}

\begin{lem}(cf.~\cite[\S3 Lemma 3.6]{tz17})
With the same notation as in Definition \ref{def:et quot ger}, for any finite \'etale stack $E$ over $k$, the induced morphism
\begin{equation*}
\Hom_k(\Gamma^{\et},E)\lto\Hom_k(\Gamma,E)
\end{equation*}
is an equivalence. Moreover, for any $i\ge 0$, the morphism $\Gamma^{\et}\lto(\Gamma^{(i)})^{\et}$ is an equivalence and for any sufficiently large $i\gg 0$, the functor $\Gamma^{(i)}\lto(\Gamma^{(i)})^{\et}$ has a section, whence the relative Frobenius morphism $\Gamma\lto\Gamma^{(i)}$ factors through $\Gamma^{\et}$. 
\end{lem}

In particular, the \'etale quotient $\Gamma^{\et}$ for a finite stack $\Gamma$ is independent of the choice of atlas $U\lto\Gamma$ and is unique up to unique isomorphism. 

\begin{definition}(cf.~\cite[\S3 Definition 3.9]{tz17})
A finite stack $\Gamma$ over $k$ is said to be \textit{local} if $\Gamma^{\et}=\Spec k$. Moreover, we define a \textit{pro-\'etale gerbe}, or \textit{pro-local gerbe} in the same manner as \cite[Definition 4.6]{bv15}.
\end{definition}

\begin{definition}\label{def:et loc ger}
Let $\calX$ be a fibered category in groupoids over $k$. An \textit{\'etale fundamental gerbe} (respectively \textit{local fundamental gerbe}) for $\calX$ is a pro-\'etale gerbe (respectively a pro-local gerbe) over $k$ together with a morphism $\calX\lto\Pi$ such that for any finite \'etale stack (respectively finite local stack) $\Gamma$ over $k$, the induced functor
\begin{equation*}
\Hom_k(\Pi,\Gamma)\lto\Hom_k(\calX,\Gamma)
\end{equation*}
is an equivalence of categories. If such a one $\Pi$ exists, it is unique up to unique isomorphism and we denote by $\Pi_{\calX/k}^{\et}$ (respectively $\Pi_{\calX/k}^{\loc}$) the \'etale fundamental gerbe (respectively the local fundamental gerbe) for $\calX$.
\end{definition}

\begin{prop}\label{prop:exist et loc ger}
Let $\calX$ be an algebraic stack of finite type over $k$. Then we have the following.
\begin{enumerate}
\renewcommand{\labelenumi}{(\arabic{enumi})}
\item The \'etale fundamental gerbe $\Pi_{\calX/k}^{\et}$ exists if and only if $\calX$ is geometrically connected, or equivalently, $H^0(\scrO_{\calX})_{\et}=k$.
\item Suppose that $\calX$ is reduced. Then the local fundamental gerbe $\Pi_{\calX/k}^{\loc}$ exists if and only if $H^0(\scrO_{\calX})$ does not contain nontrivial purely inseparable extensions of $k$.
\item If $\calX$ is inflexible, then the \'etale fundamental gerbe $\Pi^{\et}_{\calX/k}$ and the local fundamental gerbe $\Pi^{\loc}_{\calX/k}$ exist.
\end{enumerate} 
\end{prop}

\begin{proof}
(1) See \cite[\S2 Lemma 2.7; \S4 Proposition 4.3]{tz17}.
 
(2) See \cite[\S7 Theorem 7.1]{tz17}.
 
(3) See \cite[\S4 Remark 4.2]{tz17}.
\end{proof}

\subsection{Tannakian reconstruction and recognition}\label{sec:tann reconst}

A \textit{pseudo-abelian} category is an additive category $\calC$ together with a family $J_{\calC}$ of sequences of the form $c'\longrightarrow c\longrightarrow c''$ in $\calC$~(cf.~\cite[Definition]{tz17}). A $\Z$-linear functor $\Phi:\calC\lto\calD$ of pseudo-abelian categories is said to be \textit{exact} if it maps any sequence of $J_{\calC}$ to a sequence isomorphic to a one of $J_{\calD}$.

Let $R$ be a ring. If $\calX$ is a fibered category over $(\Aff/R)$, then we consider the category $\Vect(\calX)$ of vector bundles on $\calX$ as a pseudo-abelian category with the family $J_{\Vect(\calX)}$ of \textit{pointwise exact} sequences
\begin{equation*}
\calF'\longrightarrow\calF\longrightarrow\calF''
\end{equation*}
i.e.\ for any affine $R$-scheme $T$ and any object $\xi\in\calX(T)$, the sequence
\begin{equation*}
\calF'(\xi)\longrightarrow\calF(\xi)\longrightarrow\calF''(\xi)
\end{equation*}
of vector bundles on $T$ is exact. If $\calC$ and $\calD$ are $R$-linear, monoidal and pseudo-abelian categories, we denote by $\Hom_{R,\otimes}(\calC,\calD)$ the category whose objects are $R$-linear exact monoidal functors and whose morphisms are natural monoidal isomorphisms. If $f:\calX\lto\calY$ is a base preserving functor of categories $\calX$ and $\calY$ over $(\Aff/R)$, $f^*\in\Hom_{R,\otimes}(\Vect(\calY),\Vect(\calX))$.

Let $\calC$ be a pseudo-abelian monoidal $R$-linear category. We define an fpqc stack $\Pi_{\calC}$ in groupoids over $R$ by attaching to each affine $R$-scheme $T$ the category
\begin{equation*}
\Pi_{\calC}(T)\Def\Hom_{R,\otimes}(\calC,\Vect(T)).
\end{equation*}
Note that there exists a natural monoidal $R$-linear functor over $(\Aff/R)$
\begin{equation}\label{eq:tann recg}
\Phi:\calC\lto\Vect(\Pi_{\calC})~;~c\longmapsto\bigl(\Pi_{\calC}(T)\ni\xi\mapsto\xi(c)\in\Vect(T)\bigl).
\end{equation}
A pseudo-abelian monoidal $R$-linear category $\calC$ over $(\Aff/R)$ is said to satisfy \textit{tannakian recognition}~(cf.~\cite[\S1]{tz17}) if the functor (\ref{eq:tann recg}) is an equivalence and for any sequence $\chi:c'\longrightarrow c\longrightarrow c''$, $\Phi(\chi)$ is pointwise exact if and only if $\chi\in J_{\calC}$.

Let $\calX$ be a fibered category over $(\Aff/R)$. Then there exists a base preserving functor
\begin{equation}\label{eq:tann recn}
\calX\lto\Pi_{\Vect(\calX)}.
\end{equation}
A fibered category $\calX$ in groupoids over $(\Aff/R)$ is said to satisfy \textit{tannakian reconstruction}~(cf.~\cite[\S1]{tz17}) if the functor (\ref{eq:tann recn}) is an equivalence.

Then a classical Tannaka duality (cf.~\cite[Th\'eor\`eme 1.12]{de90}) can be restated in the following way.

\begin{thm}(cf.~\cite[\S1 Example 1.5]{tz17})\label{thm:tann duality}
Let $k$ be a field. If $\calT$ is a $k$-tannakian category, then it satisfies tannakian recognition and $\Pi_{\calT}$ is a tannakian gerbe over $k$. Conversely, if $\Pi$ is a tannakian gerbe over $k$, then it satisfies tannakian reconstruction and $\Vect(\Pi)$ is a $k$-tannakian category.  
\end{thm}

Moreover, Nori's reconstruction theorem (cf.~\cite[Chapter I, \S2.2 Proposition 2.9]{no82}) can be restated in the following way.

\begin{prop}\label{prop:nori reconst}
Let $G$ be an affine group scheme over $k$ and $\calX$ a fibered category in goupoids over $k$. Then there exists a natural equivalence of categories,
\begin{equation*}
\Hom_k(\calX,\cB G)\xrightarrow{~\simeq~}\Hom_{k,\otimes}(\Vect(\cB G),\Vect(\calX)).
\end{equation*}
\end{prop}

This is valid because $\cB G=\Pi_{\Vect(\cB G)}$ and the natural functor
\begin{equation*}
\Hom_k(\calX,\Pi_{\Vect(\cB G)})\lto\Hom_{k,\otimes}(\Vect(\cB G),\Vect(\calX)).
\end{equation*}
is an equivalence of categories~(cf.~\cite[\S1]{tz17}).

\subsection{Tannakian interpretation in the pseudo-proper case: Essentially finite bundles}

In this subsection, we recall a tannakian interpretation of the Nori fundamental gerbe under a properness assumption,  which was originally given by Nori \cite{no76} for the fundamental group scheme. We shall follow a simplified argument due to Borne--Vistoli~\cite[\S7]{bv15}.

\begin{definition}(cf.~\cite[\S7 Definition 7.1]{bv15})
A fibered category $\calX$ over $k$ is said to be \textit{pseudo-proper} if it satisfies the following conditions.
\begin{enumerate}
\renewcommand{\labelenumi}{(\roman{enumi})}
\item There exists a quasi-compact scheme $U$ and a morphism $U\lto\calX$ which is representable, faithfully flat, quasi-compact and quasi-separated.
\item For any locally free sheaf $E$ of $\scrO_{\calX}$-modules on $\calX$, the $k$-vector space $H^0(\calX,E)$ is finite-dimensional.  
\end{enumerate}
\end{definition}

\begin{ex}(cf.~\cite[\S7 Examples 7.2]{bv15})
\begin{enumerate}
\renewcommand{\labelenumi}{(\arabic{enumi})}
\item A finite stack $\Gamma$ over $k$ is pseudo-proper. 
\item A tannakian gerbe $\Phi$ over $k$ is pseudo-proper. 
\end{enumerate}
\end{ex}

Now let $\calX$ be a pseudo-proper fibered category in groupoids over a field $k$. Then, the category $\Vect(\calX)$ of vector bundles on $\calX$ is a $k$-linear rigid tensor category with finite-dimensional Hom vector spaces, in which the idempotents split, and the Krull--Schmidt theorem holds in $\Vect(\calX)$. Namely, every object $E$ of $\Vect(\calX)$ can be described as a direct product of indecomposable objects $E_i~(1\le i\le n)$,
$E\simeq\oplus_{i=1}^n E_i$,
and moreover such a description of $E$ is unique up to isomorphism.    

\begin{definition}(cf.~\cite[\S7 Definition 7.5]{bv15})
A vector bundle $E\in\Vect(\calX)$ is said to be \textit{finite} if there exist $f$ and $g$ in $\mathbb{N}[t]$ with $f\neq g$ such that $f(E)\simeq g(E)$, or equivalently, the set of isomorphism classes of indecomposable components of all the powers of $E$ is finite.   
\end{definition}

In particular, if $E$ is a finite bundle over $\calX$, then all the indecomposable components are also finite.  
Moreover, if $E$ and $E'$ are finite bundles on $\calX$, then the direct sum $E\oplus E'$, the tensor product $E\otimes_{\scrO_{\calX}} E'$ and the dual $E^{\vee}$ are also finite~(cf.~\cite[\S7 Proposition 7.6]{bv15}).

\begin{definition}(cf.~\cite[\S7 Definition 7.7]{bv15})
A vector bundle $E$ over a pseudo-proper fibered category in groupoids over $k$ is said to be \textit{essentially finite} if it is the kernel of a homomorphism between two finite bundles. We denote by $\EFin(\calX)$ the category of essentially finite bundles over $\calX$.
\end{definition}

\begin{ex}(cf.~\cite[\S7 Proposition 7.8]{bv15})\label{ex:EFin prof ger}
Let $\Phi$ be a profinite gerbe over a field $k$. Then all the vector bundles on $\Phi$ are essentially finite, whence
\begin{equation*}
\Vect(\Phi)=\EFin(\Phi).
\end{equation*}
Indeed, if we write $\Phi=\varprojlim_{i\in I}\Gamma_i$ with $\Gamma_i$ finite, then $\Vect(\Phi)=\varinjlim_{i\in I}\Vect(\Gamma_i)$. Therefore, it suffices to prove the claim in the case where $\Phi=\Gamma$ is finite. Since $\Gamma$ is a finite gerbe over $k$, there exists a faithfully flat representable morphism $\pi:T\Def\Spec K\lto\Gamma$ from the spectrum of a field $K$ which is finite over $k$. First note that the vector bundle $\pi_*\scrO_{T}$ is finite on $\Gamma$ 
because
\begin{equation*}
\pi_*\scrO_{T}\otimes\pi_*\scrO_T\simeq\pi_*(\scrO_T\otimes\pi^*\pi_*\scrO_{T})\simeq\pi_*(\scrO_T^{\oplus d}),
\end{equation*}
for some $d>0$. Now fix an arbitrary vector bundle $E\in\Vect(\Gamma)$. Let $E$ be of rank $r>0$. Since $T$ is the spectrum of a field, we have $\pi^*E\simeq\scrO_{T}^{\oplus r}$, whence $E\hookrightarrow\pi_*\pi^*E\simeq\pi_*\scrO_T^{\oplus r}$. 
Again by using the fact that $T$ is the spectrum of a field together with flat descent, we can easily see that the cokernel of this inclusion is also a vector bundle over $\Gamma$ and can be embedded into $\pi_*\scrO_{T}^{\oplus m}$ for some $m>0$. As $\pi_*\scrO_T$ is a finite bundle over $\Gamma$, we can conclude that $E$ is essentially finite. 
\end{ex}

\begin{thm}(Nori, Borne--Vistoli, cf.~\cite{no76}\cite[\S7 Theorem 7.9]{bv15})\label{thm:EFin}
Let $\calX$ be an inflexible and pseudo-proper fibered category over a field $k$ and $\calX\lto\Pi^{\N}_{\calX/k}$ the Nori fundamental gerbe for $\calX/k$. Then the pullback functor
\begin{equation*}
\Vect(\Pi^{\N}_{\calX/k})\lto\Vect(\calX)
\end{equation*}
is fully faithful and gives an equivalence of tensor categories between $\Vect(\Pi^{\N}_{\calX/k})$ and $\EFin(\calX)$. In particular, the category $\EFin(\calX)$ of essentially finite bundles over $\calX$ is a tannakian category over $k$. 
\end{thm}

\begin{cor}(cf.~\cite[\S7 Corollary 7.10]{bv15})\label{cor:EFin}
Let $\Phi$ be a tannakian gerbe over a field $k$. Then, the full subcategory $\EFin(\Phi)$ consisting of essentially finite bundles is a tannakian category over $k$.  
\end{cor}

Let us introduce the following notation.

\begin{definition}\label{def:EFin sub}
Let $k$ be a field and $\calT$ a tannakian category over $k$. We define the full tannakian subcategory $\EFin(\calT)$ of $\calT$ to be the essential image of the fully faithful functor $\EFin(\Pi_{\calT})\hookrightarrow\Vect(\Pi_{\calT})\xleftarrow{~\simeq~}\calT$~(cf. Theorem \ref{thm:tann duality} and Corollary \ref{cor:EFin}). We denote by $\widehat{\Pi}_{\calT}$ the tannakian gerbe $\Pi_{\EFin(\calT)}$ associated with the tannakian category $\EFin(\calT)$ over $k$. By definition, $\widehat{\Pi}_{\calT}$ is a profinite gerbe over $k$.  
\end{definition}

\subsection{Formalism for tannakian interpretations}\label{subsec:tann}

Let $k$ be a field and consider two fibered categories $\calX$ and $\calX_{\calT}$ over $(\Aff/k)$ together with a base preserving functor $\pi_{\calT}:\calX\lto\calX_{\calT}$. Define $\calT(\calX)\Def\Vect(\calX_{\calT})$, which is a pseudo-abelian rigid monoidal $k$-linear category and which admits a $k$-linear monoidal exact functor $\pi_{\calT}^*:\calT(\calX)\lto\Vect(\calX)$. 
We will apply the following formalism to the fibered categories $\calX_{\calT}=\calX^{(\infty)}$~(cf.~\S\ref{sec:Fdiv}) and $\calX_{\calT}=\Spec k$~(cf.~\S\ref{sec:loc FG})

\begin{axio}(cf.~\cite[\S5 Axioms 5.2]{tz17})\label{axio:tann}
Let $L\Def \mathrm{End}_{\calT(\calX)}(\unit_{\calT(\calX)})$ and consider the following conditions. 
\begin{description}
\item[(A)] $\calT(\calX)=\Qcohfp(\calX_{\calT})$.
\item[(B)] The functor $\pi_{\calT}^*:\calT(\calX)\lto\Vect(\calX)$ is faithful.
\item[(C)] For all finite \'etale stacks $\Gamma$ over $L$, the following functor is an equivalence of categories,
\begin{equation*}
\Hom_L(\calX_{\calT},\Gamma)\lto\Hom_L(\calX,\Gamma).
\end{equation*}
\end{description}
\end{axio}

The following provides us a formalism to get a tannakian interpretation of the \'etale fundamental gerbe.

\begin{thm}(Tonini--Zhang, cf.~\cite[\S5 Theorem 5.8]{tz17})\label{thm:tann et ger}
\begin{enumerate}
\renewcommand{\labelenumi}{(\arabic{enumi})}
\item Suppose that Axiom \ref{axio:tann}(A) is satisfied and $L$ is a field. Then Axiom \ref{axio:tann}(B) is satisfied and $\calT(\calX)$ is an $L$-tannakian category. Moreover, for any tannakian gerbe $\Gamma$ over $L$, the functor
\begin{equation*}
\Hom_L(\Pi_{\calT(\calX)},\Gamma)\lto\Hom_L(\calX_{\calT},\Gamma)
\end{equation*}
is an equivalence of categories.
\item (cf.~\cite[\S5 Proposition 5.7]{tz17}) 
Suppose that Axiom \ref{axio:tann}(A) is satisfied. If $\calX$ is connected and Axiom \ref{axio:tann}(B) is satisfied, then $L$ is a field. 
\item Suppose that Axioms \ref{axio:tann}(A), (C) are satisfied and that $L$ is a field. Then the morphism $\calX\lto\Pi_{\calT(\calX)}^{\et}\Def(\Pi_{\calT(\calX)})^{\et}$ gives the \'etale fundamental gerbe for $\calX$ over $L$.
\begin{equation*}
\Vect(\Pi^{\et}_{\calX/L})\simeq\mathrm{EFin}(\calT(\calX)).
\end{equation*}
\end{enumerate}
\end{thm}

From now on, we shall assume that $k$ is a field of positive characteristic $p>0$. If $\calX$ be a fibered category over $(\Aff/k)$, then the \textit{Frobenius pullback} 
\begin{equation*}
F^*:\Vect(\calX)\lto\Vect(\calX)
\end{equation*}
is defined by applying the absolute Frobenius pointwise. The functor $F^*$ is an $\F_p$-linear exact monoidal functor.

\begin{definition}(cf.~\cite[\S5 Definition 5.11]{tz17})
For each integer $i\ge 0$, we define the category $\calT_i(\calX)$ to be the category of tuples $(\calF,\calG,\lambda)$ where 
\begin{itemize}
\item $\calF\in\Vect(\calX)$,
\item $\calG\in\calT(\calX)$, and
\item $\lambda:F^{i*}\calF\xrightarrow{~\simeq~}\calG|_{\calX}\Def\pi_{\calT}^*\calG$ is an isomorphism. 
\end{itemize}
A morphism $(\calF,\calG,\lambda)\lto(\calF',\calG',\lambda')$ is a pair of morphisms $\calF\lto\calF'$ and $\calG\lto\calG'$ which are compatible with the isomorphisms $\lambda$ and $\lambda'$. The category $\calT_i(\calX)$ is $\F_p$-linear monoidal and rigid with the unit object $\unit_{\calT_i(\calX)}=(\scrO_{\calX},\scrO_{\calX_{\calT}},{\rm id})$. We endow $\calT_i(\calX)$ with a $k$-structure via
\begin{equation*}
k\longrightarrow\End_{\calT_i(\calX)}(\unit_{\calT_i(\calX)})~;~a\longmapsto (a,a^{p^i}).
\end{equation*}
We consider $\calT_i(\calX)$ as a pseudo-abelian category together the distinguished set $J_{\calT_i(\calX)}$ of sequences which are pointwise exact. The forgetful functor $\calT_i(\calX)\lto\Vect(\calX)$ is $k$-linear monoidal and exact. 

There exists a $k$-linear monoidal and exact functor
\begin{equation}\label{eq:transition map}
\calT_i(\calX)\longrightarrow\calT_{i+1}(\calX)~;~(\calF,\calG,\lambda)\longmapsto (\calF,F^*\calG,F^*\lambda).
\end{equation}
We define $\calT_{\infty}(\calX)$ as the direct limit of the categories $\calT_i(\calX)$. The category $\calT_{\infty}(\calX)$ is a $k$-linear monoidal rigid category.
\end{definition}

The following provides us a formalism to get a tannakian interpretation of the Nori fundamental gerbe.

\begin{thm}(Tonini--Zhang, cf.~\cite[\S5 Theorem 5.14]{tz17})\label{thm:tann Nori ger}
Suppose that Axiom \ref{axio:tann}(A) is satisfied for $\pi_{\calT}:\calX\longrightarrow\calX_{\calT}$, that $L=L_0=\End_{\calT(\calX)}(\unit_{\calT(\calX)})$ is a field and the following condition holds for $\calX$. 
\begin{equation}\label{eq:tann Nori ger}
\text{For any $\calF\in\Qcohfp(\calX)$, if $F^*\calF\in\Vect(\calX)$, then $\calF\in\Vect(\calX)$}.
\end{equation}
Then we have the following.
\begin{enumerate}
\renewcommand{\labelenumi}{(\arabic{enumi})}
\item For any $i\in\mathbb{N}\cup\{\infty\}$, the ring $L_i\Def\End_{\calT_i(\calX)}(\unit_{\calT_i(\calX)})$ is a field, $\calT_i(\calX)$ is an $L_i$-tannakian category and $\Pi_{\calT_i(\calX)}$ is a tannakian gerbe over $L_i$, and the functor $\calT_i(\calX)\longrightarrow\Vect(\calX)$ is faithful monoidal and exact. 
\item The functors $\calT_i(\calX)\longrightarrow\calT_{i+1}(\calX)$ and $\calT_i(\calX)\longrightarrow\calT_{\infty}(\calX)$ are faithful monoidal,  exact and compatible with the forgetful functors $\calT_i(\calX)\longrightarrow\Vect(\calX)$. The functor $\calT_{\infty}(\calX)\lto\Vect(\calX)$ induces a morphism $\calX\lto\Pi_{\calT_{\infty}}(\calX)$, whence $\calX$ is a fibered category over $L_{\infty}$. Moreover,
\begin{equation*}
L_{\infty}=\{x\in H^0(\scrO_{\calX})~|~x^{p^i}\in L_0~\text{for some $i\ge 0$}\}
\end{equation*}
is purely inseparable over $L_0$.
\item For any $i\in\mathbb{N}\cup\{\infty\}$, we have 
\begin{equation*}
\mathrm{EFin}(\calT_i(\calX))=\{(\calF,\calG,\lambda)~|~\calG\in\mathrm{EFin}(\calT(\calX))\}
\end{equation*}
in $\calT_i(\calX)$, and $\mathrm{EFin}(\calT_{\infty}(\calX))\simeq\varinjlim_{i}\mathrm{EFin}(\calT_i(\calX))$. Moreover, the morphism $\calX\longrightarrow\Pi_{\calT_{\infty}(\calX)}^{\loc}\Def(\Pi_{\calT_{\infty}(\calX)})^{\loc}$ is the local fundamental gerbe for $\calX$ over $L_{\infty}$. 
\item Suppose also that Axiom \ref{axio:tann}(C) holds. Then $\calX\longrightarrow\widehat{\Pi}_{\calT_{\infty}(\calX)}$~(cf.~Definition \ref{def:EFin sub}) is the Nori fundamental gerbe for $\calX$ over $L_{\infty}$ and we have an equivalence of $L_{\infty}$-tannakian categories,
\begin{equation*}
\Vect(\Pi^{\N}_{\calX/L_{\infty}})~\simeq~\mathrm{EFin}(\calT_{\infty}(\calX)).
\end{equation*} 
\end{enumerate}
\end{thm}

\begin{ex}(cf.~\cite[\S 5 Remark 5.15]{tz17})
If $\calX$ is a reduced fibered category in groupoids over a field $k$. Then the condition (\ref{eq:tann Nori ger}) is fulfilled.
\end{ex}

\subsection{Example: Frobenius divided sheaves}\label{sec:Fdiv}

The results of this subsection will be used only in Appendix \S\ref{sec:Nori ger root stack appendix}. In \S\ref{sec:Nori ger root stack appendix}, the Nori fundamental gerbes of non-smooth tame stacks will be  studied, which is not necessary for the proof of the main theorem. Thus, the reader can access the main theorem without following the results in this subsection.

Let $k$ be a perfect field of characteristic $p>0$. 

\begin{definition}(cf.~\cite[\S6 Definition 6.20]{tz17})
Let $\calX$ be a fibered category in groupoids over $k$. We define the fibered category $\calX^{(\infty)}$ in groupoids over $k$ together with a morphism $\calX\lto\calX^{(\infty)}$ as follows. Consider the sequence of the relative Frobenius morphisms
\begin{equation*}
\calX\lto\calX^{(1)}\lto\calX^{(2)}\lto\cdots\lto\calX^{(i)}\lto\cdots
\end{equation*}
and define $\calX^{(\infty)}$ to be the limit
\begin{equation*}
\calX^{(\infty)}\Def\varinjlim_{i\in\mathbb{N}}\calX^{(i)}
\end{equation*}
and the morphism $\calX\lto\calX^{(\infty)}$ to be the natural morphism. We define the category $\Fdiv(\calX)$ of \textit{Frobenius divided sheaves} on $\calX$, or shortly \textit{$F$-divided sheaves} on $\calX$, by
\begin{equation*}
\Fdiv(\calX)\Def\Vect(\calX^{(\infty)}).
\end{equation*}
\end{definition}

We apply the formalism discussed in the previous section to the morphism of fibered categories
\begin{equation*}
\pi=\pi_{\Fdiv}:\calX\lto\calX_{\Fdiv}=\calX^{(\infty)}.
\end{equation*}

\begin{thm}(dos Santos, Tonini--Zhang, cf.~\cite{ds07}\cite[\S6 Theorem 6.23]{tz17})
\label{thm:Fdiv}
Let $\calX$ be a geometrically connected algebraic stack of finite type over $k$. Then all the Axioms \ref{axio:tann}(A), (B) and (C) are satisfied for $\calX\lto\calX^{(\infty)}$ and $L=\End_{\Fdiv(\calX)}(\unit_{\Fdiv(\calX)})=k$. Therefore, $\Fdiv(\calX)$ is a $k$-tannakian category.
\end{thm}

\begin{rem}
Moreover, the associated tannakian gerbe $\Pi_{\Fdiv(\calX)}$ is banded by a pro-smooth affine group scheme~(cf.~\cite{ds07}\cite{tz17}).
\end{rem}

As a consequence, we get tannakian interpretations of the fundamental gerbes in positive characteristic.

\begin{cor}(Gieseker, dos Santos, Esnault--Hogadi, Tonini--Zhang, cf.~\cite{gi75}\cite{ds07}\cite{eh12}\cite{tz17})\label{cor:Fdiv}
Let $\calX$ be a geometrically connected algebraic stack of finite type over $k$. Then we have the following. 
\begin{enumerate}
\renewcommand{\labelenumi}{(\arabic{enumi})}
\item There exists a natural equivalence of tannakian categories over $k$,
\begin{equation*}
\Vect(\Pi_{\calX/k}^{\et})\simeq\EFin(\Fdiv(\calX)).
\end{equation*} 
\item If moreover $\calX$ is geometrically reduced, then there exists a natural equivalence of tannakian categories over $k$, which extends the one given in (1),
\begin{equation*}
\Vect(\Pi_{\calX/k}^{\N})\simeq\EFin(\Fdiv_{\infty}(\calX)).
\end{equation*} 
\end{enumerate}
\end{cor}

%%%%%%%%%%%%%%%%%%%%%%%%%%%%%%%%%%%%%%%%%%%%%%%%%%%%%%%%%%%%
%%%%%%%%%%%%%%%%%%%%%%%%%%%%%%%%%%%%%%%%%%%%%%%%%%%%%%%%%%%%
%%%%%%%%%%%%%%%%%%%%%%%%%%%%%%%%%%%%%%%%%%%%%%%%%%%%%%%%%%%%

\section{Root stacks and quasi-coherent sheaves of them}\label{sec:RS}

\subsection{Finite linearly reductive group schemes}

In this subsection, we recall the definition of \textit{linearly reductive} group schemes.
First, a finite flat group scheme $G$ over a scheme $S$ is called \textit{diagonalizable} if it is isomorphic to the Cartier dual of a constant abelian group scheme. To a finite abelian group $\Gamma$, one can associate the diagonalizable group scheme, which we denote by $\Diag_S(\Gamma)$, over $S$~(cf.~\cite[Section 2.2]{wa79}). Note that the Cartier dual of $\Diag_S(\Gamma)$ is canonically isomorphic to the constant group scheme associated with $\Gamma$. 
The formulation of $\Diag_S(\Gamma)$ is compatible with any base change, i.e.\ for any morphism $S'\lto S$ of schemes, there exists a canonical isomorphism
\begin{equation*}
\Diag_{S'}(\Gamma)\xrightarrow{~\simeq~}\Diag_S(\Gamma)\times_S S'
\end{equation*}
(cf.~\cite[\S1 1.1.2]{gm71}).

Let $S$ be the spectrum of a strictly henselian local ring or a separably closed field. For an affine flat $S$-group scheme $G$, we will denote by $\X(G)$ the group of \textit{characters} of $G$, namely $\X(G)\Def\Hom_{S\text{-gr}}(G,\G_{m,S})$. If $G$ is diagonalizable, then the character group $\X(G)$ recovers the original group scheme $G$, i.e.\ there exists a canonical isomorphism 
\begin{equation*}
G\xrightarrow{~\simeq~}\Diag_S(\X(G)).
\end{equation*}
In fact, the correspondence $\Gamma\longmapsto\Diag_S(\Gamma)$ gives an anti-equivalence between the category of finite abelian groups and the category of finite diagonalizable $S$-group schemes and a quasi-inverse functor is given by $G\mapsto\X(G)$.

A finite flat group scheme $G$ over a scheme $S$ is said to be \textit{linearly reductive}~(cf.~\cite[\S2]{aov08}) if the functor
\begin{equation*}
\Qcoh(\cB_S G)\lto\Qcoh(S);~\calF\longmapsto\calF^G
\end{equation*}
is exact. If $S=\Spec k$ is the spectrum of a field $k$, then the latter condition can be replaced by the condition that the functor
\begin{equation*}
\Rep(G)\lto\Vecf_k;~V\longmapsto V^G
\end{equation*}
is exact, or equivalently, $\Rep(G)$ is a semisimple category. 
In \cite{aov08}, a classification of finite flat linearly reductive group schemes~(cf.~\cite[Theorem 2.16]{aov08}) is given for a general base scheme $S$. Here, we recall the classification theorem only in a special case. 

\begin{prop}(cf.~\cite[Proposition 2.10]{aov08})\label{prop:classif lin red}
If $S$ is the spectrum of a separably closed field $k$, then a finite flat group scheme $G$ over $S$ is linearly reductive if and only if it admits an exact sequence
\begin{equation*}
1\lto \Delta\lto G\lto H\lto 1
\end{equation*}
such that $H$ is constant and tame and $\Delta$ is a diagonalizable group scheme. Moreover, the extension admits a splitting after an fppf cover of $S$.
Here, a finite \'etale group scheme over $S$ is said to be tame if the degree is prime to the characteristic of $k$.   
\end{prop}

\subsection{Tame stacks}\label{subsec:tame stack}

Let $S$ be a scheme and $\calX$ be an algebraic stack locally of finite presentation over $S$ with finite inertia $\calI_{\calX/S}$. In this case, according to \cite{km97}\cite[Definition 11.1.1]{ol16}, there exists a \textit{coarse moduli space} $\pi:\calX\longrightarrow X$, i.e.\ a morphism into an algebraic space over $S$ satisfying the following two conditions.
\begin{enumerate}
\renewcommand{\labelenumi}{(\roman{enumi})}
\item The induced map $\calX(\xi)\lto X(\xi)$ is bijective for any geometric point $\xi$. 
\item The morphism $\pi$ is universal for maps to algebraic spaces.
\end{enumerate}
The morphism $\pi:\calX\lto X$ is proper, and the map $\scrO_{X}\lto\pi_*\scrO_{\calX}$ is an isomorphism~(cf.~\cite[Theorem 11.1.2(ii)]{ol16}). 

\begin{definition}\label{def:tame}
Under the above situation, $\calX$ is said to be \textit{tame} if the functor $\pi_*:\Qcoh(\calX)\longrightarrow\Qcoh(X)$ is exact.
\end{definition}

\begin{ex}
Let $G$ be a finite flat group scheme of finite presentation over $S$. Then the structure morphism $\pi:\cB_S G\lto S$ is a coarse moduli of the classifying stack $\cB_SG$ and if we identify $\Qcoh(\cB_SG)$ with the category of $G$-equivariant sheaves over $S$, the functor $\pi_*:\Qcoh(\cB_SG)\longrightarrow\Qcoh(S)$ is nothing but the one taking the $G$-invariant part, i.e.\ $\pi_*\calF=\calF^G$. Therefore, $\cB_S G$ is tame if and only if $G$ is a linearly reductive $S$-group scheme.
\end{ex}

\begin{thm}(cf.~\cite[\S3 Theorem 3.2]{aov08})\label{thm:tame}
With the same notation as in Definition \ref{def:tame}, the following are equivalent.
\begin{enumerate}
\renewcommand{\labelenumi}{(\alph{enumi})}
\item $\calX$ is tame.
\item There exists an fppf cover $X'\longrightarrow X$, a linearly reductive group scheme $G$ over $X'$ acting on a finite $X'$-scheme $U$ of finite presentation such that
\begin{equation*}
\calX\times_X X'\simeq[U/G].
\end{equation*} 
\end{enumerate}
\end{thm}

\begin{cor}(cf.~\cite[\S3 Corollaries 3.3, 3.4 and 3.5]{aov08};~\cite[\S7 Proposition 7.4]{alp13})\label{cor:tame}
Let $\calX\longrightarrow S$ be as above and $\pi:\calX\longrightarrow X$ a coarse moduli space. Then we have the following.
\begin{enumerate}
\renewcommand{\labelenumi}{(\arabic{enumi})}
\item Suppose that $\calX$ is tame. For any morphism of $S$-algebraic spaces $X'\longrightarrow X$, the coarse moduli space of $\calX\times_X X'$ is the projection $\calX\times_X X'\longrightarrow X'$.
\item Suppose that $\calX$ is tame. If $\calX$ is flat over $S$, then $X$ is also flat over $S$.
\item If $\calX$ is tame, then for any morphism of schemes $S'\longrightarrow S$, $\calX\times_S S'$ is a tame stack over $S'$.
\item The algebraic stack $\calX$ over $S$ is tame if and only if for any morphism $\Spec k\lto S$ with $k$ algebraically closed, the fiber $\calX\times_S\Spec k$ is tame.
\item The morphism $\pi:\calX\lto X$ is a universal homeomorphism.
\end{enumerate}
\end{cor}

\subsection{Root stacks and quasi-coherent sheaves of them}\label{subsec:root stack}

Let $S=\Spec k$ be the spectrum of a field $k$ and $\calX$ be an algebraic stack over $k$. Recall that there exists a natural equivalence of groupoids
\begin{equation*}
\mathsf{Pic}(\calX)\simeq\Hom_k(\calX,\cB_k\G_m),
\end{equation*} 
where $\mathsf{Pic}(\calX)$ is the groupoid of invertible sheaves over $\calX$. Under this equivalence, each invertible sheaf $\mathcal{L}$ on $\calX$ corresponds to a $\G_m$-torsor $\calP(\mathcal{L})\lto\calX$ defined by 
$\calP(\mathcal{L})\Def\mathbf{Spec}_{\calX}(\bigoplus_{i\in\Z}\mathcal{L}^{\otimes i})$.
Moreover, under the above equivalence, each global section $s\in H^0(\calX,\mathcal{L})$ of an invertible sheaf $\mathcal{L}$ corresponds to a $\G_m$-equivariant function $\calP(\mathcal{L})\longrightarrow\A^1_k$. This amounts to saying that there exists a natural equivalence between the groupoid $\Hom_k(\calX,[\A_k^1/\G_m])$ and the groupoid of all the pairs $(\mathcal{L},s)$ of an invertible sheaf $\mathcal{L}$ on $\calX$ together with a global section $s\in H^0(\calX,\mathcal{L})$, where the action of $\G_m$ on $\A^1_k$ is defined by the natural multiplication. Since
\begin{equation*}
[\A_k^n/\G_m^n]=[\A^1_k/\G_m]\times_k\cdots\times_k[\A_k^1/\G_m]
\end{equation*} 
for each positive integer $n>0$, the groupoid $\Hom_k(\calX,[\A_k^n/\G_m^n])$ is equivalent to the groupoid of families $(\mathbf{L},\mathbf{s})=(\mathcal{L}_i,s_i)_{i=1}^n$ of such pairs $(\mathcal{L}_i,s_i)$. 

\begin{definition}\label{def:root stack}
Suppose given a morphism $\calX\longrightarrow[\A_k^n/\G_m^n]$, which corresponds to a family $(\mathbf{L},\mathbf{s})$, and an $n$-tuple $\bfr=(r_i)_{i=1}^n$ of positive integers $r_i>0$. We define the algebraic stack $\sqrt[\bfr]{(\mathbf{L},\mathbf{s})/\calX}$ as the 2-fiber product
\begin{equation*}
\begin{xy}
\xymatrix{\ar@{}[rd]|{\square}
\sqrt[\bfr]{(\mathbf{L},\mathbf{s})/\calX}\ar[r]\ar[d]&[\A_k^n/\G_m^n]\ar[d]^{\theta_{\bfr}}\\
\calX\ar[r]&[\A_k^n/\G_m^n],
}
\end{xy}
\end{equation*} 
where $\theta_{\bfr}$ is the $\bfr$th power map $\mathbf{a}=(a_i)_{i=1}^n\longmapsto \mathbf{a}^{\bfr}=(a_i^{r_i})_{i=1}^n$. We denote by $\pi_{\bfr}$, or simply $\pi$ if no confusion occurs, the natural map $\sqrt[\bfr]{(\mathbf{L},\mathbf{s})/\calX}\longrightarrow\calX$.
\end{definition}

From the definition, for any two $n$-tuples $\bfr$ and $\bfr'$ with $\bfr\mid\bfr'$, which means $r_i\mid r_i'$ for any $i$, there exists a natural $\calX$-morphism
\begin{equation*}
\sqrt[\bfr']{(\mathbf{L},\mathbf{s})/\calX}\longrightarrow
\sqrt[\bfr]{(\mathbf{L},\mathbf{s})/\calX}
\end{equation*} 
which stems from the endomorphism $\theta_{\bfr'/\bfr}$ of $[\A_k^n/\G_m^n]$, where $\bfr'/\bfr=(r_i'/r_i)_{i=1}^n$. In fact, if we put $\calY=\sqrt[\bfr]{(\mathbf{L},\mathbf{s})/\calX}$, then we have $\sqrt[\bfr'/\bfr]{(\mathbf{L}',\mathbf{s}')/\calY}=\sqrt[\bfr']{(\mathbf{L},\mathbf{s})/\calX}$, where $(\mathbf{L}',\mathbf{s}')$ is the family of pairs which defines the morphism $\calY\longrightarrow[\A_k^n/\G_m^n]$.

Moreover, if $f:\calY\longrightarrow\calX$ is a morphism of algebraic stacks over $k$, then we have
\begin{equation*}
\calY\times_{\calX}\sqrt[\bfr]{(\mathbf{L},\mathbf{s})/\calX}
=\sqrt[\bfr]{(f^*\mathbf{L},f^*\mathbf{s})/\calY}.
\end{equation*}

\begin{ex}\label{ex:root stack}~
\begin{enumerate}
\renewcommand{\labelenumi}{(\arabic{enumi})}
\item Let $X$ be a scheme over $k$. 
If all the $s_i$ are nowhere vanishing sections, or equivalently, if the associated $\G_m$-equivalent morphisms $\mathcal{P}(\mathcal{L}_i)\longrightarrow\A_k^1$ factor through $\G_m$, then the morphism $X\longrightarrow[\A_k^n/\G_m^n]$ factors through $[\G_m^n/\G_m^n]=S$, whence $\pi:\sqrt[\bfr]{(\mathbf{L},\mathbf{s})/X}\xrightarrow{~\simeq~}X$.
\item Let $X=\Spec A$ be an affine scheme. If $\mathcal{L}_i=\scrO_X$ for any $i$, then we have
\begin{equation*}
\sqrt[\bfr]{(\mathbf{L},\mathbf{s})/X}\simeq[\bigl(\Spec A[\mathbf{t}]/(\mathbf{t}^{\bfr}-\bfs)\bigl)/\mu_{\bfr}].
\end{equation*}
\end{enumerate}
\end{ex}

\begin{prop}\label{prop:root stack}
With the above notation, put $\calY=\sqrt[\bfr]{(\mathbf{L},\mathbf{s})/\calX}$. Then we have the following.
\begin{enumerate}
\renewcommand{\labelenumi}{(\arabic{enumi})}
\item (cf.~\cite[Theorem 2.3.2]{cadman}) The diagonal $\Delta:\calY\longrightarrow\calY\times_{\calX}\calY$ is finite.
\item (cf.~\cite[Corollary 2.3.6]{cadman}) The morphism $\pi:\calY\longrightarrow\calX$ is faithfully flat and quasi-compact.
\item (cf.~\cite[Theorem 1.2.31 and Proposition 1.2.32]{ta}) If $\calX=X$ is an algebraic space over $k$, then $\calY$ is a tame stack over $k$ with 
$\pi:\calY\longrightarrow X$ the coarse moduli space of $\calY$.
\end{enumerate}
\end{prop}

\begin{proof}
(1) It suffices to show the map $\theta_{\bfr}:[\A^n_k/\G_m^n]\lto[\A^n_k/\G^n_m]$ has finite diagonal. Since the natural projection $\A^n_k\lto[\A^n_k/\G^n_m]$ is a smooth surjective morphism, we are reduced to the case where $\calX=\A^n_k=\Spec k[x_1,\dots,x_n]$ and $\calY=\sqrt[\bfr]{(\scrO_{\A^n_k},x_i)_{i=1}^n/\A^n_k}=[\A^n_k/\mu_{\bfr}]$. Namely, it suffices to show that the morphism $[\A^n_k/\mu_{\bfr}]\lto\A^n_k$ has finite diagonal. However, since the natural morphism $\A^n_k\lto[\A^n_k/\mu_{\bfr}]$ is a finite flat surjective morphism, we get a Cartesian diagram with two vertical arrows finite flat and surjective,
\begin{equation*}
\begin{xy}
\xymatrix{\ar@{}[rd]|{\square}
\Spec \frac{k[\mathbf{x},\mathbf{y},\mathbf{t}]}{(\mathbf{x}-\mathbf{t}\mathbf{y},\mathbf{t}^{\bfr}-1)}\ar[r]\ar[d]&\Spec \frac{k[\mathbf{x},\mathbf{y}]}{(\mathbf{x}^{\bfr}-\mathbf{y}^{\bfr})}\ar[d]\\
[\A^n_k/\mu_{\bfr}]\ar[r]&[\A^n_k/\mu_{\bfr}]\times_{\A^n_k} [\A^n_k/\mu_{\bfr}].
}
\end{xy}
\end{equation*} 
Since the top horizontal arrow is finite, this implies that the bottom one is also finite. This completes the proof. 

(2) It suffices to show the claims for $\theta_{\bfr}:[\A^n_k/\G_m^n]\longrightarrow[\A^n_k/\G_m^n]$. However, since the diagram
\begin{equation*}
\begin{xy}
\xymatrix{
\A_k^n\ar[r]\ar[d]&[\A_k^n/\G_m^n]\ar[d]^{\theta_{\bfr}}\\
\A_k^n\ar[r]&[\A_k^n/\G_m^n]
}
\end{xy}
\end{equation*}
is commutative, where the two horizontal arrows are the universal $\G^n_m$-torsors and the left vertical arrow is the $\bfr$th power map, it suffices to show the claims for the $\bfr$th power map $\A_k^n\longrightarrow\A_k^n$~(cf.~\cite[Lemma 06FM]{stack}), which are standard facts. 

(3) Since the problem is Zariski local for $X$, we may assume that $X=\Spec A$ is an affine scheme over $k$ and the invertible sheaves $\mathcal{L}_i$ are trivial. Then, thanks to Theorem \ref{thm:tame}, the claim follows from the description given in Example \ref{ex:root stack}(2).
\end{proof}

\begin{prop}(cf.~\cite[Proposition 1.2.35]{ta})\label{prop:root stack qcoh}
With the same notation as in Proposition \ref{prop:root stack}, we have the following.
\begin{enumerate}
\renewcommand{\labelenumi}{(\arabic{enumi})}
\item $\scrO_{\calX}\xrightarrow{~\simeq~}\pi_*\scrO_{\calY}$.
\item For any quasi-coherent sheaves $\calE$ on $\calX$ and $\calF$ on $\calY$,
\begin{equation*}
\pi_*\calF\otimes\calE\xrightarrow{~\simeq~}\pi_*(\calF\otimes\pi^*\calE).
\end{equation*}
\item For any quasi-coherent sheaf $\calE$ on $\calX$,
\begin{equation*}
\calE\xrightarrow{~\simeq~}\pi_*\pi^*\calE.
\end{equation*}
Therefore, the functor $\pi^*:\Qcoh(\calX)\longrightarrow\Qcoh(\calY)$ is fully faithful.
\end{enumerate}
\end{prop}

\begin{proof}
(3) follows from (1) and (2). Therefore, it suffices to show  (1) and (2). As explained in the proof of \cite[Proposition 1.2.35]{ta}, by considering an fppf cover $U\longrightarrow\calX$ from a $k$-scheme $U$, the problems can be reduced to the case where $\calX=X$ is a scheme, whence, by Proposition \ref{prop:root stack}, $\calY$ is a tame stack with $\pi:\calY\longrightarrow X$ the coarse moduli space. Then (1) holds as we recall at the beginning of the previous subsection \S\ref{subsec:tame stack}. Moreover, for the second claim (2), see  \cite[Proposition 4.5]{alp13}.
\end{proof}

We will apply the above arguments to the following specific situation. 

\begin{definition}\label{def:root stack 2}
Let $X$ be a locally Noetherian scheme over $k$. Let $\bfD=(D_i)_{i\in I}$ be a finite family of reduced irreducible distinct effective Cartier divisors $D_i$ on $X$. We set $D\Def\bigcup_{i\in I}D_i\subset X$. Furthermore, suppose given a family $\bfr=(r_i)_{i\in I}$ of integers $r_i>0$. For each $i\in I$, we denote by $s_{D_i}$ the canonical section of $\scrO_X(D_i)$, i.e.\ $s_{D_i}:\scrO_{X}\hookrightarrow\scrO_{X}(D_i)$ is the natural inclusion. Let $\scrO_{X}(\bfD)=(\scrO_X(D_i))_{i\in I}$ and $s_{\bfD}=(s_{D_i})_{i\in I}$. For each $\bfr=(r_i)_{i\in I}$ with $r_i>0$, we define 
\begin{equation*}
\fX^{\bfr}~\Def~\sqrt[\bfr]{\bfD/X}~\Def~\sqrt[\bfr]{(\scrO_{X}(\bfD),s_{\bfD})/X}
\end{equation*}
and call it the \textit{root stack} associated with $X$ and  the data $(\bfD,\bfr)$.
\end{definition}

\begin{prop}(cf.~\cite[\S 2.4.1]{bo09}\cite[Lemmas 3.2 and 3.4]{bb17preprint})\label{prop:root stack D}
With the above notation, we have the following.
\begin{enumerate}
\renewcommand{\labelenumi}{(\arabic{enumi})}
\item For any open affine neighbourhood $U=\Spec A$ together with local equations $s_i=0$ for $D_i$ on $U$ so that $s_i:\scrO_{X}(D_i)|_U\xrightarrow{~\simeq~}\scrO_U$, 
we have an isomorphism
\begin{equation*}
U\times_X\sqrt[\bfr]{\bfD/X}~\simeq~[\bigl(\Spec A[\mathbf{t}]/(\mathbf{t}^{\bfr}-\bfs)\bigl)/\mu_{\bfr}].
\end{equation*}
\item For any closed point $\xi=[x]\in |\sqrt[\bfr]{\bfD/X}|_0=|X|_0$, there exists a unique gerbe $\calG_{\xi}\lto\Spec k(x)$, which we call the \textit{residual gerbe} at $\xi$, such that $\calG_{\xi}$ is a Noetherian algebraic substack of $\sqrt[\bfr]{\bfD/X}$ and the image of $|\calG_{\xi}|_0\lto|\sqrt[\bfr]{\bfD/X}|_0$ is $\xi$. Moreover, for any $\xi=[x]\in |\sqrt[\bfr]{\bfD/X}|_0$, the residual gerbe $\calG_{\xi}$ at $\xi$ is neutral and non-canonically isomorphic to $\cB_{k(x)}\mu_{\bfr_x}$, where the index $\bfr_x=(r_{x,i})_{i\in I}$ is defined to be
\begin{equation}\label{eq:bfr_x}
r_{x,i}=
\begin{cases}
1&\text{if $x\not\in D_i$,}\\
r_i&\text{if $x\in D_i$}.
\end{cases}
\end{equation}
\end{enumerate}
\end{prop}

\begin{proof}
(1) follows from the description given in Example \ref{ex:root stack}(2) together with the fact that
\begin{equation*}
(\scrO_U(D_i),{s_{D_i}})\xrightarrow[~\simeq~]{~s_i~}(\scrO_{U},s_i)
\end{equation*}
for any $i\in I$.

(2) The uniqueness of such a gerbe follows from \cite[Lemma 06MT]{stack}. Let us prove the existence. The problem is Zariski local, by (1), we can replace $\sqrt[\bfr]{\bfD/X}$ by the quotient stack $\calX\Def [\bigl(\Spec A[\mathbf{t}]/(\mathbf{t}^{\bfr}-\mathbf{s})\bigl)/\mu_{\bfr}]$. For any $x\in\Spec A$, we have
\begin{equation*}
\begin{aligned}
\calX\times_{X}x
&=[\bigl(\Spec k(x)[\mathbf{t}]/(\mathbf{t}^{\bfr}-\mathbf{s}(x))\bigl)/\mu_{\bfr}]\\
&\simeq\prod_{s_i(x)=0}[(\Spec k(x)[t_i]/(t_i^{r_i}))/\mu_{r_i}],
\end{aligned}
\end{equation*}  
hence
\begin{equation*}
(\calX\times_{X}x)_{\rm red}=\prod_{s_i(x)=0}[\Spec k(x)/\mu_{r_i}]\simeq\cB_{k(x)}\mu_{\bfr_x}
\end{equation*}
gives a desired gerbe $\calG_{\xi}$ above $x$. 
\end{proof}

The following is a variant of a theorem due to Alper~\cite[\S10 Theorem 10.3]{alp13}.

\begin{prop}\label{prop:rel alper}
With the same notation as in Definition \ref{def:root stack 2}, let $\bfr'$ be another index with $\bfr\mid\bfr'$. Then we have the following.
\begin{enumerate}
\renewcommand{\labelenumi}{(\arabic{enumi})}
\item For any closed point $\xi$ of $\fX^{\bfr'}$, the induced morphism $\pi_{\xi}:\calG'_{\xi}\longrightarrow\calG_{\xi}$ between the residual gerbes is a gerbe. More precisely, each closed point $\xi=[x]\in |\fX^{\bfr'}|_0=|X|_0$ gives rise to the 2-Cartesian diagram
\begin{equation*}
\begin{xy}
\xymatrix{\ar@{}[rd]|{\square}
\cB_{k(x)}(\mu_{\bfr'_x/\bfr_x})\ar[r]\ar[d]&\calG'_{\xi}\ar[d]\\
\Spec k(x)\ar[r]&\calG_{\xi}. 
}
\end{xy}
\end{equation*}
\item Let $\calE$ be a vector bundle on $\fX^{\bfr'}$. Suppose that for any closed point $\xi$ of $\fX^{\bfr'}$, we have $\pi_{\xi}^*\pi_{\xi*}(\calE|_{\calG'_{\xi}})\xrightarrow{~\simeq~}\calE|_{\calG'_{\xi}}$. Then $\pi^*\pi_*\calE\xrightarrow{~\simeq~}\calE$ and $\pi_*\calE$ is a vector bundle on $\fX^{\bfr}$.
\item The functor $\pi^*:\Vect(\fX^{\bfr})\longrightarrow\Vect(\fX^{\bfr'})$ is fully faithful and the essential image consists of all the vector bundles $\calE$ on $\fX^{\bfr'}$ such that for any closed point $\xi$ of $\fX^{\bfr'}$, we have $\pi_{\xi}^*\pi_{\xi*}(\calE|_{\calG'_{\xi}})\xrightarrow{~\simeq~}\calE|_{\calG'_{\xi}}$. Moreover, the essential image of the functor $\pi^*:\Vect(\fX^{\bfr})\lto\Vect(\fX^{\bfr'})$ is closed under taking subquotients. 
\end{enumerate}
\end{prop}

\begin{proof}
(1) By taking an fppf covering $T\lto\calX$ with $T$ a scheme, the assertion can be deduced from Proposition \ref{prop:root stack D}(2) (cf.~\cite[Lemma 06QF]{stack}). 

(2) Take fppf coverings $f:T\longrightarrow\fX^{\bfr}$ and $f':T'\lto\fX^{\bfr'}$ with $T$ and $T'$ schemes and consider the commutative diagram
{\small\begin{equation*}
\begin{xy}
\xymatrix{\ar@{}[rd]|{\square}
T''\ar[r]^{~g'}\ar[d]_{f''}&T'\ar[d]^{f'}\\
\fX^{\bfr'}\times_{\fX^{\bfr}}T\ar[r]^{~~g}\ar[d]_{\pi'}\ar@{}[rd]|{\square} &\fX^{\bfr'}\ar[d]^{\pi}\\
T\ar[r]_{~~f}&\fX^{\bfr}.
}
\end{xy}
\end{equation*}}
where all the squares are 2-Cartesian. Note that $f'':T''\lto\fX^{\bfr'}\times_{\fX^{\bfr}}T$ is an fppf covering  with $T''$ a scheme. Then, for any morphism $\phi:\calE_1\lto\calE_2$ of vector bundles on $\fX^{\bfr'}$, 
$\phi$ is an isomorphism if and only if 
$g'^*f'^*\phi=f''^*g^*\phi$ is an isomorphism, and the latter condition is equivalent to the condition that   
$g^*\phi$ is an isomorphism.

Then the claims are equivalent to saying that $g^*\pi^*\pi_*\calE\xrightarrow{~\simeq~}g^*\calE$ and $f^*\pi_*\calE$ is a vector bundle on $T$.  However, since $f:T\longrightarrow\fX^{\bfr}$ is flat, by flat base changing, we get
\begin{equation*}
g^*\pi^*\pi_*\calE\simeq\pi'^*f^*\pi_*\calE\simeq\pi'^*\pi'_*g^*\calE
\end{equation*}
and similarly, $f^*\pi_*\calE\simeq \pi'_* g^*\calE$. 
Hence, the problem is reduced to the absolute case $\pi:\fX^{\bfr'/\bfr}\longrightarrow X$, which follows from \cite[\S10 Theorem 10.3]{alp13}.

(3) The description of the essential image is immediate from (2). The full faithfulness follows from Proposition \ref{prop:root stack qcoh}(3). The last assertion is a consequence of (1). 
\end{proof}

\begin{prop}(cf.~\cite[Proposition 3.9(c)]{bergh-etal})\label{prop:sm root stack}
With the same notation as in Definition \ref{def:root stack 2}, suppose further that $X$ is smooth over a perfect field $k$, that each $D_i$ is smooth over $k$, and that $D$ is a simple normal crossings divisor on $X$. Then the root stack $\sqrt[\bfr]{\bfD/X}$ is a smooth algebraic stack over $k$. 
\end{prop}

\begin{proof}
Since the problem is Zariski local, we may assume that
\begin{equation*}
\fX\Def\sqrt[\bfr]{\bfD/X}\simeq [Z/\mu_{\bfr}]
\end{equation*} 
for some $Z=\Spec A[\mathbf{t}]/(\mathbf{t}^{\bfr}-\mathbf{s})$ with $X=\Spec A$~(cf.~Proposition \ref{prop:root stack D}(1)), where $s_1,\dots,s_n$ is a regular sequence in $A$. Let us begin with showing that $Z$ is smooth over $k$. Let $p:Z\lto\fX\xrightarrow{~\pi~} X$ be the composition of the quotient map with the coarse moduli space map $\pi$. Let $\bfD'\Def p^{-1}(\bfD)$. Then note that $\bfD'=(V(t_i))_{i=1}^n$, where $V(t_i)=\{t_i=0\}$, and $t_1,\dots,t_n$ is a regular sequence. By \cite[Lemma 0BIA]{stack}, the scheme theoretic intersection $\cap_{i=1}^n D_i$ is a regular scheme. However, since  
\begin{equation*}
A[\mathbf{t}]/(\mathbf{t}^{\bfr}-{\bfs}, t_1,\dots,t_n)=A/(s_1,\dots,s_n)=\scrO_{\cap_{i=1}^nD_i},
\end{equation*}  
by \cite[Lemma 00NU]{stack}, we can find that $Z$ is regular over $k$. However, as $k$ is perfect, by \cite[Lemma 0CBP]{stack}, the same is true after base change any algebraic extension of $k$. Namely $Z$ is a geometrically regular and hence by \cite[Lemma 038X]{stack}, $Z$ is a smooth $k$-scheme. 

Let $W\lto \fX$ be a smooth atlas and put
\begin{equation*}
W'\Def W\times_{\fX}Z.
\end{equation*}
Since the projection $W'\lto Z$ is smooth and $Z$ is a smooth $k$-scheme, $W'$ is smooth over $k$. 
Then we get the commutative diagram
\begin{equation*}
\begin{xy}
\xymatrix{
W'\ar[r]\ar[rd]&W\ar[d]\\
& \Spec k.
}
\end{xy}
\end{equation*}
Since $W'\lto W$ is an fppf covering, by \cite[Lemma 05B5]{stack}, $W$ must be a smooth $k$-scheme. This completes the proof.
\end{proof}

%%%%%%%%%%%%%%%%%%%%%%%%%%%%%%%%%%%%%%%%%%%%%%%%%%%%%%%%%
%%%%%%%%%%%%%%%%%%%%%%%%%%%%%%%%%%%%%%%%%%%%%%%%%%%%%%%%%

\section{Embedding problem over root stacks}\label{sec:EP}

\subsection{Local fundamental group schemes}\label{sec:loc FG}

Let $k$ be a field of characteristic $p>0$ and $\calX$ a reduced algebraic stack over $k$ such that $H^0(\scrO_{\calX})$ contains no nontrivial purely inseparable extension of $k$, which admits the local fundamental gerbe $\calX\lto\Pi_{\calX/k}^{\loc}$~(cf. Definition \ref{def:et loc ger} and Proposition \ref{prop:exist et loc ger}(2)). By applying the formalism in \S\ref{subsec:tann}, to the structure morphism $\calX\lto\Spec k=\calX_{\calT}$, we get a tannakian interpretation of the local fundamental gerbe $\Pi^{\loc}_{\calX/k}$. With the same notation as in \S\ref{subsec:tann}, we consider the categories
\begin{equation*}
\calD_i(\calX)\Def\calT_i(\calX)
\end{equation*}
for $i\ge 0$ with $\calX_{\calT}=\Spec k$ and the functors
\begin{equation*}
\calD_i(\calX)\lto\calD_{i+1}(\calX)~;~(\calF,V,\lambda)\longmapsto(F^{*}\calF,F_k^{*}V,F^*\lambda).
\end{equation*}
Recall that $\calD_{\infty}(\calX)=\varinjlim_{i}\calD_i(\calX)$.

\begin{thm}(Tonini--Zhang, cf.~\cite[\S7 Theorem 7.1]{tz17})
\label{thm:tann loc ger}
Let $k$ be a field of characteristic $p>0$ and $\calX$ a reduced algebraic stack over $k$ such that $H^0(\scrO_{\calX})$ contains no nontrivial purely inseparable extension of $k$. Then, all the categories $\calD_{i}(\calX)~(i\ge 0)$ and $\calD_{\infty}(\calX)$ are tannakian categories over $k$ and the above functors $\calD_{i}(\calX)\lto\calD_{i+1}(\calX)$ are $k$-linear monoidal exact functors. Moreover, 
there exists a canonical equivalence of tannakian categories
\begin{equation*}
\calD_{\infty}(\calX)\xrightarrow{~\simeq~}\Vect(\Pi_{\calX/k}^{\loc}).
\end{equation*}
\end{thm}

\begin{proof}
All the statements except for the last equivalence are immediate consequences of Theorem \ref{thm:tann Nori ger}(1) and (2). However, by Theorem \ref{thm:tann Nori ger}(3), we have $\calD_{\infty}(\calX)=\EFin(\calD_{\infty}(\calX))$ and 
the morphism $\calX\lto\Pi^{\loc}_{\calD_{\infty}(\calX)}=\Pi_{\calD_{\infty}(\calX)}$ is the local fundamental gerbe over $k$. This completes the proof.  
\end{proof}

Furthermore, if $k$ is perfect, then the local fundamental gerbe $\Pi^{\loc}_{\calX/k}$ is neutral.

\begin{prop}(Romagny--Tonini--Zhang, cf.~\cite[\S2]{rtz17})
\label{prop:loc ger neutral}
Let $\Gamma$ be a pro-local gerbe over a perfect field of characteristic $p>0$. Then $\Gamma(k)\neq\emptyset$ and, for any two objects $\xi,\xi'\in\Gamma(k)$, there exists exactly one isomorphism $\xi\xrightarrow{~\simeq~}\xi'$. In other words, the tannakian category $\Vect(\Gamma)$ has a neutral fiber functor which is unique up to unique isomorphism.  
\end{prop}

\begin{proof}
Since $\Vect(\Gamma)\simeq\Vect(\Pi^{\loc}_{\Gamma/k})$, we may assume that $\Gamma=\Pi^{\loc}_{\calX/k}$ for some reduced algebraic stack $\calX$ over $k$. Then by Theorem \ref{thm:tann loc ger}, it suffices to show the claim for $\calD_{\infty}(\calX)$. Indeed, the functors $\calD_i(\calX)\lto\Vecf_k~;~(\calF,V,\lambda)\longmapsto F_k^{-i*}V$ define neutral fiber functors compatible with transition functors $\calD_i(\calX)\lto\calD_{i+1}(\calX)$, whence we get a neutral fiber functor of $\calD_{\infty}(\calX)$. Suppose given two neutral fiber functors $\omega$ and $\omega'$ of $\calD_{\infty}(\calX)$. Then the sheaf of isomorphisms $P\Def\underline{\mathrm{Isom}}_k^{\otimes}(\omega,\omega')$ is a torsor over $k$ under a pro-local $k$-group scheme, which admits a unique $k$-rational point $\Spec k=P_{\rm red}\lto P$. Therefore, there exists a unique isomorphism $\omega\xrightarrow{~\simeq~}\omega'$. This completes the proof.
\end{proof}

\begin{cor}(\"Unver, Zhang, Romagny--Tonini--Zhang, cf.~\cite{un10}\cite{zh16}\cite[\S2]{rtz17})\label{cor:loc ger neutral}
Let $\calX$ be a reduced algebraic stack over a perfect field $k$ with $\calX(k)\neq\emptyset$. Then the local fundamental gerbe $\pi^{\loc}(\calX,x)$ is independent of the choice of the $k$-rational point $x$. More precisely, for any two $k$-rational points $x,x'\in\calX(k)$, there exists exactly one isomorphism of pro-local $k$-group schemes $\pi^{\loc}(\calX,x)\xrightarrow{~\simeq~}\pi^{\loc}(\calX,x')$. Moreover, for any $k$-rational point $x\in\calX(k)$, there exist equivalences of categories
\begin{equation*}
\begin{xy}
\xymatrix{
\Hom_k(\pi^{\loc}(\calX,x),G)\ar[r]^{~\simeq}&\Hom(\Pi^{\loc}_{\calX/k},\cB G)\ar[r]^{~~\simeq}&H^1_{\fppf}(\calX,G)
}
\end{xy}
\end{equation*}
for any finite local $k$-group scheme $G$.
\end{cor}

Hence, from now on, we write $\pi^{\loc}(\calX)$ instead of $\pi^{\loc}(\calX,x)$ for the local fundamental group scheme of $(\calX,x)$. The local fundamental group scheme $\pi^{\loc}(\calX)$ has a more explicit description as follows.
Let  $\mathsf{L}(X)$ be the category of pairs $(\mathcal{P},G)$ where $G$ is a finite local $k$-group scheme and $\calP\lto \calX$ is a $G$-torsor. Then the category $\mathsf{L}(\calX)$ is cofiltered and the projective limit
\begin{equation*}
\varprojlim_{(\calP,G)\in \mathsf{L}(\calX)}(\calP,G)
\end{equation*}
exists and the underlying group scheme of the limit is canonically isomorphic to $\pi^{\loc}(\calX)$.

Let $\calX$ be a reduced algebraic stack of finite type over an algebraically closed field $k$ of characteristic $p>0$. Let $\pi^{\loc}(\calX)$ be the local fundamental group scheme of $\calX$. 
As $k$ is algebraically closed of characteristic $p>0$, a finite local group scheme $G$ over $k$ is linearly reductive if and only if it is diagonalizable, i.e.\ $G\simeq\Diag(A)$ for some abelian $p$-group $A$~(cf.~Proposition  \ref{prop:classif lin red}). Therefore, the isomorphism class of the maximal linearly reductive quotient of $\pi^{\loc}(\calX)$ is determined completely by the $p$-primary torsion subgroup of the group $\X(\pi^{\loc}(\calX))=\Hom(\pi^{\loc}(\calX),\G_m)$ of characters of the local fundamental group scheme $\pi^{\loc}(\calX)$, which can be calculated as follows
\begin{equation*}
\X(\pi^{\loc}(\calX))[p^{\infty}]=\varinjlim_{m>0}\Hom_k(\pi^{\loc}(\calX),\mu_{p^m})\xrightarrow{~\simeq~}\varinjlim_{m>0}H^1_{\fppf}(\calX,\mu_{p^m}).
\end{equation*}

Hence, we have seen the following.

\begin{prop}\label{prop:max loc lr}
If $k$ is an algebraically closed field of characteristic $p>0$ and $\calX$ is a reduced algebraic stack of finite type over $k$, then there exists a canonical isomorphism:
\begin{equation*}
\pi^{\loc}(\calX)^{\lr}~\simeq~\varprojlim_{m>0}\Diag(H_{\fppf}^1(\calX,\mu_{p^m})).
\end{equation*}
\end{prop}

\begin{cor}\label{cor:max loc lr}
Let $X$ be a proper smooth connected curve over an algebraically closed field of characteristic $p>0$ with $p$-rank $\gamma$ and $\emptyset\neq U\subseteq X$ a nonempty open subscheme of $X$ with $n=\#(X\setminus U)\ge 0$. Then we have the following.
\begin{equation*}
\pi^{\loc}(U)^{\lr}~\simeq~
\begin{cases}
\Diag((\Q_p/\Z_p)^{\oplus\gamma})&\text{if $n=0$,}\\
\Diag((\Q_p/\Z_p)^{\oplus\gamma+n-1})&\text{if $n>0$}.
\end{cases}
\end{equation*}
\end{cor}

\begin{proof}
If $n=0$, this is standard. In the case where $n>0$, see \cite[Proposition 3.2]{ot17}.
\end{proof}

\begin{rem}\label{rem:loc nori red}
Let $\calX$ be a reduced algebraic stack over an algebraically closed field $k$ of characteristic $p>0$ and $\pi^{\loc}(\calX)$ the local fundamental group scheme of $\calX/k$. We can identify the local fundamental gerbe $\Pi^{\loc}_{\calX/k}$ with the classifying stack $\cB_k\pi^{\loc}(\calX)$~(cf.~Corollary \ref{cor:loc ger neutral}). Suppose given a $k$-homomorphism $\pi^{\loc}(\calX)\lto G$ to a finite local $k$-group scheme $G$ and the corresponding $G$-torsor $\calP\lto\calX$. Then, under the assumption that $k$ is an algebraically closed field, the following conditions are equivalent.
\begin{enumerate}
\renewcommand{\labelenumi}{(\alph{enumi})}
\item The homomorphism $\pi^{\loc}(\calX)\lto G$ is surjective.
\item The corresponding morphism $\Pi^{\loc}_{\calX/k}=\cB_k\pi^{\loc}(\calX)\lto\cB_kG$ is Nori-reduced.
\item There exists no strict closed subgroup scheme $H\subsetneq G$ such that the $G$-torsor $\calP\lto\calX$ is reduced to an $H$-torsor $\mathcal{Q}\lto\calX$, i.e.\ $\calP\simeq{\rm Ind}_H^G(\mathcal{Q})$.
\item If we denote by $\calE_{\infty}(\calP)$ the image of the regular representation $(k[G],\rho_{\rm reg})$ of $G$ in $\calD_{\infty}(\calX)$ under the restriction functor $\Rep(G)\lto\Rep(\pi^{\loc}(\calX))\simeq\calD_{\infty}(\calX)$~(cf.\ Theorem \ref{thm:tann loc ger}), then
\begin{equation*}
\Dim_k\Hom_{\calD_{\infty}(\calX)}(\unit,\calE_{\infty}(\calP))=1.
\end{equation*}
\end{enumerate}
Indeed, the equivalence between (a) and (b) is a consequence of Proposition \ref{prop:loc ger neutral}. The equivalence between (a) and (c) is immediate from Corollary \ref{cor:loc ger neutral}. The implication (a)$\Longrightarrow$(d) follows from the fact that $\Hom_{\Rep(G)}(\unit,(k[G],\rho_{\rm reg}))=k[G]^G=k$. Conversely, let us suppose that the condition (a) is not satisfied. If $H$ denotes the image of the given homomorphism $\pi^{\loc}(\calX)\lto G$, then $H$ is a closed subgroup scheme of $G$ with $H\neq G$ and the restriction map $\Rep(G)\lto\Rep(\pi^{\loc}(\calX))$ factors through the category $\Rep(H)$, which implies that 
\begin{equation*}
k[G]^H=\Hom_{\Rep(H)}(\unit,(k[G],\rho_{\rm reg}))\subseteq\Hom_{\calD_{\infty}(\calX)}(\unit,\calE_{\infty}(\calP)). 
\end{equation*}      
As $H\neq G$, we have $\Dim k[G]^H\neq 1$ and hence the condition (d) does not hold. This proves the implication (d)$\Longrightarrow$(a). 
\end{rem}

\subsection{Nori fundamental gerbes of root stacks in positive characteristic}\label{subsec:Nori ger root stack}

We shall use the same notation as in \S\ref{subsec:root stack}. Let $X$ be a geometrically connected and geometrically reduced scheme of finite type over the spectrum $S=\Spec k$ of a perfect field $k$ of characteristic $p>0$. Let $\bfD=(D_i)_{i\in I}$ be a finite family of reduced irreducible effective Cartier divisors on $X$ and put  $D=\cup_{i\in I}D_i\subset X$. For each $\bfr=(r_i)_{i\in I}$ with $r_i>0$, as in \S\ref{subsec:root stack}, we put
\begin{equation*}
\fX^{\bfr}=\sqrt[\bfr]{\bfD/X}.
\end{equation*} 
In this subsection, motivated by \cite{bb17}, the Nori fundamental gerbes associated with root stacks are studied, dropping the properness assumption on $X$. However, under the smoothness assumption as put in Proposition \ref{prop:sm root stack}, a more direct approach can be applied, which is enough to prove the main theorem.

\begin{prop}\label{prop:Nori ger sm root stack}
Under the same assumption as in Proposition \ref{prop:sm root stack}, if we put $U=X\setminus D$, then the induced morphisms
\begin{equation*}
\Pi^{\N}_U\lto\Pi^{\N}_{\fX^{\bfr}}
\end{equation*}
are gerbes.
\end{prop}

\begin{proof}
First note that, for any finite gerbe $\Gamma$ over $k$, a morphism $\Pi^{\N}_{\fX^{\bfr}}\lto\Gamma$ into $\Gamma$ is a gerbe if and only if the composition $\fX^{\bfr}\lto\Pi^{\N}_{\fX^{\bfr}}\lto\Gamma$ is Nori-reduced~(cf.\ Remark \ref{rem:nori-reduced vs gerbe}). Therefore, it suffices to show that, for any Nori-reduced morphism $\fX^{\bfr}\lto\Gamma$ into a finite gerbe $\Gamma$, the composition $U\lto\fX^{\bfr}\lto\Gamma$ is still Nori-reduced. As $k$ is assumed to be perfect, without loss of generality, we may assume that $k$ is an algebraically closed field~(cf.~\cite[\S6]{bv15}) and hence may assume that $\Gamma=\cB G$ for some finite $k$-group scheme $G$. If the morphism $U\lto\fX^{\bfr}\lto\Gamma$ is not Nori-reduced, then it factors through $\cB H$ where $H$ is a strict subgroup scheme of $G$. One can show that the resulting morphism $U\lto\cB H$ is extended to a morphism $\fX^{\bfr}\lto\cB H$. Indeed, the problem is Zariski local~(cf.~Proposition \ref{prop:nori reconst}), we may assume that $\fX^{\bfr}=[Z/\mu_{\bfr}]$ with $Z=\Spec \scrO_X[\mathbf{t}]/(\mathbf{t}^{\bfr}-\mathbf{s})$. As we saw in the proof of Proposition \ref{prop:sm root stack}, $Z$ is smooth over $k$. Therefore, by applying \cite[Chapter II Proposition 6]{no82} to the open immersion $Z\times_{\fX^{\bfr}} U\hookrightarrow Z$, we get an extension $Z\lto\cB H$ of the morphism $Z\times_{\fX^{\bfr}}U\lto U\lto\cB H$. 
\begin{equation*}
\begin{xy}
\xymatrix{
Z\times_{\fX^{\bfr}} U\ar[r]\ar[d]&Z\ar[d]\ar@/^2mm/@{-->}[ldd]\\
U\ar[r]\ar[d]&\fX^{\bfr}\ar[d]\\
\cB H\ar[r]&\cB G
}
\end{xy}
\end{equation*}
By dividing by the action of $\mu_{\bfr}$, we get a desired extension $\fX^{\bfr}=[Z/\mu_{\bfr}]\lto\cB H$. However, as $H\neq G$, this contradicts with the Nori-reducedness of the morphism $\fX^{\bfr}\lto\cB G$. Therefore, the morphism $U\lto\fX^{\bfr}\lto\cB G$ is Nori-reduced.
\end{proof}

\begin{rem}\label{rem:nori-reduced vs gerbe}
Let $\calX$ be an inflexible fibered category over a field $k$ with $\calX\lto\Pi^{\N}_{\calX}$ its Nori fundamental gerbe. Let $\phi:\Pi^{\N}_{\calX}\lto\Gamma$ be a morphism into a finite gerbe $\Gamma$ over $k$. Then $\phi$ is a gerbe if and only if the composition $\calX\lto\Pi^{\N}_{\calX}\xrightarrow{~\phi~}\Gamma$ is Nori-reduced. Indeed, as $\Pi^{\N}_{\calX}$ itself is inflexible~(cf.\ \cite[Proposition 5.4]{bv15}), by the universal property of the Nori fundamental gerbe $\calX\lto\Pi^{\N}_{\calX}$, we may assume that $\calX=\Pi^{\N}_{\calX}$, hence may assume that $\calX$ is pseudo-proper~(cf.\ \cite[Example 7.2(b)]{bv15}), in which case the equivalence is already remarked in \cite[Remark 1.14]{abetz}. 
\end{rem}

\begin{prop}\label{prop:Nori ger root stack}
Let $\bfr$ and $\bfr'$ be two indices with $\bfr\mid\bfr'$. The natural morphism $\Pi^{\N}_{\fX^{\bfr'}/k}\to\Pi^{\N}_{\fX^{\bfr}/k}$ is a gerbe. Moreover, for any finite $k$-group scheme $G$ and any $G$-torsor $\calY'\longrightarrow\fX^{\bfr'}$, there exists a $G$-torsor $\calY\longrightarrow\fX^{\bfr}$ such that $\calY'\simeq\calY\times_{\fX^{\bfr}}\fX^{\bfr'}$ if and only if for any closed point $\xi$ of $\fX^{\bfr'}$, the composition
\begin{equation*}
\calG'_{\xi}\longrightarrow\fX^{\bfr'}\xrightarrow{~\calY'~}\cB_kG
\end{equation*} 
factors through the gerbe $\calG'_{\xi}\longrightarrow\calG_{\xi}$~(cf.~Proposition \ref{prop:rel alper}(1)), where $\calG_{\xi}$ (respectively $\calG'_{\xi}$) denotes the residual gerbe of $\fX^{\bfr}$ (respectively of $\fX^{\bfr'}$) at $\xi$
\end{prop}

\begin{proof}
First let us show the second assertion.  Thanks to Proposition \ref{prop:nori reconst}, we have a commutative diagram where horizontal arrows are equivalence of categories,
\begin{equation*}
\begin{xy}
\xymatrix{
\Hom_k(\fX^{\bfr'},\cB G)\ar[r]^{\simeq~~~~~~~~~~}&\Hom_{k,\otimes}(\Vect(\cB G),\Vect(\fX^{\bfr'}))\\
\Hom_k(\fX^{\bfr},\cB G)\ar[r]^{\simeq~~~~~~~~~~}\ar[u]&\Hom_{k,\otimes}(\Vect(\cB G),\Vect(\fX^{\bfr})).\ar[u] 
}
\end{xy}
\end{equation*}
Therefore, by applying Proposition \ref{prop:rel alper}(3), we get the second assertion. 

We prove the first assertion of the proposition under the same assumption as in Proposition \ref{prop:sm root stack}. In the general case, see Appendix \S\ref{sec:Nori ger root stack appendix}. Under the assumption in Proposition \ref{prop:sm root stack}, for any indices $\bfr\mid\bfr'$,  the commutative diagram
\begin{equation*}
\begin{xy}
\xymatrix{
\Pi^{\N}_U\ar[r]\ar[rd]&\Pi^{\N}_{\fX^{\bfr'}}\ar[d]\\
&\Pi^{\N}_{\fX^{\bfr}}
}
\end{xy}
\end{equation*}
induces the restriction functors
\begin{equation*}
\Vect(\Pi^{\N}_{\fX^{\bfr}})\xrightarrow{~u~}\Vect(\Pi^{\N}_{\fX^{\bfr'}})\xrightarrow{~v~}\Vect(\Pi^{\N}_U),
\end{equation*}
where $v$ and $v\circ u$ are fully faithful by Proposition \ref{prop:Nori ger sm root stack}. Thus, it follows that the restriction functor $u:\Vect(\Pi^{\N}_{\fX^{\bfr}})\lto\Vect(\Pi^{\N}_{\fX^{\bfr'}})$ is fully faithful as well.   
Now let us show that $\Pi^{\N}_{\fX^{\bfr'}}\lto\Pi^{\N}_{\fX^{\bfr}}$ is a gerbe. If we write $\Pi^{\N}_{\fX^{\bfr}/k}=\varprojlim_{i}\Gamma_i$, where $\fX^{\bfr}\lto\Gamma_i$ are Nori-reduced with $\Gamma_i$ finite gerbes~(cf.\ \cite[Proof of Theorem 5.7]{bv15}), then all the projections $\Pi^{\N}_{\fX^{\bfr}/k}\lto\Gamma_i$ are gerbes~(cf.\ Remark \ref{rem:nori-reduced vs gerbe}). We have to show that the composition $\Pi^{\N}_{\fX^{\bfr'}/k}\lto\Pi^{\N}_{\fX^{\bfr}/k}\lto\Gamma_i$ is still a gerbe for each $i$. However, as $\Gamma_i$ is finite, the map $\Pi^{\N}_{\fX^{\bfr'}}\lto\Gamma_i$ is a gerbe if and only if the restriction functor $\Vect(\Gamma_i)\lto\Vect(\Pi^{\N}_{\fX^{\bfr'}})$ is fully faithful~(cf.\ \cite[Remark B.7]{tz17}). The latter condition is fulfilled because the restriction functor is the composition of fully faithful functors $\Vect(\Gamma_i)\lto\Vect(\Pi^{\N}_{\fX^{\bfr}})\xrightarrow{~u~}\Vect(\Pi^{\N}_{\fX^{\bfr'}})$. This completes the proof. 
\end{proof}

\begin{rem}\label{rem:Nori ger root stack}
Proposition \ref{prop:Nori ger root stack} particularly implies that if $\bfr\mid\bfr'$, then the natural map between the local fundamental group schemes is surjective, i.e.\ 
$\pi^{\loc}(\fX^{\bfr'})\twoheadrightarrow\pi^{\loc}(\fX^{\bfr})$. One can prove this fact without the smoothness assumption. Indeed, for any surjective homomorphism $\pi^{\loc}(\fX^{\bfr})\twoheadrightarrow G$ onto a finite local $k$-group scheme $G$, let $H\subseteq G$ be the image of the composition of the homomorphisms $\pi^{\loc}(\fX^{\bfr'})\lto\pi^{\loc}(\fX^{\bfr})\twoheadrightarrow G$. Then we obtain a commutative diagram
\begin{equation*}
\begin{xy}
\xymatrix{
H^1_{\fppf}(\fX^{\bfr'},H)\ar@{^{(}->}[r]&H^1_{\fppf}(\fX^{\bfr'},G)\\
H^1_{\fppf}(\fX^{\bfr},H)\ar@{^{(}->}[r]\ar[u]&H^1_{\fppf}(\fX^{\bfr},G),\ar[u] 
}
\end{xy}
\end{equation*}
where the injectivity of the horizontal maps is valid because the root stacks $\fX^{\bfr}$ and $\fX^{\bfr'}$ are reduced and the reduced subscheme of the quotient space $G/H$ is trivial, i.e.\ $(G/H)_{\rm red}=\Spec k$. Then one can deduce from Propositions \ref{prop:nori reconst} and \ref{prop:rel alper} together with the identification $H^1_{\fppf}(-,G)=\Hom_k(-,\cB G)$ that the above diagram is Cartesian. This implies that $H=G$.
\end{rem}

\subsection{Torsors over root stacks and tamely ramified torsors}

We will continue to use the same notation as in the previous subsection.

\begin{definition}\label{def:tame cover}
Let $X$ be a scheme and $D$ an effective Cartier divisor of $X$. Put $U=X\setminus D$. A finite flat cover $Y\lto X$ which is \'etale over $U$ is said to be \textit{tamely ramified} along $D$ if for any point $x\in D$, all the connected components of $Y\times_X\Spec\scrO_{X,x}^{\rm sh}$ are of the form $\Spec\scrO_{X,x}^{\rm sh}[T]/(T^m-a)$ where $m$ is a positive integer which is prime to the characteristic of $k(x)$ and $a$ is a local equation of $D$ in $\Spec\scrO_{X,x}^{\rm sh}$, where $\scrO_{X,x}^{\rm sh}$ is the strict hensenlization of the local ring $\scrO_{X,x}$ at $x$. 
\end{definition}

\begin{lem}(Abhyankar, cf.~\cite[Theorem 2.3.2 and Corollary 2.3.4]{gm71})\label{lem:abhyankar}
Let $X$ be a normal scheme and $D\subset X$ a simple normal crossings divisor. Let $f:Y\lto X$ be a finite flat cover which is \'etale over $U=X\setminus D$. Then the following are equivalent.
\begin{enumerate}
\renewcommand{\labelenumi}{(\alph{enumi})}
\item $Y$ is normal and $f$ is tamely ramified above the generic points of $D$.
\item $f:Y\lto X$ is tamely ramified along $D$. 
\end{enumerate}
\end{lem}

\begin{definition}\label{def:Kummer mor}
Suppose given a locally Noetherian $k$-scheme $X$ and an $n$-tuple $\mathbf{s}=(s_i)_{i=1}^n$ of regular functions $s_i\in H^0(X,\scrO_X)$. A \textit{Kummer morphism} associated with the data $(\mathbf{s},\bfr)$ is a finite flat $k$-morphism $Y\lto X$ defined by
\begin{equation*}
\scrO_{Y}=\scrO_X[\mathbf{t}]/(\mathbf{t}^{\bfr}-\mathbf{s})\Def\bigotimes_{i=1}^n\scrO_X[t_i]/(t_i^{r_i}-s_i),
\end{equation*}
which has a natural action of $\mu_{\bfr}=\prod_{i=1}^n\mu_{r_i}$.
\end{definition}

\begin{definition}(cf.~\cite[Definition 2.2]{bb17preprint})\label{def:tame ramif}
Let $G$ be a finite $k$-group scheme. A \textit{tamely ramified $G$-torsor} over $X$ with ramification data $(\bfD,\bfr)$ is a scheme $Y$ endowed with an action of $G$ and a finite flat $G$-invariant morphism $Y\longrightarrow X$ such that for any closed point $x$ of $X$, there exists a monomorphism $\mu_{\bfr_x}\longrightarrow G$ (cf.~(\ref{eq:bfr_x})) defined over an extension $k'/k$ such that in a fppf neighbourhood of $x$ in $X$, the morphism $Y\longrightarrow X$ is isomorphic to 
\begin{equation*}
(Z\times_k k')\times^{\mu_{\bfr_x}}G,
\end{equation*}
where $Z$ is the Kummer morphism associated with the data $(\mathbf{s}=(s_i)_{i=1}^n,\bfr_x)$, where each $s_i$ is a local equation of $D_i$.
\end{definition}

\begin{thm}(Biswas--Borne, cf.~\cite[\S3]{bb17preprint})\label{thm:biswas-borne}
Let $Y$ be a scheme endowed with an action of a finite $k$-group scheme $G$ and a finite flat $G$-invariant morphism $Y\longrightarrow X$. Then we have the following.
\begin{enumerate}
\renewcommand{\labelenumi}{(\arabic{enumi})}
\item If $Y\longrightarrow X$ is a tamely ramified $G$-torsor with ramification data $(\bfD,\bfr)$, then the morphism $Y\longrightarrow X$ uniquely factors through a $G$-torsor $Y\longrightarrow\sqrt[\bfr]{\bfD/X}$.
\item Suppose that $G$ is abelian. Then the converse of (1) is true. Namely, the morphism $Y\longrightarrow X$ factors through a $G$-torsor $Y\longrightarrow\sqrt[\bfr]{\bfD/X}$ if and only if $Y$ is a tamely ramified $G$-torsor over $X$ with ramification data $(\bfD,\bfr)$.
\end{enumerate}
\end{thm}

Particularly, the first result (1) indicates that a $G$-torsor $Y\lto\sqrt[\bfr]{\bfD/X}$ over the root stack $\sqrt[\bfr]{\bfD/X}$ which is representable by a $k$-scheme gives a candidate of a tamely ramified $G$-torsor over $X$ with ramification data $(\bfD,\bfr)$ in the sense of Definition \ref{def:tame ramif}.

\begin{rem}
With the above notation, suppose that $X$ is a regular connected $k$-scheme, $D$ is a simple normal crossings divisor   and $G$ is a finite constant $k$-group scheme. Then, for a finite flat $G$-invariant morphism $Y\lto X$ whose restriction to $X\setminus D$ is an \'etale Galois $G$-cover, the following are equivalent.
\begin{enumerate}
\renewcommand{\labelenumi}{(\alph{enumi})}
\item The morphism $Y\lto X$ is tamely ramified along $D$ in the sense of Definition \ref{def:tame cover}.
\item The morphism $Y\lto X$ is a tamely ramified $G$-torsor with ramification data $(\bfD,\bfr)$ for some $n$-tuple $\bfr$ in the sense of Definition \ref{def:tame ramif}.
\item The morphism $Y\lto X$ (uniquely) factors through a $G$-torsor $Y\lto\sqrt[\bfr]{\bfD/X}$ for some $n$-tuple $\bfr$.
\end{enumerate}
Indeed, the implication (a)$\Longrightarrow$(b) is immediate from Definitions \ref{def:tame cover} and \ref{def:tame ramif}. The implication (b)$\Longrightarrow$ (c) is nothing other than the assertion of Theorem \ref{thm:biswas-borne}(1). Finally, the implication (c)$\Longrightarrow$(a) can be deduced from Abhyankar's lemma~(cf.~Lemma \ref{lem:abhyankar}). Indeed, since the problem is Zariski local, without loss of generality, we may assume that $[Z/\mu_{\bfr}]\simeq\sqrt[\bfr]{\bfD/X}$, where $Z$ is a Kummer morphism. Then, we get a commutative diagram
\begin{equation*}
\begin{xy}
\xymatrix{\ar@{}[rd]|{\square}
P\ar[r]\ar[d]&Y\ar[d]\ar@/^10mm/[dd]\\
Z\ar[r]\ar[rd]&\sqrt[\bfr]{\bfD/X}\ar[d]\\
& X
}
\end{xy}
\end{equation*}
with the 2-Cartesian square, where $P\lto Z$ is a $G$-torsor and $Z\lto\sqrt[\bfr]{\bfD/X}$ is the quotient map $Z\lto[Z/\mu_{\bfr}]\simeq\sqrt[\bfr]{\bfD/X}$. Let $x_i$ denote the generic point of the divisor $D_i$ for any $i$. Then, the above commutative diagram implies that the ramification index $e_{x_i}$ at $x_i$ divides $r_i$. As $Y$ is representable, the composition map
\begin{equation*}
\cB\mu_{r_i}=\calG_{x_i}\lto\sqrt[\bfr]{\bfD/X}\xrightarrow{~Y~}\cB G
\end{equation*}
is faithful, hence $\mu_{r_i}$ is \'etale group scheme, i.e.\ $r_i$ is prime-to-$p$. As $e_{x_i}$ divides $r_i$, this implies that $e_{x_i}$ is prime-to-$p$. Therefore, $Y\lto X$ is tamely ramified above $x_i$. Moreover, since $\sqrt[\bfr]{\bfD/X}$ is regular~(cf.~Proposition \ref{prop:sm root stack}), so is $Y$. Now, Lemma \ref{lem:abhyankar} implies that the condition (a) holds.
\end{rem}

\begin{rem}\label{rem:Rydh}
To figure out the reason why the abelianness assumption is required in Theorem \ref{thm:biswas-borne}(2), let us recall the argument due to Biswas--Borne. The problem is Zariski local, so we may assume that $\sqrt[\bfr]{\bfD/X}=[Z/\mu_{\bfr}]$. Let $P\Def\underline{\mathrm{Isom}}_{[Z/\mu_{\bfr}]}(Z\wedge^{\mu_{\bfr}}G,~Y)$ be the fppf sheaf of isomorphisms of $G$-torsors over $[Z/\mu_{\bfr}]$. Then $P$ is a right $N=Z\wedge^{\mu_{\bfr}}G$-torsor over $[Z/\mu_{\bfr}]$, where the action of $\mu_{\bfr}$ on $G$ is defined by conjugation. For example, if $G$ is abelian, then the conjugacy action is trivial, hence $N=[Z/\mu_{\bfr}]\times_k G$ and Proposition \ref{prop:nori reconst} can be applied to show that $P$ descents to a $G$-torsor over $X$~(cf.~Proposition \ref{prop:Nori ger root stack}), which implies that $P$ can be trivialized by an fppf cover $X'\lto X$. In the general case, it does not seem that $\cB N$ is a fibered category over schemes which satisfies a tannakian reconstruction~(\S\ref{sec:tann reconst}), and the same argument cannot be applied. 
In fact, as explained in \cite[Appendix \S B]{bb17preprint}, David Rydh constructs examples of $G=\mu_p\ltimes\alpha_p$-torsors over smooth root stacks $\sqrt[p]{D/X}$ for which the associated sheaf $P$ of isomorphisms of $G$-torsors cannot be trivialized by any fppf cover $X'\lto X$. 
\end{rem}

For later use, we shall show the following lemma.

\begin{lem}\label{lem:tors rs}
Suppose that $X$ is a connected smooth curve over an algebraically closed field $k$ of characteristic $p>0$ with $\bfD=(x_i)_{i\in I}$, where $x_i$ are finite distinct closed points of $X$.
Let $G$ be a finite $k$-group scheme. Let $\calY\longrightarrow\fX^{\bfr}$ be a Nori-reduced $G$-torsor. Then there exists a family $\bfr'=(r'_i)_{i\in I}$ of integers $r'_i> 0$ with $\bfr'|\bfr$ and a $G$-torsor $Y'\longrightarrow\fX^{\bfr'}$ with $Y'$ representable by a $k$-scheme such that
\begin{equation*}
\calY\simeq Y'\times_{\fX^{\bfr'}}\fX^{\bfr}
\end{equation*}
as a $G$-torsor over $\fX^{\bfr}$. 
\end{lem}

\begin{proof}
First note that if a $G$-torsor $\calY\lto\fX^{\bfr}$ is representable, then as remarked in \cite[Remark 3.6(2)]{bb17preprint}, $\calY$ can be represented by a $k$-scheme, and by \cite[\S3.2 Proposition 11]{bb17}, the latter condition is equivalent to the condition that the composition of morphisms
\begin{equation*}
\mathcal{G}_{x_i}\lto\fX^{\bfr}\xrightarrow{~\calY~}\mathcal{B}G
\end{equation*}
is representable for any $i\in I$. Since the problem is Zariski local, we may assume that $\# I=1$, so let us put $x=x_1$, $r=r_1$. Moreover, we may assume that $\fX^r=[Z/\mu_r]$ where $Z\lto X$ is a Kummer morphism. Then we have $\mathcal{G}_{x}=\cB\mu_r$~(cf.~Proposition \ref{prop:root stack D}(2)). Let $\mu_r\lto G$ be the homomorphism of $k$-group schemes associated with the composition $\cB\mu_r=\calG_x\lto[Z/\mu_r]=\fX^r\xrightarrow{~\calY~}\cB G$ and set $\mu_{r''}\Def\Ker(\mu_r\lto G)$.  Then, by the definition, the restriction
\begin{equation}\label{eq:tors rs}
[Z/\mu_{r''}]\lto[Z/\mu_r]=\fX^{r}\xrightarrow{~\mathcal{Y}~}\cB G
\end{equation}
is trivial on inertia groups. Therefore, by \cite[\S2.3 Corollary 5]{bb17}, which can be applied to non-proper tame stacks as discussed in the first paragraph of the proof of Proposition \ref{prop:Nori ger root stack}(1), the morphism (\ref{eq:tors rs}) factors through the coarse moduli space $Z'\Def Z/\mu_{r''}$.  By dividing by the natural action of $\mu_{r'}$, we get a morphism
\begin{equation*}
\fX^{r'}=[Z'/\mu_{r'}]\lto\cB G,
\end{equation*} 
which is representable by the construction. 
Therefore, we can conclude that the $G$-torsor $\calY\lto\fX^{r}$ descends to a $G$-torsor $Y\lto\fX^{r'}$ which is representable by a $k$-scheme. This completes the proof.   
\end{proof}

\subsection{Brauer groups of root stacks}\label{subsec:Br}

In this subsection, we always work with the lisse-\'etale topology~(cf.~\cite{LMB00}). 
Let $\calX$ be an algebraic stack over a field $k$. We denote by $\Liset(\calX)$ the lisse-\'etale site of $\calX$. A gerbe over $\calX$ always means a gerbe over the site $\Liset(\calX)$~(cf.~\cite[Definition 06NZ]{stack}). 
To each gerbe $\calG\lto\calX$, we associate a sheaf $\calI(\calG)$ on $\calG$, which we call the \textit{inertia sheaf}, as follows. For any object $(U,u)\in\Liset(\calX)$ and any section $x:U\lto\calG$, $\calI(\calG)(x)\Def\Aut_{\calG(U)}(x)$ is the group of automorphisms of the section $x$.  

\begin{definition}(cf.~\cite[Definition 2.2.1.6]{li08})\label{def:inertial action}
Let $\mathcal{F}$ be a sheaf on $\calG$. Then $\mathcal{F}$ admits a right action $\calF\times\calI(\calG)\lto\calF$ of the inertia sheaf $\mathcal{I}({\calG})$, which is called the \textit{inertial action} on $\mathcal{F}$, as follows. For any  object $(U,u)\in\Liset(\calX)$ and any section $x:U\lto\calG$, there exists a natural action
\begin{equation*}
\calF(x)\times\calI(\calG)(x)\lto \calF(x)~;~ (a,\sigma)\longmapsto \sigma^*a.
\end{equation*}
Here, note that each element $\sigma\in\calI(\calG)(x)$ induces an isomorphism $\sigma^*:\calF(x)\xrightarrow{~\simeq~}\calF(x)$. 
\end{definition}

A gerbe $f:\calG\lto\calX$ is said to be \textit{abelian} if for any $(U,u)\in\Liset(\calX)$ and any section $x:U\lto\calG$, the sheaf $\underline{\mathrm{Aut}}_U(x)$ on $\Liset(U)$ is abelian. If a gerbe $f:\calG\lto\calX$ is abelian, for any $(U,u)\in\Liset(\calX)$ and any two sections $x,y\in\calG_{(U,u)}$, there exists a canonical isomorphism $\underline{\mathrm{Aut}}_U(x)\simeq\underline{\mathrm{Aut}}_U(y)$ of sheaves. Moreover, there exists a sheaf $\mathcal{A}$ of abelian groups on $\Liset(\calX)$ such that for any $(U,u)\in\Liset(\calX)$ and any $x:U\lto \calG$, there exists an isomorphism $u^*\mathcal{A}\simeq \underline{\mathrm{Aut}}_U(x)$ of sheaves such that for any $(U,u)\in\Liset(\calX)$ and any two sections $x,y:U\lto \calG$, the diagram
\begin{equation*}
\begin{xy}
\xymatrix{
u^*\mathcal{A}\ar[r]\ar[rd]&\underline{\mathrm{Aut}}_U(x)\ar[d]\\
&\underline{\mathrm{Aut}}_U(y).
}
\end{xy}
\end{equation*}
is commutative~(cf.~\cite[Lemma 0CJY]{stack}). In other words, there exists an isomorphism of sheaves on $\calG$,
\begin{equation*}
\mathcal{A}_{\calG}\Def f^*\mathcal{A}\simeq \mathcal{I}({\calG}).  
\end{equation*} 
In this case, such a gerbe $\calG\lto\calX$ is called an $\mathcal{A}$-\textit{gerbe}. More precisely, for a given  abelian sheaf $\mathcal{A}$ on $\Liset(\calX)$, an $\mathcal{A}$-\textit{gerbe} is a gerbe $\calG\lto\calX$ together with an isomorphism $\mathcal{A}_{\calG}\simeq\mathcal{I}(\calG)$ of sheaves on $\calX$. 
Let $\mathcal{A}$ be an abelian sheaf on $\Liset(\calX)$. Let $\calG_1,\calG_2$ be $\mathcal{A}$-gerbes over $\calX$. Then $\calG_1$ is said to be isomorphic to $\calG_2$ if there exists an equivalence of fibered categories $\phi:\calG_1\xrightarrow{~\simeq~}\calG_2$ over $\calX$ such that the composition
\begin{equation*}
\mathcal{A}_{\calG_2}\xrightarrow{~\simeq~}\mathcal{I}({\calG_2})\lto \phi^*\mathcal{I}({\calG_1})\xrightarrow{~\simeq~}\phi^*\mathcal{A}_{\calG_1}=\mathcal{A}_{\calG_2}
\end{equation*}
is the identity~(cf.~\cite[Definition 2.2.1.3]{li08}). Then the set of isomorphism classes of $\mathcal{A}$-gerbes can be described as the second cohomology group  $H_{\liset}^2(\calX,\mathcal{A})$. The unit of the abelian group $H_{\liset}^2(\calX,\mathcal{A})$ is represented by the fibered category 
\begin{equation*}
\TORS_{\calX}(\mathcal{A})\lto\calX
\end{equation*}
which classifies $\mathcal{A}$-torsors on the site $\Liset(\calX)$. 

From now on, we shall consider the multiplicative group $\mathcal{A}=\G_{m}$. Let $f:\calG\lto\calX$ be a $\G_{m}$-gerbe. Then $\calG$ is an algebraic stack over $k$~(cf.~\cite[Proposition 1.1]{hs09}).  In this case, any quasi-coherent sheaf $\calF$ on $\calG$ admits a left action of $\G_{m}$ which comes from the left $\scrO_{\calG}$-module structure. 

\begin{definition}(cf.~\cite[Definition 3.1.1.1]{li08})
With the above notation, 
a quasi-coherent sheaf $\mathcal{F}$ on $\calG$ is said to be a \textit{twisted sheaf} if the right action associated with the left $\G_{m}$-action on $\calF$ coincides with the inertial action $\calF\times\calI(\calG)\lto\calF$~(cf.~Definition \ref{def:inertial action}). Here, recall that $\G_{m,\calG}=\calI(\calG)$.
\end{definition}

Let us recall the following result.

\begin{prop}(cf.~\cite[Lemma 3.1.1.8]{li08})\label{prop:lieblich's lem}
Let $\calX$ be an algebraic stack over a field $k$. 
A $\G_{m}$-gerbe $\calG$ is isomorphic to $\cB\G_{m,\calX}$ as a $\G_{m}$-gerbe if and only if there exists a  twisted invertible sheaf $\mathcal{L}$ on $\calG$. 
\end{prop}

\begin{proof}
First note that there exists an equivalence of groupoids
\begin{equation*}
\Hom_{\calX}(\calG,\cB\G_{m,\calX})\simeq\mathsf{Pic}(\calG).
\end{equation*}
Let $\phi:\calG\lto\cB\G_{m,\calX}$ be a morphism of fibered categories over $\calX$ and $\mathcal{L}$ the invertible sheaf on $\calG$ corresponding to $\phi$. We have to show that $\phi$ is an isomorphism of $\G_{m}$-gerbes if and only if $\mathcal{L}$ is a twisted sheaf on $\calG$. Indeed, $\mathcal{L}$ is a twisted sheaf on $\calG$ if and only if, for any object $(U,u)\in\Liset(\calX)$ and any section $x:U\lto\calG$, the diagram
\begin{equation*}
\begin{xy}
\xymatrix{
x^*\mathcal{L}\times\G_{m,U}\ar[r]^{\simeq~~~}\ar[rd]&x^*\mathcal{L}\times\uAut_U(x)\ar[d]\\
&x^*\mathcal{L}
}
\end{xy}
\end{equation*}
is commutative, where the action $x^*\mathcal{L}\times\G_{m,U}\lto x^*\mathcal{L}$ is given by $(f,a)\longmapsto a^{-1}\cdot f$ and the one $x^{*}\mathcal{L}\times\uAut_U(x)\lto x^*\mathcal{L}$ is the inertial action~(cf.~Definition \ref{def:inertial action}). However, the commutativity of the diagram is equivalent to the commutativity of the following diagram,
{\small\begin{equation*}
\begin{xy}
\xymatrix{
\underline{\mathrm{Aut}}_U(x)\ar[r]\ar[d]&\underline{\mathrm{Aut}}_U(x^*\mathcal{L})\ar[d]\\
\G_{m,U}\ar[r]&\uAut_U(\phi\circ x), 
}
\end{xy}
\end{equation*}}
or equivalently, to saying that the composition
\begin{equation*}
\G_{m,U}\xrightarrow{~\simeq~}\uAut_U(x)\lto\uAut_U(\phi\circ x)\xrightarrow{~\simeq~}\G_{m, U}
\end{equation*}
is the identity. This completes the proof. 
\end{proof}

Now we apply the above arguments to root stacks. Let $X$ be a geometrically connected quasi-compact smooth scheme over the spectrum $S=\Spec k$ of a field $k$ with $\eta$ the generic point. Let $\bfD=(D_i)_{i\in I}$ be a finite family of reduced irreducible effective Cartier divisors on $X$ and put  $D=\cup_{i\in I}D_i\subset X$. For each $\bfr=(r_i)_{i\in I}$ with $r_i>0$, as in \S\ref{subsec:root stack}, we put
\begin{equation*}
\fX^{\bfr}=\sqrt[\bfr]{\bfD/X}.
\end{equation*} 
Recall that there exists a natural quasi-compact morphism 
$\eta\hookrightarrow\fX^{\bfr}$.

\begin{prop}(cf.~\cite[Proposition 3.1.3.3]{li08})\label{prop:brauer inj}
With the above notation, suppose further that $D_i$ are smooth and that $D$ is a simple normal crossings divisor on $X$. Then the restriction map 
\begin{equation*}
H_{\liset}^2(\fX^{\bfr},\G_m)\lto H_{\liset}^2(\eta,\G_m)
\end{equation*}
is injective.
\end{prop}

As a consequence, we have the following.

\begin{cor}\label{cor:brauer=0}
With the same notation as in Proposition \ref{prop:brauer inj}, if $X$ is a connected smooth curve defined over an algebraically closed field, then we have $H_{\liset}^2(\fX^{\bfr},\G_m)=0$.
\end{cor}

This is an immediate consequence of Proposition \ref{prop:brauer inj} because if $X$ is a connected smooth curve defined over an algebraically closed field, then the Brauer group of the generic point $\eta$ vanishes~(cf.~\cite[Lemma 03RF]{stack}). Let us prove  Proposition \ref{prop:brauer inj}.

\begin{lem}\label{lem:regular noether rs}
With the same notation as in Proposition \ref{prop:brauer inj}, the root stack $\fX^{\bfr}$ is a smooth quasi-compact algebraic stack over $S$. 
\end{lem}

\begin{proof}
The smoothness follows from Proposition \ref{prop:sm root stack}. Recall that we have assumed that $X$ is quasi-compact. As the morphism $\fX^{\bfr}\lto X$ is of finite type~(cf.~\S\ref{subsec:root stack}), $\fX^{\bfr}$ is quasi-compact. 
\end{proof}

\begin{lem}\label{lem:gerbe regular noether}
With the same notation as in Proposition \ref{prop:brauer inj}, let $\calG\lto\fX^{\bfr}$ be a $\G_m$-gerbe. Then $\calG$ is a smooth quasi-compact algebraic stack over $k$.
\end{lem}

\begin{proof}
The classifying stack $\cB \G_m$ is a smooth quasi-compact algebraic stack. As any $\G_m$-gerbe $\calG\lto\fX^{\bfr}$ is locally isomorphic to $\cB\G_m$, the gerbe $\calG$ must be smooth and quasi-compact. 
\end{proof}

\begin{proof}[Proof of Proposition \ref{prop:brauer inj}]
Let $\calG\lto\fX^{\bfr}$ be a $\G_m$-gerbe and $[\calG]\in H^2_{\liset}(\fX^{\bfr},\G_m)$ the corresponding class. Suppose that $[\calG_{\eta}]=0$ in $H_{\liset}^2(\eta,\G_m)$. We have to show that $[\calG]=0$. 
Since $[\calG_{\eta}]=0$, by Proposition \ref{prop:lieblich's lem}, there exists an invertible $\calG_{\eta}$-twisted sheaf $\mathcal{L}_{\eta}$. Since $j:{\eta}\lto\fX^{\bfr}$ is quasi-compact, so is $i:\calG_{\eta}\lto\calG$. Therefore, the sheaf $i_*\mathcal{L}_{\eta}$ is a quasi-coherent $\calG$-twisted sheaf~(cf.~\cite[Proposition (13.2.6)(i)]{LMB00}). Since $\calG$ is Noetherian~(Lemma \ref{lem:gerbe regular noether}), according to \cite[Proposition 3.1.1.9]{li08}, we can write $\mathcal{M}\Def i_*\mathcal{L}_{\eta}$ as the colimit of coherent $\calG$-twisted subsheaves of it as follows.
\begin{equation*}
\mathcal{M}=\varinjlim_{\lambda}\mathcal{M}_{\lambda}
\end{equation*}
However, since $\mathcal{M}_{\eta}$ is coherent, we have $\mathcal{M}_{\eta}=i^*\mathcal{M}=\varinjlim_{\lambda}i^*\mathcal{M}_{\lambda}=i^*\mathcal{M}_{\lambda}$ for some $\lambda$. By replacing $\mathcal{M}_{\lambda}$  with $\mathcal{L}\Def\mathcal{M}_{\lambda}^{\vee\vee}$, if necessary, we obtain a reflexible $\calG$-twisted sheaf $\mathcal{L}$ of rank one such that $i^*\mathcal{L}\simeq\mathcal{L}_{\eta}$. However, since $\calG$ is smooth over $k$~(cf.~Lemma \ref{lem:gerbe regular noether}), any reflexible sheaf of rank one must be an invertible sheaf. Therefore, again by using Proposition \ref{prop:lieblich's lem}, we can conclude that $[\calG]=0$. This completes the proof. 
\end{proof}

\subsection{Embedding problems over root stacks}

Let $k$ be an algebraically closed field of characteristic $p>0$ and $S=\Spec k$. Let $X$ be a projective smooth curve over $k$ of genus $g\ge 0$ and of $p$-rank $\gamma\ge 0$. Let $\emptyset\neq U\subsetneq X$ be a nonempty open subscheme with $\# X\setminus U=n\ge 1$. Let $X\setminus U=\{x_0,x_1,\cdots,x_{n-1}\}$. Fix an integer $m$ with $0\le m\le n-1$. We shall denote by $X_m$ the smooth affine curve $X\setminus\{x_0,x_1,\dots,x_m\}$ and consider $\bfD=(x_{i})_{i=m+1}^{n-1}$ as a family of reduced distinct Cartier divisors on $X_m$. For each family $\bfr=(r_{i})_{i=m+1}^{n-1}\in\prod_{i=m+1}^{n-1}\Z_{\ge 0}$ of integers, we denote by $\fX_m^{p^{\bfr}}=\sqrt[p^{\bfr}]{\bfD/X_m}$ the root stack associated with $X_m$ and the data $(\bfD,p^{\bfr})$. 

Then there exists an exact sequence
\begin{equation}\label{eq:Pic trunc}
0\lto\Pic(X_m)\lto \Pic(\fX_m^{p^{\bfr}})\lto\prod_{i={m+1}}^{n-1} H^0(x_{i},\Z/p^{r_i}\Z)\lto 0,
\end{equation}
where each group $H^0(x_i,\Z/p^{r_i}\Z)$ is nothing but the Picard group of the residual gerbe $\calG_{x_i}$ of $\fX^{p^{\bfr}}$ at the closed point $x_i$ (cf.~\cite{bo09}). From the definition of root stacks, for any two families $\bfr,\bfr'$ of positive integers with $\bfr \le \bfr'$, there exists a natural morphism of algebraic stacks
\begin{equation*}
\fX_m^{p^{\bfr'}}\lto\fX_m^{p^{\bfr}}.
\end{equation*} 
Then the exact sequence (\ref{eq:Pic trunc}) implies that there exists an exact sequence
\begin{equation}\label{eq:Pic p-typical}
0\lto\Pic(X_m)\lto\varinjlim_{\bfr}\Pic(\fX_m^{p^{\bfr}})\lto(\Q_p/\Z_p)^{\oplus n-m-1}\lto 0.
\end{equation}
Since $X_m$ is affine, there exists a surjective homomorphisms $\Pic^0(X)\twoheadrightarrow\Pic(X_m)$, hence $\Pic(X_m)$ is divisible. Thus,  the exact sequence (\ref{eq:Pic p-typical}) implies that the abelian group $\varinjlim_{\bfr}\Pic(\fX_m^{p^\bfr})$ is $p$-divisible.

\begin{lem}\label{lem:coho root stack}
Let $\bfr=(r_i)_{i=m+1}^{n-1}$ be a family of positive integers. Then we have the following.
\begin{enumerate}
\renewcommand{\labelenumi}{(\arabic{enumi})}
\item For any quasi-coherent sheaf $E$ on $\fX_m^{p^{\bfr}}$, we define the fppf sheaf $W(E)$ of abelian groups on $\fX_m^{p^{\bfr}}$ to be
 \begin{equation*}
 W(E)(Y)\Def\Gamma(Y,E\otimes\scrO_Y)
 \end{equation*}
 for any morphism $Y\lto\fX_m^{p^{\bfr}}$ from a scheme $Y$. Then $H_{\fppf}^q(\fX_m^{p^{\bfr}},W(E))=0$ for any $q>0$. 
\item $H^0_{\fppf}(\fX_m^{p^{\bfr}},\G_{m})=H^0(X_m,\scrO_{X_m})^{\times}$,
\item $H^1_{\fppf}(\fX_m^{p^{\bfr}},\G_m)={\rm Pic}(\fX_m^{p^{\bfr}})$.
\item $H^2_{\fppf}(\fX_m^{p^{\bfr}},\G_m)=0$. 
\end{enumerate}
\end{lem}

\begin{proof}
(1) First note that there exists a natural isomorphism
\begin{equation*}
H^q_{\fppf}(\fX_m^{p^{\bfr}},W(E))\xrightarrow{~\simeq~} H_{\liset}^q(\fX_m^{p^{\bfr}},W(E))=H^q(\fX_m^{p^{\bfr}},E)
\end{equation*}
for any $q\ge 0$~(cf.~\cite[Th\'eor\`eme B.2.5]{br09}). Next since $X_m$ is affine, we have
\begin{equation*}
H^q(\fX_m^{p^{\bfr}},E)\simeq H^0(X_m,R^q\pi_{p^{\bfr}*}E)
\end{equation*}
for any $q\ge 0$. 
On the other hand, since $\pi_{p^{\bfr}}:\fX_m^{p^{\bfr}}\lto X_m$ is cohomologically affine and $\fX_m^{p^{\bfr}}$ has affine diagonal~(cf.~Proposition \ref{prop:root stack}(1)(3)), $R^q\pi_{p^{\bfr}*}E=0$ for any $q>0$~(cf.~\cite[Remark 3.5]{alp13}). Therefore we can conclude that
\begin{equation*}
H_{\fppf}^q(\fX_m^{p^{\bfr}},W(E))=0
\end{equation*}
for any $q>0$. 

(2) This follows from
\begin{equation*}
\begin{aligned}
H^0_{\fppf}(\fX_m^{p^{\bfr}},\G_m)&\simeq H_{\liset}^0(\fX_m^{p^{\bfr}},\G_m)=\varprojlim_{(U,u)\in\Liset(\fX^{p^{\bfr}})}\G_m(U)\\
&\simeq \bigl(\varprojlim_{(U,u)\in\Liset(\fX_m^{p^{\bfr}})}\Gamma(U, \calO_U)\bigl)^{\times}
\simeq H^0(\fX_m^{p^{\bfr}},\calO_{\fX_m^{p^{\bfr}}})^{\times}\\
&\simeq H^0(X_m,\pi_{p^{\bfr}*}\calO_{\fX_m^{p^{\bfr}}})^{\times}=H^0(X_m,\calO_{X_m})^{\times}.
\end{aligned}
\end{equation*}
Here, recall that $\pi_{p^{\bfr}*}\calO_{\fX_m^{p^{\bfr}}}=\calO_{X_m}$. 

(3) See \cite{br09}.

(4) This follows from Corollary \ref{cor:brauer=0} together with \cite[Th\'eor\`eme B.2.5]{br09}.
\end{proof}

\begin{lem}\label{lem:Pic trunc torsion}
We have the following.
\begin{enumerate}
\renewcommand{\labelenumi}{(\arabic{enumi})}
\item For any integer $s\ge 1$ and any two families $\bfr,\bfr'$ of positive integers with $(s)_{i=m+1}^{n-1}\le\bfr\le\bfr'$, we have ${\rm Pic}(\fX_m^{p^{\bfr}})[p^s]={\rm Pic}(\fX_m^{p^{\bfr'}})[p^s]$.
\item For any family $\bfr$ of positive integers, we have ${\rm Pic}(\fX_m^{p^{\bfr}})/p{\rm Pic}(\fX_m^{p^{\bfr}})\simeq (\Z/p\Z)^{\oplus n-m-1}$.
\end{enumerate}
\end{lem}

\begin{proof}
This can be verified by using the exact sequence (\ref{eq:Pic trunc}) and the fact that $\Pic(X_m)$ is divisible.
\end{proof}

\begin{lem}\label{lem:H^1 mu_p}
For any integer $s>0$, we have 
$\varinjlim_{\bfr}H_{\fppf}^1(\fX_m^{p^{\bfr}},\mu_{p^s})\simeq(\Z/p^s\Z)^{\oplus \gamma + n-1}$. 
\end{lem}

\begin{proof}
By Proposition \ref{prop:Nori ger sm root stack}, the restriction maps
\begin{equation*}
\varinjlim_{\bfr'}H^1_{\fppf}(\fX_{m-1}^{p^{\bfr'}},\mu_{p^s})\lto\varinjlim_{\bfr}H^1_{\fppf}(\fX_m^{p^{\bfr}},\mu_{p^s})
\end{equation*}
are injective for all $0< m\le n-1$. Since $\varinjlim_{\bfr}H^1_{\fppf}(\fX_{n-1}^{p^{\bfr}},\mu_{p^s})=H^1_{\fppf}(U,\mu_{p^s})\simeq(\Z/p^s\Z)^{\oplus \gamma+n-1}$~(cf.~\cite[Proposition 3.2]{ot17}), this reduces us to the case when $m=0$. In this case, as $H^0(X_0,\scrO_{X_0})=k^{\times}$, by Lemmas \ref{lem:coho root stack}(2)(3) and \ref{lem:Pic trunc torsion}(1), we can find that
\begin{equation*}
\varinjlim_{\bfr}H^1_{\fppf}(\fX_0^{p^{\bfr}},\mu_{p^s})\simeq\varinjlim_{\bfr}H^1_{\fppf}(\fX_0^{p^{\bfr}},\G_m)[p^s]
\simeq\varinjlim_{\bfr}\Pic(\fX_0^{p^{\bfr}})[p^s]\simeq\Pic(\fX_0^{p^{\bfr}})[p^s]
\end{equation*}
for an arbitrary family $\bfr\ge (s)_{i=m+1}^{n-1}$. Fix such an index $\bfr$. 
Since $H_{\fppf}^1(X_0,\mu_{p^s})\simeq H_{\fppf}^1(X,\mu_{p^s})\simeq (\Z/p^s\Z)^{\oplus \gamma}$~(cf.~\cite[Proposition 3.2]{ot17}), the assertion follows from the exact sequence (\ref{eq:Pic trunc}) together with the fact that $\Pic(X_0)$ is a  $p$-divisible group. 
\end{proof}

\begin{lem}\label{lem:vanish H^2}
We have the following.
\begin{enumerate}
\renewcommand{\labelenumi}{(\arabic{enumi})}
\item Let $N$ be a finite flat abelian $\fX_m^{p^{\bfr}}$-group scheme. If $N$ and its Cartier dual $N^D$ are of height one, then we have $H_{\fppf}^2(\fX_m^{p^{\bfr}},N)=0$.
\item If $G$ is a finite local diagonalizable $k$-group scheme, we have $\varinjlim_{\bfr}H_{\fppf}^2(\fX_m^{p^{\bfr}},G)=0$. 
\end{enumerate}
\end{lem}

\begin{proof}
(1) See the last assertion of Proposition \ref{prop:inf alpha p}.

(2) As $G$ is a successive extension of $\mu_p$, without loss of generality, we may assume that $G=\mu_p$. We will prove that the transition map
\begin{equation*}
H^2_{\fppf}(\fX_m^{p^{\bfr}},\mu_p)\lto H^2_{\fppf}(\fX_m^{p^{\bfr+\unit}},\mu_p)
\end{equation*}
is the zero map for all $\bfr$. 
By Lemma \ref{lem:coho root stack}(4), for any family $\bfr$ of positive integers, there exists an isomorphism of abelian groups 
${\rm Pic}(\fX_m^{p^{\bfr}})/p{\rm Pic}(\fX_m^{p^{\bfr}})\xrightarrow{~\simeq~} H_{\fppf}^2(\fX_m^{p^{\bfr}},\mu_p)$. 
Therefore it suffices to show that the composition
\begin{equation*}
\begin{xy}
\xymatrix{
{\rm Pic}(\fX_m^{p^{\bfr}})\ar[r]&{\rm Pic}(\fX_m^{p^{\bfr+\unit}})\ar@{->>}[r]&{\rm Pic}(\fX_m^{p^{\bfr+\unit}})/p{\rm Pic}(\fX_m^{p^{\bfr+\unit}
})
}
\end{xy}
\end{equation*}
is the zero map. To prove this, it suffices to notice that there exists a commutative diagram
{\small\begin{equation*}
\begin{xy}
\xymatrix{
&&&0\ar[d]&\\
0\ar[r]&{\rm Pic}(X_m)\ar[r]\ar@{=}[d]&{\rm Pic}(\fX_m^{p^{\bfr}})\ar[r]\ar[d]&\prod_{i=m+1}^{n-1}\Z/p^{r_i}\Z\ar[r]\ar@{^{(}->}[d]&0\\
0\ar[r]&{\rm Pic}(X_m)\ar[r]&{\rm Pic}(\fX_m^{p^{\bfr+\unit}})\ar[r]\ar@{->>}[d]&\prod_{i=m+1}^{n-1}\Z/p^{r_i+1}\Z\ar[r]\ar[d]&0\\
&&{\rm Pic}(\fX_m^{p^{\bfr+\unit}})/p{\rm Pic}(\fX_m^{p^{\bfr+\unit}})\ar[r]^{~~~~~~\simeq}&\prod_{i=m+1}^{n-1}\Z/p\Z\ar[d]&\\
&&&0&
}
\end{xy}
\end{equation*}}
where the right vertical sequence is exact, and the isomorphism $\Pic(\fX_m^{p^{\bfr+\unit}})/p\Pic(\fX_m^{p^{\bfr+\unit}})\xrightarrow{~\simeq~}\prod_{i=m+1}^{n-1}\Z/p\Z$ is Lemma \ref{lem:Pic trunc torsion}(2). This completes the proof.  
\end{proof}

Now we prove the main theorem. As a notation, we define $\pi^{\loc}(\fX_m^{\mathbf{p}^{\infty}})$ to be the projective limit of the pro-system $\{\pi^{\loc}(\fX_m^{p^{\bfr}})\}_{\bfr}$, i.e.\ 
\begin{equation}\label{eq:pi loc inf root stack}
\pi^{\loc}(\fX_m^{\mathbf{p}^{\infty}})~\Def~\varprojlim_{\bfr}\pi^{\loc}(\fX_m^{p^{\bfr}}). 
\end{equation}

\begin{thm}\label{thm:PISAC solv}
Suppose given an exact sequence of finite local $k$-group schemes
\begin{equation}\label{eq:thm PISAC solv}
1\lto G'\lto G\xrightarrow{~\pi~} G''\lto 1.
\end{equation}
Suppose that the following conditions are satisfied.
\begin{enumerate}
\renewcommand{\labelenumi}{(\roman{enumi})}
\item There exists an injective homomorphism $\X(G)\hookrightarrow(\Q_p/\Z_p)^{\oplus\gamma+n-1}$.
\item There exists a surjective $k$-homomorphism $\overline{\phi}:\pi^{\loc}(\fX_m^{\mathbf{p}^{\infty}})\twoheadrightarrow G''$. 
\item $G'$ is solvable.
\end{enumerate}
Then there exist an $(n-m-1)$-tuple $\bfr$ of non-negative integers and a Nori-reduced $G$-torsor $Y\lto\fX_m^{p^{\bfr}}$ which is representable by a $k$-scheme. 
\end{thm}

To prove the theorem, we need the following lemma.

\begin{lem}\label{lem:PISAC solv}
Suppose given an exact sequence of finite local $k$-group schemes,
\begin{equation}\label{eq:lem PISAC solv}
1\lto G'\lto G\lto G''\lto 1
\end{equation}
together with a surjective homomorphism $\overline{\phi}:\pi^{\loc}(\fX_m^{\mathbf{p}^{\infty}})\twoheadrightarrow G''$. If $G'$ is abelian, then $\overline{\phi}$ can be lifted to a homomorphism $\pi^{\loc}(\fX_m^{\mathbf{p}^{\infty}})\lto G$. 
\end{lem}

\begin{proof}
As $G'$ is a finite abelian $k$-group scheme with $k$ perfect, it decomposes into a direct product $G'=G'_1\times G'_2$ of a finite local unipotent group scheme $G'_1$ and a finite local diagonalizable group scheme $G'_2$.  As the decomposition is preserved under any automorphism of $G'$, we get a commutative diagram of group schemes with exact rows and columns
\begin{equation*}
\begin{xy}
\xymatrix{
&1\ar[d]&1\ar[d]&&\\
&G'_1\ar@{=}[r]\ar[d]&G'_1\ar[d]&&\\
1\ar[r]&G'\ar[r]\ar[d]&G\ar[r]\ar[d]&G''\ar[r]\ar@{=}[d]&1\\
1\ar[r]&G_2'\ar[r]\ar[d]&G/G'_1\ar[r]\ar[d]&G''\ar[r]&1,\\
&1&1&&
}
\end{xy}
\end{equation*}
which allows us to assume that $G'=G'_1$, or that $G'=G'_2$. In the latter case, as the automorphism group scheme $\uAut(G')$ is \'etale, the extension (\ref{eq:lem PISAC solv}) is central, i.e.\ the conjugacy action of $G''$ on $G'$ is trivial, hence Giraud's theory of non-abelian cohomology~(cf.~\cite[p.\ 284, Remarque 4.2.10]{gi71}) can be applied, which implies that the assertion follows from the vanishing of the cohomology group $\varinjlim_{\bfr}H^2_{\fppf}(\fX_m^{p^{\bfr}},G')$~(cf.~Lemma \ref{lem:vanish H^2}(2)). 

Therefore, we have only to deal with the case when $G'$ is a finite local abelian unipotent group scheme. As $G'$ is obtained by successive extensions of finite local group schemes of height one, without loss of generality, we may assume that $G'$ is of height one. Similarly, we may further assume that the Cartier dual $G'^D$ is also of height one. Take an $m$-tuple $\bfr$ so that $\overline{\phi}$ factors through $\pi^{\loc}(\fX_m^{p^{\bfr}})$. The exact sequence (\ref{eq:lem PISAC solv}) gives rise to a gerbe $\cB G\lto\cB G''$. By pulling back the gerbe along the morphism 
\begin{equation*}
\fX_m^{p^{\bfr}}\lto\cB\pi^{\loc}(\fX_m^{p^{\bfr}})\lto\cB G'',
\end{equation*}
we get a gerbe $\calG\lto \fX_m^{p^{\bfr}}$. If $P''\lto \fX_m^{p^{\bfr}}$ denotes the $G''$-torsor associated with the above morphism, then the composition $P''\lto \fX_m^{p^{\bfr}}\lto\cB G''$ factors through the neutral section $\Spec k\lto \cB G''$ and the restriction $\calG\times_{\fX_m^{p^{\bfr}}} P''$ is isomorphic to the trivial gerbe $\cB G'\times_k P''$ over $P''$.
\begin{equation*}
\begin{xy}
\xymatrix{\ar@{}[rd]|{\square}
\cB G'\times_k P''\ar[r]\ar[d]&\calG\ar[d]\\
P''\ar[r]& \fX_m^{p^{\bfr}}.
}
\end{xy}
\end{equation*}
Therefore, it turns out that $\calG\lto \fX_m^{p^{\bfr}}$ is an $N$-gerbe over $\fX_m^{p^{\bfr}}$, where $N\Def P''\wedge^{G''}G'$ is a finite flat abelian $\fX_m^{p^{\bfr}}$-group scheme. Here the action of $G''$ on $G'$ is given by the conjugation, which is well-defined because $G'$ is abelian. As $G'$ and its Cartier dual $G'^D$ are of height one, the same is true for the twist $N$, i.e.\ both $N$ and $N^D$ are of height one as well. Therefore, by Lemma \ref{lem:vanish H^2}(1), we have $H^2_{\fppf}(\fX_m^{p^{\bfr}},N)=0$, hence $\calG\simeq\cB N$. This completes the proof. 
\end{proof}

\begin{proof}[Proof of Theorem \ref{thm:PISAC solv}]
Thanks to Lemma \ref{lem:tors rs}, we have only to solve the embedding problem
\begin{equation*}
\begin{xy}
\xymatrix{
&&&\pi^{\loc}(\fX_m^{\mathbf{p}^{\infty}})\ar@{->>}[d]^{\overline{\phi}}&\\
1\ar[r]&G'\ar[r]&G\ar[r]^{\pi}&G''\ar[r]&1.
}
\end{xy}
\end{equation*}
By virtue of Lemmas \ref{lem:H^1 mu_p} and \ref{lem:PISAC solv}, a similar argument in the proof of \cite[Proposition 3.4]{ot17} can work even after replacing $U$ by the pro-system of root stacks $\{\fX_m^{p^{\bfr}}\}_{\bfr}$. Recall that, thanks to Corollary \ref{cor:loc ger neutral}, for any finite local $k$-group scheme $G$, there exists a canonical bijective map
\begin{equation*}
\Hom(\pi^{\loc}(\fX_m^{p^{\bfr}}),G)\xrightarrow{~\simeq~}H^1_{\fppf}(\fX_m^{p^{\bfr}},G). 
\end{equation*}
of pointed sets. 

First we assume that $G'$ is abelian, in which case, thanks to Lemma \ref{lem:PISAC solv}, without loss of generality, we may assume that the extension (\ref{eq:thm PISAC solv}) is split~(cf.~\cite[Appendix, 1), a) and b)]{pop}). Again by considering the decomposition $G'=G'_1\times G'_2$ into the direct product of a unipotent group scheme and a diagonalizable group scheme, we may assume that $G'$ is either unipotent or diagonalizable. In the latter case, as $\uAut(G')$ is \'etale, we get  $G=G'\times G''$. Moreover, as $G'$ is a successive extension of $\mu_p$, we are reduced to the case when $G'=\mu_p$, in which case we have the   inequalities 
\begin{equation*}
\Dim_{\F_p}\Hom_k(G'',\mu_p)<\Dim_{\F_p}\Hom_k(G,\mu_p)\le \gamma+n-1,
\end{equation*}
where the second one is due to the condition (i). Therefore, by Lemma \ref{lem:H^1 mu_p}, one can see that the inclusion 
\begin{equation*}
H^1_{\fppf}(\fX_m^{p^{\bfr}},\mu_p)\supset\Ker(\Hom_k(\pi^{\loc}(\fX_m^{p^{\bfr}}),\mu_p)\to\Hom_k(K_{\bfr},\mu_p))\xleftarrow{~\simeq~}\Hom_k(G'',\mu_p),
\end{equation*}
is strict for sufficiently large $\bfr$, where we set $K_{\bfr}\Def\Ker(\pi^{\loc}(\fX_m^{p^{\bfr}})\to G'')$ for any index $\bfr$ for which $\overline{\phi}$ factors through $\pi^{\loc}(\fX_m^{p^{\bfr}})$~(cf.\ (\ref{eq:pi loc inf root stack})). Thus, one can take an element
\begin{equation*}
g\in\Hom(\pi^{\loc}(\fX_m^{p^{\bfr}}),\mu_p)\setminus \Ker(\Hom_k(\pi^{\loc}(\fX_m^{p^{\bfr}}),\mu_p)\to\Hom_k(K_{\bfr},\mu_p))
\end{equation*}
for some $\bfr$ so that the homomorphism $(g,\overline{\phi}):\pi^{\loc}(\fX_m^{p^{\bfr}})\lto \mu_p\times G''=G$ gives rise to a surjective lifting $\phi:\pi^{\loc}(\fX_m^{\mathbf{p}^{\infty}})\twoheadrightarrow G$ of $\overline{\phi}$. This completes the proof in the case where $G'$ is diagonalizable.

Let us assume that $G'$ is abelian and unipotent. As the Frobenius kernels $\Ker(F^{(i)}:G'\lto G'^{(i)})\subseteq  G'~(i\ge 1)$ are stable under any automorphism of $G'$, they are normal in $G$. Thus, by taking the quotients by them, we are reduced to the case where $G'$ is of height one. Similarly for the Frobenius kernels of the Cartier dual $G'^D$ of $G'$, and thus we may further assume that the Cartier dual $G'^D$ is of height one as well. Recall that such a $G'$ must be isomorphic to a direct sum of $\alpha_p$.
Namely, we are reduced to the case where $G'=\alpha_p^t$ for some integer $t>0$. Moreover, without loss of generality, we may assume that the conjugacy action $G''\lto\uAut(G')$ is irreducible, i.e.\ $G'$ does not contain a nontrivial strict subgroup scheme $1\neq H\subsetneq G'$ which is stable under the conjugacy action by $G''$. Now we use the same notation as in the proof of Lemma \ref{lem:PISAC solv}. Suppose $\overline{\phi}$ factors through $\pi^{\loc}(\fX_m^{p^{\bfr}})$ for some $\bfr$ and let $P''\lto\fX_m^{p^{\bfr}}$ be the $G''$-torsor associated with the resulting homomorphism $\overline{\psi}:\pi^{\loc}(\fX_m^{p^{\bfr}})\twoheadrightarrow G''$. Fix a section $\sigma_0:G''\lto G$ of $\pi$. Put $N=P''\wedge^{G''} G'$. Then, the liftings of the given surjective homomorphism $\overline{\psi}$ are completely parametrized by the abelian group $H^1_{\fppf}(\fX_m^{p^{\bfr}},N)$~(cf.\ Remark \ref{rem:pf PIAC solv}).  
\begin{equation*}
\begin{xy}
\xymatrix{\ar@{}[rd]|{\square}
\cB G'\ar[r]\ar[d]&\ar@{}[rd]|{\square}\cB (G''\ltimes G')\ar[d]&\cB N\ar[d]\ar[l]\\
\Spec k\ar[r]&\cB G''\ar@/^3mm/[u]^{\sigma_{0*}}&\fX_m^{p^{\bfr}}\ar[l]^{~P''}.
}
\end{xy}
\end{equation*}
Let $\Hom_{G''}(G'',G)$ be the set of sections of the surjective homomorphism $\pi:G\twoheadrightarrow G''$, i.e.\ 
\begin{equation*}
\Hom_{G''}(G'',G)\Def\{\sigma\in\Hom(G'',G)\,|\,\pi\circ\sigma={\rm id}_{G''}\}.
\end{equation*}
Moreover, set
\begin{equation*}
P''\wedge\Hom_{G''}(G'',G)~\Def~\{P''\wedge^{G''}{}^{\sigma}{G}~|~\sigma\in\Hom_{G''}(G'',G)\}\subset H^1_{\fppf}(\fX_m^{p^{\bfr}},G).
\end{equation*} 
Note that
\begin{equation*}
P''\wedge\Hom_{G''}(G'',G)~\subseteq~ H^1_{\fppf}(\fX_m^{p^{\bfr}},N)~\subseteq~ H^1_{\fppf}(\fX_m^{p^{\bfr}},G).
\end{equation*}
Since $G'$ is irreducible under the conjugacy action of $G''$, for a lifting $\psi\in H^1_{\fppf}(\fX^{p^{\bfr}}_m,N)$ of $\overline{\psi}$, if $\im(\psi)\cap G'\neq 1$, then the $G''$-orbit of $\im(\psi)\cap G'$ generates $G'$. However, as $G'$ is abelian, the $G''=\im(\overline{\psi})$-orbit coincides with the $\im(\psi)$-orbit of $\im(\psi)\cap G'$ by the conjugacy action. 
Therefore, the complement
\begin{equation*}
H^1_{\fppf}(\fX_m^{p^{\bfr}},N)\setminus (P''\wedge\Hom_{G''}(G'',G))
\end{equation*}
exactly consists of surjective liftings of $\overline{\psi}$. Therefore, we have to show that the inclusion is strict, i.e.
\begin{equation*}
P''\wedge\Hom_{G''}(G'',G)~\subsetneq~H^1_{\fppf}(\fX_m^{p^{\bfr}},N).
\end{equation*}
We denote by $Z(G'',G')$ the set of crossed homomorphisms $G''\lto G'$ with respect to the conjugacy action $G''\lto\uAut_k(G')$. 
Namely, each element $z\in Z(G'',G')$ is a morphism $z:G''\lto G'$ of $k$-schemes satisfying the condition that
\begin{equation*}
z(g_1g_2)=z(g_1)+g_1^{-1}z(g_2)g_1
\end{equation*}
for any sections $g_1,g_2\in G''$, where the multiplication of $G'$ is written additively. 
Note that $Z(G'',G')$ is a subset of the group $G'(k[G''])$ of $k[G'']$-valued points of $G'$. As $G'=\alpha_p^{\oplus t}$, the group $G'(k[G''])=\alpha_p(k[G''])^{\oplus t}$ has a natural structure of $k$-vector space and $Z(G'',G')$ becomes a $k$-subspace of $G'(k[G''])$. As both the group schemes $G'$ and $G''$ are finite, the $k$-vector space $Z(G'',G')$ is finite dimensional. Moreover, the map
\begin{equation*}
\Hom_{G''}(G'',G)\lto Z(G'',G')~;~\sigma\longmapsto \sigma\cdot \sigma_0^{-1}
\end{equation*}
is bijective with the inverse map $z\mapsto z\cdot \sigma_0$, and the composition
\begin{equation*}
Z(G'',G')\xrightarrow[\simeq]{z\mapsto z\cdot\sigma_0}P''\wedge\Hom_{G''}(G'',G)\hookrightarrow H^1_{\fppf}(\fX_m^{p^{\bfr}},N)
\end{equation*} 
is a $k$-linear map. By Lemma \ref{prop:inf alpha p}, we have
\begin{equation*}
\Dim_k Z(G'',G')<\infty=\Dim_k H^1_{\fppf}(\fX_m^{p^{\bfr}},N),
\end{equation*}
hence $P''\wedge\Hom_{G''}(G'',G)\neq H^1_{\fppf}(\fX_m^{p^{\bfr}},N)$. Therefore, $\overline{\psi}$ has a surjective lifting onto $G$. This implies the theorem is true if $G'$ is abelian.

Let us prove the theorem in the general case by induction on the order of $G'$. As $G'$ is solvable, the derived subgroup scheme $D(G')$ is a strict subgroup scheme of $G'$, i.e.\ $D(G')\subsetneq G'$. Since $D(G')$ is also normal in $G$, we get a commutative diagram of group schemes with exact rows and columns
\begin{equation*}
\begin{xy}
\xymatrix{
&1\ar[d]&1\ar[d]&&\\
&D(G')\ar@{=}[r]\ar[d]&D(G')\ar[d]&&\\
1\ar[r]&G'\ar[r]\ar[d]&G\ar[r]\ar[d]&G''\ar[r]\ar@{=}[d]&1\\
1\ar[r]&G'^{\rm ab}\ar[r]\ar[d]&G/D(G')\ar[r]\ar[d]&G''\ar[r]&1.\\
&1&1&&
}
\end{xy}
\end{equation*}
As we have already seen that the theorem is true for $G'^{\rm ab}$, we can take a surjective homomorphism $\pi^{\loc}(\fX_m^{\mathbf{p}^{\infty}})\twoheadrightarrow G/D(G')$. As the order of $D(G')$ is strictly less than the order of $G'$, by the induction hypothesis, the theorem is true for the exact sequence
\begin{equation*}
1\lto D(G')\lto G\lto G/D(G')\lto 1.
\end{equation*}
Therefore, there exist an $m$-tuple $\bfr$ and a Nori-reduced $G$-torsor over $\fX_m^{p^{\bfr}}$ which is representable by a $k$-scheme. This completes the proof of the theorem. 
\end{proof}

\begin{rem}\label{rem:pf PIAC solv}
Let
\begin{equation*}
1\lto G'\lto G\lto G''\lto 1
\end{equation*}
be an exact sequence of finite local $k$-group schemes. Let $\calX$ be a reduced algebraic stack over $k$ which admits the local fundamental group scheme $\pi^{\loc}(\calX)$. Let $\overline{\phi}:\pi^{\loc}(\calX)\lto G''$ be a homomorphism of $k$-group schemes. Let $P''$ denote a $G''$-torsor over $\calX$ which corresponds to the homomorphism $\overline{\phi}$ via the bijection in Corollary \ref{cor:loc ger neutral}. Then, the set of liftings $\phi:\pi^{\loc}(\calX)\lto G$ of $\overline{\phi}$ is naturally bijective into the set
\begin{equation*}
S(\overline{\phi})\Def \{[P]\in H^1_{\fppf}(\calX,G)\,;\,[P\wedge^G G'']=[P'']~\text{in $H^1_{\fppf}(\calX,G'')$}\}.
\end{equation*} 
Recall that for any finite local $k$-group scheme $H$, the category $\Hom_k(\calX,\cB H)$ is a setoid~(cf.\ Proposition \ref{prop:loc ger neutral}), and it is actually equivalent to the set $H^1_{\fppf}(\calX,H)$. In particular, the isomorphism class $[P'']$ in $H^1_{\fppf}(\calX,G'')$ gives rise to a unique morphism $\xi'':\calX\lto\cB G''$ up to unique isomorphism. We set
\begin{equation*}
\calG\Def \cB G\times_{\cB G'',\xi''}\calX,
\end{equation*}
which is a gerbe over $\calX$. By forgetting isomorphisms into $\xi''$ in $\Hom(\calX,\cB G'')$, we obtain a natural map $\Hom_{\calX}(\calX,\calG)_{/{\simeq}}\lto S(\overline{\phi})$. However, as $\Hom(\calX,\cB G'')$ is a setoid, this map has to be bijective~(cf.\ \cite[Lemma 04SD]{stack}). 
\end{rem}

\begin{cor}\label{cor:PIAC solv}
Suppose given an exact sequence of finite local $k$-group schemes
\begin{equation*}
1\lto G'\lto G\lto G''\lto 1.
\end{equation*}
Suppose that the following conditions are satisfied.
\begin{enumerate}
\renewcommand{\labelenumi}{(\roman{enumi})}
\item There exists an injective homomorphism $\X(G)\hookrightarrow(\Q_p/\Z_p)^{\oplus\gamma+n-1}$.
\item There exists a surjective $k$-homomorphism $\pi^{\loc}(U)\twoheadrightarrow G''$. 
\item $G'$ is solvable.
\end{enumerate}
Then there exists a Nori-reduced $G$-torsor over $U$. In particular, for any finite local solvable $k$-group scheme $G$, the group scheme $G$ appears as a quotient of $\pi^{\loc}(U)$ if and only if the character group $\X(G)$ can be embedded into the abelian group $(\Q_p/\Z_p)^{\oplus\gamma+n-1}$.
\end{cor}

\begin{proof}
This is nothing other than Theorem \ref{thm:PISAC solv} specialized in the case when $m=n-1$. 
\end{proof}

\begin{cor}\label{cor:PISAC solv}
Let $G$ be a finite local abelian $k$-group scheme. Suppose that there exists an injective homomorphism $\X(G)\hookrightarrow(\Q_p/\Z_p)^{\oplus \gamma+n-1}$. Then there exists an $(n-1)$-tuple $\bfr$ of integers $r_i\ge 0$ and a tamely ramified $G$-torsor $Y\lto X_0$ with ramification data $(\bfD,p^{\bfr})$ such that the restriction $Y\times_{X_0} U\lto U$ gives a Nori-reduced $G$-torsor.  
\end{cor}

\begin{proof}
This is immediate from Theorems \ref{thm:biswas-borne}(2) and  \ref{thm:PISAC solv}.
\end{proof}

\appendix

\section{On local-to-global extensions of finite local  torsors}\label{sec:local global}

Let $k$ be an algebraically closed field of characteristic $p>0$. We denote by $t$ the coordinate of the affine line $\A_k^1$, i.e.\ $\A_k^1=\Spec k[t]$. Let $\fD\Def\Spec k[[t^{-1}]]$ be the formal disk and $\eta\Def\Spec k((t^{-1}))$ the formal punctured disk. Note that there exist natural morphisms $\eta\lto\A^1_k$ and $\eta\lto \fD$. 

\begin{definition}
Let $G$ be a finite group. An \'etale Galois $G$-cover  $f:Y\lto\G_{m,k}=\Spec k[t,t^{-1}]$ is said to be \textit{special} if the following conditions are satisfied.
\begin{enumerate}
\renewcommand{\labelenumi}{(\roman{enumi})}
\item $f$ is tamely ramified above $t=0$.
\item The monodromy group of $f$, i.e.\ the image of the homomorphism $\pi_1^{\et}(\G_m,1)\lto G$ which corresponds to $f$ (up to conjugacy) has a unique $p$-Sylow subgroup.   
\end{enumerate}
\end{definition}

Then the \textit{Katz--Gabber correspondence} for finite \'etale coverings can be stated as follows.

\begin{thm}(Katz--Gabber, cf.~\cite{ka86})\label{thm:katz-gabber}
Let $G$ be a finite group. 
The morphism $\eta\lto\G_{m,k}$ induces an equivalence of categories between the category of special $G$-covers of $\G_{m,k}$ and the category of $G$-covers of $\eta$. 
\end{thm}

In particular, the theorem includes the following fact.

\begin{prop}\label{prop:katz-gabber}
The morphism $\eta\lto\A^1_k$ induces an isomorphism of pro-$p$-groups
\begin{equation*}
\Gal\bigl(\overline{k((t^{-1}))}/k(((t^{-1}))\bigl)^{(p)}\xrightarrow{~\simeq~}\pi_1^{\et}(\A^1_k,1)^{(p)}, 
\end{equation*}
where $\overline{k((t^{-1}))}$ is a separable closure of $k((t^{-1}))$ and $(-)^{(p)}$ means the maximal pro-$p$-quotient. 
\end{prop}

\begin{ex}\label{ex:katz-gabber}
Let us consider $G=\F_p$. The natural inclusion $k[t]\hookrightarrow k((t^{-1}))$ induces an isomorphism
\begin{equation*}
H^1_{\et}(\A^1_k,\F_p)=k[t]/\mathscr{P}(k[t])\xrightarrow{~\simeq~}k((t^{-1}))/\mathscr{P}(k((t^{-1})))=H^1_{\et}(\eta,\F_p),
\end{equation*}
where $\mathscr{P}$ is the $\F_p$-linear map $f\longmapsto f-f^p$. This is valid because
\begin{equation*}
k((t^{-1}))=k[t]\oplus t^{-1}k[[t^{-1}]]~~\text{and}~~\mathscr{P}(t^{-1}k[[t^{-1}]])=t^{-1}k[[t^{-1}]], 
\end{equation*}
where the second equation can be seen as follows. For an arbitrary element $f\in t^{-1}k[[t^{-1}]]$, we have
\begin{equation*}
f=\sum_{i=0}^{\infty}\mathscr{P}(f^{p^i})=\mathscr{P}(\sum_{i=0}^{\infty}f^{p^i})
\end{equation*}
with $\sum_{i=0}^{\infty}f^{p^i}\in t^{-1}k[[t^{-1}]]$, hence $f\in\mathscr{P}(t^{-1}k[[t^{-1}]])$.  
\end{ex}

Let us consider what happens for finite local torsors. 
To ease of notation, for a reduced inflexible algebraic stack $\calX$ over a field $k$, we write $\varpi(\calX)$ for the local fundamental group scheme $\pi^{\loc}(\calX)$ of $\calX$. 
Since the natural inclusion $\eta\lto \fD$ is an open immersion and $\fD$ is normal, according to \cite[Chapter II \S2]{no82}, the natural homomorphism $\varpi(\eta)\lto\varpi(\fD)$ is surjective.

\begin{lem}\label{lem:ab varpi eta}
We have the following.
\begin{enumerate}
\renewcommand{\labelenumi}{(\arabic{enumi})}
\item For any integer $m>0$, the composition of natural maps
\begin{equation*}
H^1_{\fppf}(\G_m,\mu_{p^m})\hookrightarrow H^1_{\fppf}(\eta,\mu_{p^m})\twoheadrightarrow H^1_{\fppf}(\eta,\mu_{p^m})/H^1_{\fppf}(\fD,\mu_{p^m})
\end{equation*} 
is an isomorphism of abelian groups.
\item The composition of natural maps
\begin{equation*}
H^1_{\fppf}(\A_k^1,\alpha_{p})\hookrightarrow H^1_{\fppf}(\eta,\alpha_{p})\twoheadrightarrow H^1_{\fppf}(\eta,\alpha_{p})/H^1_{\fppf}(\fD,\alpha_{p})
\end{equation*} 
is an isomorphism of abelian groups.
\end{enumerate}
\end{lem}

\begin{proof}
(1) This follows from the descriptions of cohomology groups
\begin{equation*}
\begin{aligned}
H_{\fppf}^1(\eta,\mu_{p^m})
~&\simeq~ k((t^{-1}))^{\times}/k((t^{-1}))^{\times p^m}\\
~&\simeq~ t^{-\Z}/t^{-p^m\Z}\times k[[t^{-1}]]^{\times}/k[[t^{-1}]]^{\times p^m},\\
H_{\fppf}^1(\fD,\mu_{p^m})~&\simeq~ k[[t^{-1}]]^{\times}/k[[t^{-1}]]^{\times p^m},\\
H^1_{\fppf}(\G_m,\mu_{p^m})~&\simeq~ t^{-\Z}/t^{-p^m\Z}.
\end{aligned}
\end{equation*}
(2) This follows from the descriptions of cohomology groups
\begin{equation*}
\begin{aligned}
H_{\fppf}^1(\eta,\alpha_{p})
~&\simeq~ k((t^{-1}))/k((t^{-1}))^{p}\\
~&\simeq~ k[t]/k[t]^p \oplus k[[t^{-1}]]/k[[t^{-1}]]^{p},\\
H^1_{\fppf}(\fD,\alpha_{p})~&\simeq~ k[[t^{-1}]]/k[[t^{-1}]]^{p},\\
H^1_{\fppf}(\A^1_k,\alpha_p)~&\simeq~k[t]/k[t]^p.
\end{aligned}
\end{equation*}
\end{proof}

\begin{prop}\label{prop:ab varpi eta}
There exists a canonical isomorphism of affine $k$-group schemes
\begin{equation*}
\varpi^{\mathrm{ab}}(\eta)\xrightarrow{~\simeq~}\varpi^{\mathrm{ab}}(\fD)\times\varpi^{\mathrm{ab}}\bigl(\sqrt[p^{\infty}]{0/\A^1_k}\bigl),
\end{equation*}
where $\varpi(\sqrt[p^{\infty}]{0/\A^1_k})$ is defined to be 
\begin{equation*}
\varpi\bigl(\sqrt[p^{\infty}]{0/\A^1_k}\bigl)~\Def~\varprojlim_{r>0}\varpi\bigl(\sqrt[p^r]{0/\A^1_k}\bigl). 
\end{equation*}
\end{prop}

\begin{proof}
Since we have natural $k$-homomorphisms $\varpi(\eta)\lto\varpi(\fD)$ and $\varpi(\eta)\lto\varpi\bigl(\sqrt[p^{\infty}]{0/\A_k^1}\bigl)$, there exists a canonical homomorphism
\begin{equation*}
\varpi(\eta)\lto\varpi(\fD)\times\varpi\bigl(\sqrt[p^{\infty}]{0/\A_k^1}\bigl).
\end{equation*}
We shall show that this homomorphism induces the desired isomorphism. 
Since we have
\begin{equation*}
\varpi^{\ab}\bigl(\fD\bigsqcup\sqrt[p^{\infty}]{0/\A^1_k}\bigl)=\varpi^{\ab}(\fD)\times\varpi^{\ab}\bigl(\sqrt[p^{\infty}]{0/\A^1_k}\bigl),
\end{equation*}
by the universal property of the local fundamental group scheme, it suffices to show that for any abelian local finite $k$-group scheme $G$, 
the induced map
\begin{equation*}
H^1_{\fppf}\bigl(\fD\bigsqcup\sqrt[p^{\infty}]{0/\A_k^1},G\bigl)=H^1_{\fppf}(\fD,G)\oplus H^1_{\fppf}\bigl(\sqrt[p^{\infty}]{0/\A^1_k},G\bigl)\lto H^1_{\fppf}(\eta,G)
\end{equation*}
is an isomorphism. By virtue of Lemma \ref{lem:vanish H^2}, we may assume that $G$ is isomorphic to $\mu_p$ or $\alpha_p$~(cf.~\cite[p. 284, Remarque 4.2.10]{gi71}). Since there exist canonical isomorphisms
\begin{equation*}
\begin{aligned}
&H^1_{\fppf}\bigl(\sqrt[p^{\infty}]{0/\A_k^1},\mu_p\bigl)\xrightarrow{~\simeq~}H^1_{\fppf}(\G_m,\mu_p),\\
&H^1_{\fppf}(\A_k^1,\alpha_p)\xrightarrow{~\simeq~}H_{\fppf}^1\bigl(\sqrt[p^{\infty}]{0/\A_k^1},\alpha_p\bigl),
\end{aligned}
\end{equation*}
(cf.\ Lemma \ref{lem:H^1 mu_p} and Proposition \ref{prop:Nori ger root stack} respectively), 
the claim is a consequence of Lemma \ref{lem:ab varpi eta}.
\end{proof}

\begin{cor}\label{cor:ab varpi eta}
There exists a natural isomorphism
\begin{equation*}
\Ker(\varpi^{\ab}(\eta)\lto\varpi^{\ab}(\fD))~\xrightarrow{~\simeq~}~\varpi^{\ab}\bigl(\sqrt[p^{\infty}]{0/\A_k^1}\bigl).
\end{equation*} 
\end{cor}

\section{Frobenius divided sheaves on root stacks}\label{sec:Nori ger root stack appendix}

The aim of this appendix is to complete the proof of Proposition \ref{prop:Nori ger root stack}, namely to prove the proposition without the smoothness assumption. 
We will use the same notation as in \S\ref{subsec:Nori ger root stack}. Let $X$ be a geometrically connected and geometrically reduced scheme of finite type over the spectrum $S=\Spec k$ of a perfect field $k$ of characteristic $p>0$. Let $\bfD=(D_i)_{i\in I}$ be a finite family of reduced irreducible effective Cartier divisors on $X$ and put  $D=\cup_{i\in I}D_i\subset X$. For each $\bfr=(r_i)_{i\in I}$ with $r_i>0$, as in \S\ref{subsec:root stack}, we put
\begin{equation*}
\fX^{\bfr}=\sqrt[\bfr]{\bfD/X}.
\end{equation*} 
By definition of root stacks, the $n$th Frobenius twist commutes with the root construction, 
\begin{equation*}
\bigl(\sqrt[\bfr]{\bfD/X}\bigl)^{(n)}=\sqrt[\bfr]{\bfD^{(n)}/X^{(n)}}.
\end{equation*}

\begin{prop}\label{prop:Fdiv root stack}
Let $\bfr$ and $\bfr'$ be two indices with $\bfr\mid\bfr'$. With the same notation as in \S\ref{sec:Fdiv}, we have the following. 
\begin{enumerate}
\renewcommand{\labelenumi}{(\arabic{enumi})}
\item The functor $\pi^*:\Fdiv(\fX^{\bfr})\longrightarrow\Fdiv(\fX^{\bfr'})$ is fully faithful and the essential image consists of all the Frobenius divided sheaves $\calE=(\calE_i,\sigma_i)_{i=0}^{\infty}$ with $\pi^*\pi_*(\calE_i|_{\calG'^{(i)}_{\xi}})\xrightarrow{~\simeq~}\calE_i|_{\calG'^{(i)}_{\xi}}$ for any $i$ and for any closed point $\xi$ of $\fX^{\bfr'}$, where $\calG'_{\xi}$ denotes the residual gerbe of $\fX^{\bfr'}$ at $\xi$~(cf.~Proposition \ref{prop:root stack D}(2)).
\item The functor $\pi^*:\Fdiv_{\infty}(\fX^{\bfr})\longrightarrow\Fdiv_{\infty}(\fX^{\bfr'})$ is fully faithful and the essential image consists of all the objects $(\calF,\calE,\lambda)$ of $\Fdiv_{\infty}(\fX^{\bfr'})$ with $\pi^*\pi_*(\calF|_{\calG'_{\xi}})\xrightarrow{~\simeq~}\calF|_{\calG'_{\xi}}$ and $\pi^*\pi_*(\calE_i|_{\calG'^{(i)}_{\xi}})\xrightarrow{~\simeq~}\calE_i|_{\calG'^{(i)}_{\xi}}$ for any $i$ and for any closed point $\xi$ of $\fX^{\bfr'}$.
\item The induced morphism of tannakian gerbes $\Pi_{\Fdiv_{\infty}(\fX^{\bfr'})}\lto\Pi_{\Fdiv_{\infty}(\fX^{\bfr})}$ is a gerbe. 
\end{enumerate}
\end{prop}

\begin{proof}
(1) For the full faithfulness, let us begin with an observation. Let $\calE_i~(i=1,2)$ be two vector bundles on $\fX^{\bfr}$ which admit Frobenius descents, i.e.\ there exist vector bundles $\calE^{(1)}_i$ on $\fX^{\bfr(1)}$ together with isomorphisms $\sigma_i:F^{(1)*}\calE_i^{(1)}\simeq\calE_i$. Suppose given an $\scrO_{\fX^{\bfr}}$-linear map $\alpha:\calE_1\lto\calE_2$ and an $\scrO_{\fX^{\bfr(1)}}$-linear map $\beta:\calE_1^{(1)}\lto\calE_2^{(1)}$. Then, by Proposition \ref{prop:root stack qcoh}(3), the diagram
\begin{equation*}
\begin{xy}
\xymatrix{
F^{(1)*}\calE_1^{(1)}\ar[r]^{~~~\sigma_1}\ar[d]_{F^{(1)*}\beta}&\calE_1\ar[d]^{\alpha}\\
F^{(1)*}\calE_2^{(1)}\ar[r]_{~~~\sigma_2}&\calE_2
}
\end{xy}
\end{equation*}
is commutative if and only if it is commutative after applying the pullback functor $\pi^*$. Therefore, the full faithfulness of the functor $\pi^*:\Fdiv(\fX^{\bfr})\lto\Fdiv(\fX^{\bfr'})$ follows from this observation together with Proposition \ref{prop:root stack qcoh}(3).

For the description of the essential image, by virtue of  Proposition \ref{prop:rel alper}, it suffices to show that for any vector bundle $\calE$ on $\fX^{\bfr}$, if $\pi^*\calE$ admits a Frobenius descent which is of the form $\pi^{(1)*}\calE^{(1)}$ for some vector bundle $\calE^{(1)}$ on $\fX^{\bfr(1)}$, the isomorphism $\sigma:F^{(1)*}\pi^{(1)*}\calE^{(1)}\simeq\pi^{*}\calE$ defines a canonical isomorphism $F^{(1)*}\calE^{(1)}\simeq\calE$. However, since $F^{(1)*}\pi^{(1)*}\calE^{(1)}=\pi^*F^{(1)*}\calE^{(1)}$, the pushforward $\pi_*\sigma$ defines a desired isomorphism $F^{(1)*}\calE^{(1)}\simeq\calE$.

(2) Similarly, the assertion is a consequence of Propositions \ref{prop:root stack qcoh}(3) and \ref{prop:rel alper}. Let us begin with showing the full faithfulness. Fix an arbitrary integer $j>0$. 
Let $\calF_i~(i=1,2)$ be two vector bundles on $\fX^{\bfr}$ such that there exist $F$-divided sheaves $\calG_{i}=\{\calE^{(n)}_i,\sigma_{i}^{(n)}\}_{n=0}^{\infty}$ on $\fX^{\bfr}$ together with isomorphisms $\lambda_i:F^{j*}\calF_i\xrightarrow{~\simeq~}\calG_i|_{\fX^{\bfr}}=\calE_i^{(0)}$. Suppose given an $\scrO_{\fX^{\bfr}}$-linear map $\alpha:\calF_1\lto\calF_2$ and a morphism $\beta=\{\beta^{(n)}\}_{n=0}^{\infty}:\calG_1\lto\calG_2$ of $F$-divided sheaves. Then, by Proposition \ref{prop:root stack qcoh}(3), the diagram
\begin{equation*}
\begin{xy}
\xymatrix{
F^{j*}\calF_1\ar[r]^{~~\lambda_1}\ar[d]_{F^{j*}\alpha}&\calE_1^{(0)}\ar[d]^{\beta^{(0)}}\\
F^{j*}\calF_2\ar[r]_{~~\lambda_2}&\calE_2^{(0)}
}
\end{xy}
\end{equation*}
is commutative if and only if it is commutative after applying the pullback functor $\pi^*$.
From the definition of $\Fdiv_{\infty}(-)$ together with the full faithfulness of the functor $\pi^*:\Fdiv(\fX^{\bfr})\lto\Fdiv(\fX^{\bfr'})$~(cf.~(1)), the observation implies that the functor $\Fdiv_{\infty}(\fX^{\bfr})\lto\Fdiv_{\infty}(\fX^{\bfr'})$ is fully faithful.   

Let us discuss on the description of the essential image. Fix an arbitrary integer $j>0$. Let  $(\calF,\calG,\lambda)$ be an object of $\Fdiv_{j}(\fX^{\bfr'})$ satisfying $\pi^*\pi_*(\calF|_{\calG'_{\xi}})\xrightarrow{~\simeq~}\calF|_{\calG'_{\xi}}$ and $\pi^*\pi_*(\calE_i|_{\calG'^{(i)}_{\xi}})\xrightarrow{~\simeq~}\calE_i|_{\calG'^{(i)}_{\xi}}$ for any $i$ and for any closed point $\xi$ of $\fX^{\bfr'}$. From Proposition \ref{prop:rel alper}, we have $\pi^*\pi_*\calF\xrightarrow{~\simeq~}\calF$ in $\Vect(\fX^{\bfr'})$. On the other hand, by (1), we also have $\pi^*\pi_*\calG\xrightarrow{~\simeq~}\calG$ in $\Fdiv(\fX^{\bfr'})$. The isomorphism $\lambda:F^{j*}\calF\xrightarrow{~\simeq~}\calG|_{\fX^{\bfr'}}$ then  descends to a one $\pi_*\lambda:F^{j*}\pi_*\calF\xrightarrow{~\simeq~}(\pi_*\calG)|_{\fX^{\bfr}}$, which implies that $(\calF,\calG,\lambda)$ descends to an object of $\Fdiv_j(\fX^{\bfr})$. This completes the proof.  

(3) This follows from (2) together with the last assertion in Proposition \ref{prop:rel alper}(3).  
\end{proof}

\begin{proof}[Proof of Proposition \ref{prop:Nori ger root stack}]
We have already seen the second assertion~(cf.~Proof of Proposition \ref{prop:Nori ger root stack}). For the first assertion, as we saw in the proof of Proposition \ref{prop:Nori ger root stack}, we have only to prove that the restriction functor
\begin{equation*}
\Vect(\Pi^{\N}_{\fX^{\bfr}/k})\lto\Vect(\Pi^{\N}_{\fX^{\bfr'}/k})
\end{equation*}
is fully faithful. However, by virtue of Corollary \ref{cor:Fdiv}(2), the full faithfulness is an immediate consequence of Proposition \ref{prop:Fdiv root stack}.
\end{proof}

\section{Infiniteness of the first fppf cohomology group with coefficient $\alpha_p$}\label{sec:inf alpha_p}

\begin{lem}\label{lem:inf alpha p}
Let $U$ be an affine smooth geometrically connected curve over a perfect field $k$ of characteristic $p>0$, $\calX$ an algebraic stack over $k$ and $f:\calX\lto U$ a proper quasi-finite flat generically \'etale morphism. Let $E$ be a locally free $\scrO_{\calX}$-module of finite rank. Let
\begin{equation*}
\phi:E\lto E^{(1)}
\end{equation*}
be the natural $p$th power map, where $E^{(1)}$ is the pullback of $E$ along the absolute Frobenius morphism $F:\calX\lto\calX$. Then the $k$-vector space $\Gamma(\calX,E^{(1)})/\Gamma(\calX,E)^p$ is of infinite dimension.  
\end{lem}

\begin{proof}
By taking a finite separable surjective $k$-morphism $g:U\lto U_0$ between smooth affine curves with $U_0$ an open subscheme of the projective line $\mathbb{P}^1_k$. Since $g$ is generically \'etale, by replacing $f$ with $g\circ f$, we are reduced to the case when $U$ is an open dense subscheme of the projective line $\mathbb{P}^1_k$. 
We are reduced to showing that the $k$-vector space $\Gamma(U,f_*(E^{(1)}))/\Gamma(U,f_*E)^p$ is of infinite dimension. Since $f$ is generically \'etale, the commutative diagram
\begin{equation*}
\begin{xy}
\xymatrix{
\calX\ar[r]^F\ar[d]_{f}&\calX\ar[d]^{f}\\
U\ar[r]^{F}&U
}
\end{xy}
\end{equation*}
is generically Cartesian. Since $F$ is a flat morphism on the smooth curve $U$, it follows that the canonical base change map
\begin{equation*}
(f_*E)^{(1)}\lto f_*(E^{(1)})
\end{equation*}
is a generically isomorphism of torsion-free coherent sheaves, hence its kernel is trivial and its cokernel is a torsion coherent sheaf over $U$. Therefore, $\Gamma(U,f_*(E^{(1)}))/\Gamma(U,f_*E)^p$ is of infinite dimension if and only if $\Gamma(U,(f_*E)^{(1)})/\Gamma(U,f_*E)^p$ is of infinite dimension. However, as $U$ is an affine dense open subscheme of $\mathbb{P}^1_k$, any torsion-free coherent sheaf is a free $\scrO_U$-module, and $\Gamma(U,\scrO_U)/\Gamma(U,\scrO_U)^p$ is of infinite dimension. Therefore, $\Gamma(U,(f_*E)^{(1)})/\Gamma(U,f_*E)^p$ is of infinite dimension. This completes the proof. 
\end{proof}

As an application, one can see the following result.

\begin{prop}\label{prop:inf alpha p}
Let $U$ be an affine smooth curve over a perfect field $k$ of characteristic $p>0$ with function field $K$. Let $\bfD=(x_i)_{i=1}^m$ be a family of distinct closed points of $U$. Let $\bfr=(r_i)_{i=1}^m\in\prod_{i=1}^m\Z_{\ge 1}$ be a family of integers, we denote by $\fX=\sqrt[\bfr]{\bfD/U}$ the associated root stack. 
Let $N$ be a finite flat abelian $\fX$-group scheme. If both $N$ and its Cartier dual $N^D$ are of height one, then the first cohomology group $H^1_{\fppf}(\fX,N)$ is infinite dimensional over $k$. Furthermore, we have $H^q_{\fppf}(\fX,N)=0$ for $q>1$.  
\end{prop}

\begin{proof}
As $N^D$ is of height one, there exists a locally free $\calO_{\fX}$-module $\omega_{N^D}$ of finite rank together with the exact sequence of fppf abelian sheaves
\begin{equation*}
0\lto N\lto V(\omega_{N^D})\xrightarrow{~\phi~} V(\omega_{N^D})^{(1)}\lto 0
\end{equation*} 
(cf.~\cite[Chapter III, Theorem 5.1]{mil ADT}), where $V(\omega_{N^D})$ is the vector group over $\fX$ associated with $\omega_{N^D}$, which is defined to be
\begin{equation*}
V(\omega_{N^D})(Y)\Def\Hom_{\scrO_{Y}}(\omega_{N^D}\otimes\scrO_{Y},\scrO_Y).
\end{equation*}
for any morphism $Y\lto\fX$ from a scheme $Y$. Moreover, as $N$ is of height one, the homomorphism $\phi:V(\omega_{N^D})\lto V(\omega_{N^D})^{(1)}$ is the $p$th power Frobenius map $\phi=F^{(1)}$ of the vector group $V(\omega_{N^D})$~(cf.~\cite[Chapter III, \S5]{mil ADT}).  
As $V(\omega_{N^D})=W(\omega_{N^D}^{\vee})$, by the same argument as in the proof of Lemma \ref{lem:coho root stack}(1),  we have $H^q_{\fppf}(\fX,V(\omega_{N^D}))=0$ 
for $q>0$ and the same is true for $V(\omega_{N^D})^{(1)}=V(\omega_{N^D}^{(1)})$. Therefore, by considering the long exact sequence associated with the above short exact sequence, we have
\begin{equation*}
H^1_{\fppf}(\fX,N)\simeq\Gamma(\fX,\omega_{N^D}^{\vee (1)})/\Gamma(\fX,\omega_{N^D}^{\vee})^p,
\end{equation*}
which is an infinite dimensional $k$-vector space by Lemma \ref{lem:inf alpha p}, and $H^q_{\fppf}(\fX,N)=0$ for $q>1$. This completes the proof. 
 \end{proof}

\begin{rem}
In the proof of \cite[Proposition 3.4]{ot17}, the infiniteness of the cohomology group $H^1_{\fppf}(U,\alpha_p)$ for an affine smooth geometrically connected curve over $k$ is mentioned, but is wrongly explained. As in Lemma \ref{lem:inf alpha p}, one should have taken a dominant    morphism $U\lto\A^1_k$ which is generically \'etale. 
\end{rem}

%%%%%%%%%%%%%%%%%%%%%%%%%%%%%%%%%%%%%%%%%%%%%%%%%%%%
%%%%%%%%%%%%%%%%%%%%%%%%%%%%%%%%%%%%%%%%%%%%%%%%%%%%

\vspace{3mm}

{\small
\begin{flushleft}
Kavli Institute for the Physics and Mathematics of the Universe (WPI), The University of Tokyo\\ 
5-1-5 Kashiwanoha, Kashiwa, Chiba, 277-8583, Japan\\
\textit{E-mail adress}: \texttt{shusuke.otabe@ipmu.jp}
\end{flushleft}
}

\end{document}